\documentclass[oneside, 11pt]{amsart}

\usepackage[a4paper, margin=2.8cm]{geometry}
\usepackage[foot]{amsaddr}
\usepackage{comment}
\usepackage[style=alphabetic, maxnames=5]{biblatex}
\renewbibmacro{in:}{}
\AtBeginBibliography{\small}

\usepackage[normalem]{ulem}
\usepackage{amsmath}
\usepackage{hyperref}
\usepackage{mathtools}
\usepackage[makeroom]{cancel}
\usepackage{amssymb}
\usepackage{amsthm}
\usepackage{tensor}
\usepackage{enumerate}
\usepackage{cleveref}
\usepackage{bbm}
\usepackage{float}
\usepackage{varwidth}
\usepackage{pictexwd, dcpic}
\usepackage{xcolor}
\hypersetup{
    colorlinks,
    linkcolor={blue!60!black},
    citecolor={red!60!black},
    urlcolor={blue!80!black}
}
\usepackage{graphicx}
\usepackage{fancyhdr}
\usepackage{mathrsfs}
\usepackage{xfrac}
\usepackage[utf8]{inputenc}
\usepackage{tikz}
    \usetikzlibrary{decorations.pathreplacing}
    \usetikzlibrary{decorations.pathmorphing}
    \usetikzlibrary{decorations.text}
    \usetikzlibrary{arrows}
\usepackage{scalerel,stackengine}
\usepackage{booktabs}
\usepackage{slashed}
\usepackage[T1]{fontenc} 
\usepackage{soul}
\definecolor{aoenglish}{rgb}{0.0, 0.5, 0.0} 
\definecolor{marine}{rgb}{0.0, 0.05, 0.55}

\def\sideremark#1{\ifvmode\leavevmode\fi\vadjust{\vbox to0pt{\vss
 \hbox to 0pt{\hskip\hsize\hskip1em
 \vbox{\hsize2cm\tiny\raggedright\pretolerance10000
  \noindent #1\hfill}\hss}\vbox to8pt{\vfil}\vss}}}

\theoremstyle{plain}
\newtheorem{theorem}{Theorem}
\newtheorem*{theorem*}{Theorem}
\newtheorem*{main_theorem*}{Main Theorem}
\newtheorem{lemma}[theorem]{Lemma}

\newtheorem*{corollary*}{Corollary}

\theoremstyle{definition}
\newtheorem{definition}[theorem]{Definition}

\theoremstyle{remark}
\newtheorem{remark}[theorem]{Remark}

\numberwithin{theorem}{section}
\numberwithin{equation}{section}

\renewcommand{\d}{\mathrm{d}}
\renewcommand{\leq}{\leqslant}
\renewcommand{\geq}{\geqslant}
\renewcommand{\epsilon}{\varepsilon}

\renewcommand{\Im}{\operatorname{Im}}

\newcommand{\D}{\mathrm{D}}
\newcommand{\la}{\lesssim}

\newcommand{\scri}{\mathscr{I}}
\newcommand{\e}{\mathrm{e}}

\newcommand{\R}{\mathbb{R}}

\DeclareFontFamily{U}{MnSymbolC}{}
\DeclareSymbolFont{MnSyC}{U}{MnSymbolC}{m}{n}
\DeclareFontShape{U}{MnSymbolC}{m}{n}{
    <-6>  MnSymbolC5
   <6-7>  MnSymbolC6
   <7-8>  MnSymbolC7
   <8-9>  MnSymbolC8
   <9-10> MnSymbolC9
  <10-12> MnSymbolC10
  <12->   MnSymbolC12}{}
\DeclareMathSymbol{\intprod}{\mathbin}{MnSyC}{'270}

\newcommand*{\defeq}{\mathrel{\vcenter{\baselineskip0.5ex \lineskiplimit0pt
                     \hbox{\scriptsize.}\hbox{\scriptsize.}}}%
                     =}

\newcommand*{\eqdef}{=\mathrel{\vcenter{\baselineskip0.5ex \lineskiplimit0pt
                     \hbox{\scriptsize.}\hbox{\scriptsize.}}}%
                     }

\DeclareFontFamily{U}{BOONDOX-calo}{\skewchar\font=45 }
\DeclareFontShape{U}{BOONDOX-calo}{m}{n}{
  <-> s*[1.05] BOONDOX-r-calo}{}
\DeclareFontShape{U}{BOONDOX-calo}{b}{n}{
  <-> s*[1.05] BOONDOX-b-calo}{}
\DeclareMathAlphabet{\mathcalboondox}{U}{BOONDOX-calo}{m}{n}
\SetMathAlphabet{\mathcalboondox}{bold}{U}{BOONDOX-calo}{b}{n}
\DeclareMathAlphabet{\mathbcalboondox}{U}{BOONDOX-calo}{b}{n}

\def\Xint#1{\mathchoice
{\XXint\displaystyle\textstyle{#1}}%
{\XXint\textstyle\scriptstyle{#1}}%
{\XXint\scriptstyle\scriptscriptstyle{#1}}%
{\XXint\scriptscriptstyle\scriptscriptstyle{#1}}%
\!\int}
\def\XXint#1#2#3{{\setbox0=\hbox{$#1{#2#3}{\int}$ }
\vcenter{\hbox{$#2#3$ }}\kern-.6\wd0}}

\def\dashint{\Xint-}

\newcommand{\filledcircle}[1]{\tikz \filldraw[fill=white, draw=black] (0,0) circle (#1);}

\DeclareMathOperator{\dvol}{dv}

\title[Finite Energy Charged Scalar Fields on the Einstein Cylinder]{Global Finite Energy Solutions of the \mbox{Maxwell-Scalar Field} System on the Einstein Cylinder}
\author{Jean-Philippe Nicolas}
\address[J.-P. Nicolas]{University of Brest, 6 avenue Victor Le Gorgeu, 29238 Brest Cedex 3, France.}
\email{jnicolas@univ-brest.fr}

\author{Grigalius Taujanskas}
\address[G. Taujanskas]{Department of Pure Mathematics and Mathematical Statistics, University of Cambridge, CB3 0WB, UK; Trinity Hall, Trinity Lane, Cambridge, CB2 1TJ, UK.}
\email{taujanskas@dpmms.cam.ac.uk}

\date{\today}

\begin{document}

\maketitle

\begin{abstract}
    We prove the existence and uniqueness of global finite energy solutions of the Maxwell-scalar field system in Lorenz gauge on the Einstein cylinder $\mathbb{R} \times \mathbb{S}^3$. Our method is a combination of a conformal patching argument, the finite energy existence theorem in Lorenz gauge on Minkowski space of Selberg and Tesfahun, a careful localization of finite energy data, and null form estimates of Foschi--Klainerman type. Although we prove that the energy-carrying components of the solution maintain regularity, due to the incompleteness of the null structure in Lorenz gauge and the nature of our foliation-change arguments we find small losses of regularity in both the scalar field and the potential.
\end{abstract}

\setcounter{tocdepth}{1}
\tableofcontents

\section{Introduction}

The aim of this paper is to study the initial value problem for the Maxwell-scalar field system\footnote{Historically and in the broader literature, the system is also called the Maxwell--Klein--Gordon system, however we prefer to reserve this terminology for the analogous massive system, when the scalar curvature $\mathrm{R}$ in the Lagrangian is replaced with a constant $m^2$.} on the Einstein cylinder $\mathbb{R} \times \mathbb{S}^3$ with finite energy initial data. The system to be considered is the simplest classical field theory exhibiting a non-trivial gauge dependence, with action on a four-dimensional spacetime $(\mathcalboondox{M},g)$ given by
\[ S = \int_{\mathcalboondox{M}} - \frac{1}{4} F_{ab} F^{ab} + \frac{1}{2} \D_a \phi \overline{\D^a \phi} - \frac{1}{12} \mathrm{R} |\phi|^2, \]
where $F_{ab}$ is the Maxwell field, $\D_a \phi$ the gauge-covariant derivative of the scalar field $\phi$, and $\mathrm{R}$ the scalar curvature of $\mathcalboondox{M}$. Note that the scalar curvature term in the Lagrangian guarantees the conformal invariance, and hence no preferred length scale, of the associated Euler--Lagrange equations; in other words, the system is massless, and reads
\begin{align*}
    & \nabla^b F_{ab} = \operatorname{Im}(\bar{\phi} \D_a \phi), \\
    & \D^a \D_a \phi + \frac{1}{6}\mathrm{R} \phi = 0, \\
    & \nabla_{[a} F_{bc]} = 0.
\end{align*}

Questions of low-regularity global well-posedness and scattering for the Maxwell-scalar field system have been actively studied since the celebrated works of Eardley and Moncrief \cite{EardleyMoncrief1982a,EardleyMoncrief1982b} (see also \cite{GinibreVelo1981}), building towards a complete rigorous understanding of the global behaviour of dispersive nonlinear PDEs---in particular the closely related Yang--Mills and Yang--Mills--Higgs equations. In the context of the well-posedness in Sobolev spaces $H^s(\mathbb{R}^3)$ on Minkowski space, the scaling properties of the equations predict the critical\footnote{In higher dimensions, i.e. $\mathbb{R}^{1+n}$, $n \geq 4$, $s_c = \frac{n-2}{2}$. Here global well-posedness for $s>s_c$, and in some cases $s=s_c$, is essentially known \cite{KlainermanMachedon1997,Selberg2002,RodnianskiTao2004,KriegerSterbenzTataru2015,OhTataru2016,OhTataru2016b,OhTataru2018,Pecher2018a}.} exponent $s_c = \frac{1}{2}$, with finite energy regularity corresponding to $s=1$. Global well-posedness for $s=1$ was first proved\footnote{See also \cite{PonceSideris1993} for the local well-posedness in the case $s>1$ using Strichartz estimates.} by Klainerman and Machedon \cite{KlainermanMachedon1994} in the Coulomb gauge (cf. \cite{MasmoudiNakanishi2003}), by exploiting the null structure present in this gauge together with improved spacetime estimates for null forms \cite{KlainermanMachedon1993}. After constructing the solution in the Coulomb gauge, Klainerman and Machedon also exhibited global well-posedness in the temporal gauge by employing a gauge change a posteriori. In Lorenz gauge, partial null structure in the equations was first detected by Selberg and Tesfahun \cite{SelbergTesfahun2010}, who proved global well-posedness in the finite energy class and found a small loss\footnote{More precisely, $A_a(t,\cdot) \in H^{1-}(\mathbb{R}^3)$. The loss is related to the fact that, unlike in the Coulomb gauge, in Lorenz gauge not all components of the solution are controlled in $H^1$ by the finite energy norm.} in the regularity of the potential. Global existence for finite energy data has also been treated directly in the temporal gauge by Yuan \cite{Yuan2016} and Pecher \cite{Pecher2015}, relying on more general estimates (which hold also for small data below the energy norm for the Yang--Mills equations) due to Tao \cite{Tao2003}.

Indeed, various forms of well-posedness are by now known below the energy norm: in the Coulomb gauge, local well-posedness for $s > \frac{3}{4}$ due to Cuccagna \cite{Cuccagna1999}, local well-posedness for small data for $s> \frac{1}{2}$ (almost all the way down to criticality $s_c$) due to Machedon and Sterbenz \cite{MachedonSterbenz2004}, global well-posedness for $s > \sqrt{3}/2 \approx 0.866$ due to Keel--Roy--Tao \cite{KeelRoyTao2011}, and unconditional local uniqueness for $s > s_0$, $s_0 \approx 0.914$ due to Pecher \cite{Pecher2017} (also for $s_0 \approx 0.907$ in Lorenz gauge). In Lorenz gauge or the temporal gauge, Pecher has obtained the local well-posedness for asymmetric regularity of the data $(s_\phi, s_A) > (\frac{3}{4}, \frac{1}{2})$, where $s_\phi$ denotes the exponent for the initial data for the scalar field and $s_A$ the exponent for the potential \cite{Pecher2014,Pecher2017a,Pecher2018} (see also \cite{Pecher2020} for regularity in Fourier--Lebesgue spaces $\widehat{H}^{s,r}$ in Lorenz gauge). These results rely invariably on the use of Bourgain--Klainerman--Machedon spaces $X^{s,b}_\pm$, and, for instance, constructions of almost-conservation laws paired with Strichartz, null form, and commutator estimates.

For much more smooth, small, fast-decaying data, global well-posedness for the Maxwell-scalar field system has also been studied using the conformal compactification of Minkowski space and standard a priori energy estimates on the Einstein cylinder $\mathbb{R} \times \mathbb{S}^3$ \cite{ChoquetBruhatChristodoulou1981,ChoquetBruhat1982,ChoquetBruhatPaneitzSegal1983,CagnacChoquetBruhat1984,GeorgievSchirmer1992} (Eardley and Moncrief's methods have also been extended to more general curved spacetimes by Chru\'sciel and Shatah \cite{ChruscielShatah1997} by using Friedlander's representation formula \cite{Friedlander1975} for solutions to the wave equation; see also \cite{Ghanem2016} and \cite{Taujanskas2019} showing in particular that the assumption of the smallness of data may be removed). Although this approach restricts the class of data, these works have the advantage of automatically showing that the solution extends \emph{through} the null infinity $\scri$ of Minkowski space and hence induces scattering data there, something that does not directly follow from global existence results in Minkowski space. 

From a slightly different perspective, with a focus on detailed asymptotics, decay and peeling estimates for the Maxwell-scalar field system on Minkowski space have been obtained by Lindblad and Sterbenz \cite{LindbladSterbenz2006} using Morawetz and Strichartz estimates instead of the conformal compactification (and hence allowing non-zero charge, but still requiring high regularity and smallness); a simplified proof was later given by Bieri--Miao--Shahshahani \cite{BieriMiaoShahshahani2017}. Extensions, first to large Maxwell field, and subsequently to more general large data, were given by Yang \cite{Yang2016,Yang2018}, Yang--Yu \cite{YangYu2019}, and Yang--Wei--Yu \cite{YangWeiYu2024}. As first observed by Candy--Kauffman--Lindblad \cite{CandyKauffmanLindblad2019}, the peeling behaviour of solutions is a consequence of the weak null condition in the equations in Lorenz gauge; more finely, at leading order the scalar field exhibits charge-dependent oscillatory asymptotics, and the potential carries a long-range part essentially proportional to the charge. In the context of the backward scattering problem (the Goursat problem) from $\scri$, the weak null condition and information about asymptotics has recently been used to prove that a global solution can be obtained from certain scattering data at infinity by Lindblad--Schlue \cite{LindbladSchlue2023} and He \cite{He2021} (see also \cite{ChenLindblad2023}). He, in particular, obtains a further refinement of the asymptotics of Candy--Kauffman--Lindblad, and uses weighted conformal Morawetz estimates and a spacetime Hardy inequality to obtain a solution (with loss of derivatives) from small, regular scattering data whose asymptotics are chosen to be consistent with the asymptotics of the forward evolution. This approach has been extended to permit a class of large data, as well as data with slow decay towards $i^+$. \cite{DaiMeiWeiYang2023,DaiYang2024}.

The \textbf{main result} of this paper is to prove that the Maxwell-scalar field system admits global unique solutions for merely finite energy initial data in Lorenz gauge on the Einstein cylinder $\mathbb{R} \times \mathbb{S}^3$: see \Cref{thm:main_theorem} for a precise statement. As far as we know, this constitutes the first curved background well-posedness theorem for the Maxwell-scalar field system at finite energy regularity. Our result extends trivially to all background spacetimes which are globally conformal to $\mathbb{R} \times \mathbb{S}^3$; in particular, we extend the scattering theory of \cite{Taujanskas2018} on de Sitter space $\mathrm{dS}_4$ by one derivative, and remove the assumption of smallness of data in that paper. Similarly, our result enlarges the class of data considered in the classical papers \cite{ChoquetBruhat1982,ChoquetBruhatChristodoulou1981,ChoquetBruhatPaneitzSegal1983,CagnacChoquetBruhat1984} by one derivative. Via a conformal compactification, our result furthermore extends the theorems of Klainerman--Machedon \cite{KlainermanMachedon1994} and Selberg--Tesfahun \cite{SelbergTesfahun2010} in the sense that it implies that finite energy solutions with vanishing charge on Minkowski space extend through null infinity $\scri$, and therefore can be said to \emph{scatter}. Finally, our result removes precisely the obstruction (for the Maxwell-scalar field system) encountered by Baez \cite{Baez1989} in his pioneering attempt\footnote{Due to a lack of a well-posedness theorem for finite energy initial data on $\mathbb{R} \times \mathbb{S}^3$, Baez worked instead with solutions with higher regularity, for which he could not prove the invertibility of the usual scattering operator. By using special features of the null infinity of Minkowski space, however, he was able to represent "scattering" as a distinguished (central) element of the conformal group $\mathrm{SO}(2,4)$.} to construct a scattering operator for the Yang--Mills equations on Minkowski space by using their conformal invariance properties.

\section{Overview of the Method}

Our method relies on the finite energy well-posedness theorem on Minkowski space of Selberg and Tesfahun \cite{SelbergTesfahun2010}, the conformal invariance of the Maxwell-scalar field system, a careful localization procedure, and a patching argument similar to that used in \cite{Taujanskas2019}. The main idea is to use conformal embeddings to cover a strip $I \times \mathbb{S}^3 \subset \mathbb{R} \times \mathbb{S}^3$ for an interval $I$ with two copies of Minkowski space, $\mathbb{M}$ and $\mathbb{M}'$, situated antipodally on $\mathbb{S}^3$ (see \Cref{fig:wrapping_of_Minkowski_spaces}). We then show that finite energy initial data on the Einstein cylinder induces, when suitably modified, finite energy initial data on each of $\mathbb{M}$ and $\mathbb{M}'$. On each of $\mathbb{M}$ and $\mathbb{M}'$ we obtain the existence of a finite energy solution using the theorem of Selberg and Tesfahun \cite{SelbergTesfahun2010}; this step relies crucially on the locality of the construction of \cite{SelbergTesfahun2010}. The locality is a consequence of the Lorenz gauge, in which the Maxwell-scalar field system takes the form of a system of coupled nonlinear wave equations. This is not the case in general; for instance in the Coulomb gauge (which exhibits better null structure properties) the system contains a non-local elliptic equation. The resulting solutions on $\mathbb{M}$ and $\mathbb{M}'$ respect the foliations of $\mathbb{M}$ and $\mathbb{M}'$ induced by the standard timelike Killing fields $\partial_t$ and $\partial_{t'}$. These foliations are (a) not the same (considered as partial foliations of $\mathbb{R} \times \mathbb{S}^3$), and (b) are unsuitable for continuing the solution to the next strip. In order to correctly patch the solutions on the overlap region $\mathbb{M} \cap \mathbb{M}' \subset \mathbb{R} \times \mathbb{S}^3$ and obtain the appropriate regularity, we must change their foliations to the standard spacelike foliation of $\mathbb{R} \times \mathbb{S}^3$. We achieve this by first proving that the right-hand sides of the wave equations for the components of the solution, considered as linear wave equations, possess sufficiently good $L^2$-based spacetime Sobolev regularity. This step relies on the null structure in the equations; as already mentioned, however, in Lorenz gauge this null structure is only partial, and we must therefore contend with small losses of regularity that permeate our argument as a result. These losses are, precisely, an arbitrarily small loss of differentiability in the scalar field $\phi$, and slightly more than one half of a derivative in the potential $A_a$. The sharpness of these losses is not completely clear (see \Cref{rmk:regularity_losses}). Nevertheless, we are able to show that the components which directly carry the energy---the Maxwell field $F_{ab}$ and the gauge-covariant derivative $\D_a \phi$ of $\phi$---do not exhibit any loss; this is analogous to a similar situation in (indeed, is partly inherited from) \cite{SelbergTesfahun2010}, where the authors find an arbitrarily small loss in the potential already on Minkowski space, but nevertheless obtain a global finite energy solution. Finally, we extend the local-in-time solution on the cylinder $\mathbb{R} \times \mathbb{S}^3$ globally by iterating the procedure. 

We outline the method more precisely in the following 5 steps.

\begin{figure}[h]
    \centering
    \begin{tikzpicture}
    \node[inner sep=0pt] (patching) at (2,0)
        {\includegraphics[scale=0.5]{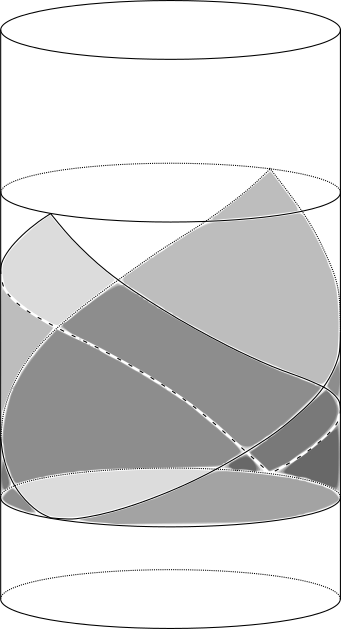}};
    \draw[-latex] (-1.8,-2.85) .. controls (-1.8,-3.8) and (0.43,-3.8) .. (0.45,-2.75);
    \node[label={[shift={(0.45,-3.03)}]\filledcircle{0.05cm}}] {};
    \node[label={[shift={(-2,-3)}]$(i^0)'= O$}] {};
    \draw[-latex] (5.95,-2.85) .. controls (5.8,-3.6) and (3.5,-3.6) .. (3.33,-2.15);
    \node[label={[shift={(3.33,-2.44)}]\filledcircle{0.05cm}}] {};
    \node[label={[shift={(6,-3)}]$i^0= O'$}] {};
    \node[label={[shift={(-2,-3)}]$(i^0)'= O$}] {};
    \draw[-latex] (5.6,2.45) .. controls (5,4) and (2.1,3) .. (2,1);
    \draw[-latex] (5.6,2.45) .. controls (5,2) and (4.5,1) .. (4,1);
    \node[label={[shift={(6.2,2)}]$(\scri^+)'$}] {};
    \node[label={[shift={(0.35,2.7)}, rotate=90]$\mathbb{R} \times \mathbb{S}^3$}] {};
    \draw[-latex] (-1.9,2.1) .. controls (-1.4,2) and (-0.5,1.5) .. (-0.1,1.1);
    \draw[-latex] (-1.9,2.1) .. controls (1,3) and (2,0.3) .. (2.05,0.1);
    \node[label={[shift={(-2.2,1.8)}]$\scri^+$}] {};
    \node[label={[shift={(0.45,0.2)}]$\mathbb{M}_+$}] {};
    \node[label={[shift={(3.25,0.3)}]$\mathbb{M}'_+$}] {};
    \node[label={[shift={(2,-3.45)}]$\Sigma$}] {};
    \node[label={[shift={(0.55,1.2)}]$i^+$}] {};
    \node[label={[shift={(0.425,1.02)}]\filledcircle{0.05cm}}] {};
    \node[label={[shift={(3.33,1.75)}]$(i^+)'$}] {};
    \node[label={[shift={(3.325,1.62)}]\filledcircle{0.05cm}}] {};
    \node[label={[shift={(2,-4.8)}]}] {};
    \end{tikzpicture}
    \caption{A depiction of two copies of Minkowski space (for positive times), $\mathbb{M}_+ = \{ t \geq 0\} \cap \mathbb{M}$ and $\mathbb{M}'_+ = \{t' \geq 0\} \cap\mathbb{M}'$, conformally embedded in $\mathbb{R} \times \mathbb{S}^3$ in a way that situates the origin $O$ of $\mathbb{M}$ at the spatial infinity $(i^0)'$ of $\mathbb{M}'$, and vice-versa. The Einstein cylinder $\mathbb{R} \times \mathbb{S}^3$ is represented by the surface of the cylinder and the two future-pointing halves $\mathbb{M}_+$, $\mathbb{M}_+'$ of Minkowski space are represented by the half-diamonds wrapped around the cylinder. Two spherical dimensions are suppressed. Via the Hopf fibration $\mathbb{S}^1 \hookrightarrow \mathbb{S}^3 \to \mathbb{S}^2$, each point on the surface of the depicted cylinder may be thought of as a 2-sphere. The future timelike infinities $i^+$ and $(i^+)'$, and spatial infinities $i^0$ and $(i^0)'$, are points rather than spheres in the usual Penrose conformal compactification of $\mathbb{M}$ (resp. $\mathbb{M}'$), and are therefore not included in $\mathbb{M}_+$, $\mathbb{M}_+'$. The future null infinities $\scri^+$ and $(\scri^+)'$ are (open) backward lightcones of $i^+$ and $(i^+)'$, with topology $\mathbb{R} \times \mathbb{S}^2$, forming the future null boundaries of $\mathbb{M}_+$ and $\mathbb{M}'_+$ in $\mathbb{R} \times \mathbb{S}^3$. The initial surface $\Sigma \simeq \mathbb{S}^3 \subset \mathbb{R} \times \mathbb{S}^3$ on the cylinder is the one-point compactification of each of the Minkowskian initial surfaces $\tilde{\Sigma} = \{t = 0\}$ and $\tilde{\Sigma}' = \{ t' = 0 \}$, $\Sigma \simeq \tilde{\Sigma} \cup i^0 \simeq \tilde{\Sigma}' \cup (i^0)'$.}
    \label{fig:wrapping_of_Minkowski_spaces}
\end{figure}
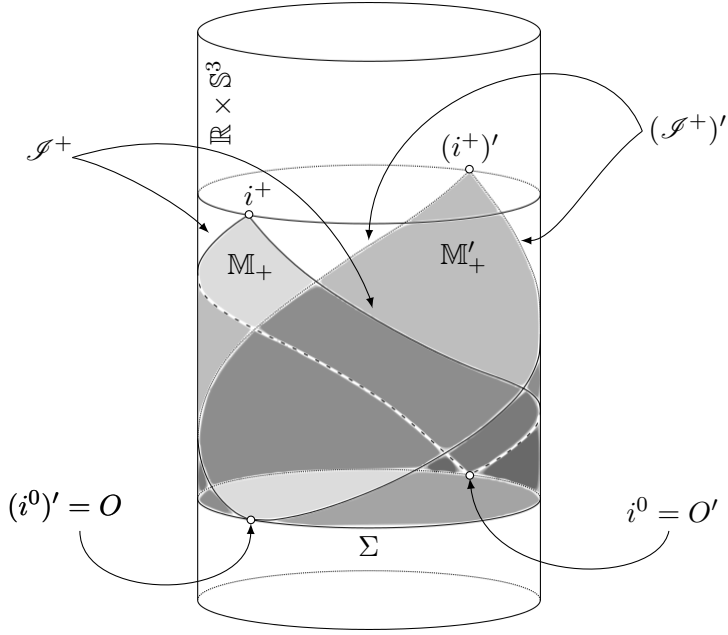

\subsection{Step 1. Localization of data.}

We begin with a set of finite energy initial data on $\Sigma \simeq \mathbb{S}^3 \subset \mathbb{R} \times \mathbb{S}^3$. By the conformal invariance of the Maxwell-scalar field system, the embeddings $\mathbb{M}, \, \mathbb{M}' \hookrightarrow \mathbb{R} \times \mathbb{S}^3$ induce conformally scaled data on the initial surfaces $\tilde{\Sigma}$ and $\tilde{\Sigma}'$ in $\mathbb{M}$ and $\mathbb{M}'$. Let us focus on the data on, say, $\tilde{\Sigma}$, the same statements being true on $\tilde{\Sigma}'$. This conformal scaling of the data imposes fast decay rates of all components of the Maxwell data towards spatial infinity, in particular precluding the presence of non-zero electric charge $\tilde{\mathfrak{q}}$ (see \Cref{rmk:electric_charge_zero}). For the scalar field data, however, the conformal scaling imposes fast decay of the time derivative component, but introduces\footnote{This is a direct consequence of the conformal invariance of the system, since the electric constraint equation $\boldsymbol{\nabla} \cdot \mathbf{E} = \operatorname{Im}(\bar{\phi} \D_0 \phi)$ must remain unchanged.} a long-range tail in the Dirichlet component. Hence, although all but one of the components of the induced data have fast decay towards spatial infinity, the long-range tail in the Dirichlet component of the scalar field results in the conformally induced data on $\tilde{\Sigma}$ having \emph{infinite} energy in the Minkowskian sense. To deal with this, we carefully modify the data for the scalar field; crucially, our modification preserves the constraint equations everywhere on $\tilde{\Sigma}$. This modification in effect undoes the conformal transformation---of just the scalar field data---near spatial infinity, and leaves it unchanged elsewhere: see equation \eqref{modification_of_initial_scalar_field_data}. The domain of influence of the support of the modified part of the data is eventually discarded, see \Cref{fig:flat_Minkowski_intersections}. The above statements are made precise and proven in \Cref{sec:induced_data_and_modification} and \Cref{lem:modification_of_induced_Minkowskian_data}.

\subsection{Step 2. Refoliation of local Minkowskian solutions.}

Having obtained finite energy data on each of $\tilde{\Sigma}$ and $\tilde{\Sigma}'$, in \Cref{sec:local_solution} we now deduce the existence of a solution in each of $\mathbb{M}$ and $\mathbb{M}'$ using \cite{SelbergTesfahun2010}. These solutions satisfy their own respective Minkowskian Lorenz gauges $\tilde{\nabla}_a \tilde{A}^a = 0 = \tilde{\nabla}' (\tilde{A}')^a$, and have finite energy regularity (i.e. regularity of the type $\mathcal{C}^0 H^s \cap \mathcal{C}^1 H^{s-1} \cap H^{s,b}$) with respect to their own standard Minkowskian foliations by hypersurfaces $\tilde{\Sigma}_t = \{t = \text{const.} \}$ and $\tilde{\Sigma}_{t'}= \{ t' = \text{const}. \}$. We again remark that at this stage there is already a small loss of regularity in the potential in Lorenz gauge, for which $s=1-\delta$ for any $\delta > 0$, as compared to the case of the Coulomb gauge \cite{KlainermanMachedon1994} ($s=1$). For the scalar field, $s=1$. In \Cref{lem:regularity_with_respect_to_hyperboloidal_foliation} we prove that these solutions have $\mathcal{C}^0 H^s$-type regularity \emph{with respect to a locally hyperboloidal foliation}, the foliation that, locally, under the conformal embedding into the Einstein cylinder will become orthogonal to the standard timelike Killing field $\partial_\tau$ (in fact, our proof works for any uniformly spacelike foliation). This is done using two general lemmas (\Cref{lem:foliation_change_positive_s} and \Cref{lem:foliation_change_negative_s}) for inhomogeneous linear waves on manifolds, which we prove in the Appendix, combined with very explicit decompositions of the nonlinearities into null forms (\Cref{sec:spacetime_estimates_for_nonlinearities}), for which we use spacetime null form estimates of Foschi--Klainerman\footnote{See also the review article \cite{KlainermanSelberg2002} on null form estimates and applications.} and Klainerman--Machedon type \cite{FoschiKlainerman2000,KlainermanMachedon1996}. \Cref{lem:foliation_change_positive_s,lem:foliation_change_negative_s} rely on $L^2$-based spacetime Sobolev regularity of the inhomogeneities. In Lorenz gauge there is no null structure in the equation for the potential $\tilde{A}_a$, and therefore we can only prove weak spacetime regularity for the nonlinearity $\mathcalboondox{N}(\tilde{A}_a, \tilde{\phi})$ in this equation, which eventually results in a loss of a half of a derivative in the regularity of the potential with respect to the new foliation ($s=\frac{1}{2}-\delta$). For the scalar field, although the nonlinearity $\mathcalboondox{M}(\tilde{A}_a, \tilde{\phi})$ verifies the null condition, the small loss $\delta$ feeds into the spacetime regularity of $\mathcalboondox{M}(\tilde{A}_a, \tilde{\phi})$, and eventually results in a small loss ($s=1-\delta$) in the regularity of the scalar field $\tilde{\phi}$ with respect to the new foliation. Despite the lack of null structure in the equation for the potential, the electric and magnetic fields satisfy wave equations with null structure independent of gauge.

\begin{remark}
    \label{rmk:regularity_losses}
    The regularity $\mathcal{C}^0H^{\frac{1}{2}-\delta}$ of the potential is probably not sharp; we expect that it can be improved to $\mathcal{C}^0H^{1-\delta}$, as in the case of the Minkowskian foliation, though we make no attempt to do so here. It is unclear if the regularity $\mathcal{C}^0 H^{1-\delta}$ (for the potential and the scalar field) can be improved further to $\mathcal{C}^0 H^1$. Indeed, the loss $\delta$ is a direct consequence of the loss seen by Selberg and Tesfahun \cite{SelbergTesfahun2010} in Lorenz gauge. Since the Lorenz gauge makes the system local, the equations are, in a loose sense, insensitive to the global geometry of the spacetime; this is in contrast to the Coulomb gauge\footnote{On the other hand, we also mention here that Oh's caloric-temporal gauge \cite{Oh2014,SungJinOh2015} for the Yang--Mills equations appears to at once address the issues of hyperbolicity, null structure, and energy estimates, and therefore combine the advantages of the Coulomb and Lorenz gauges.} \cite{KlainermanMachedon1994} due to the presence of an elliptic equation. 
    But on spatially compact manifolds null form estimates appear to exhibit a similar arbitrarily small loss of regularity \cite{Taujanskas2025}, at least globally \cite{Sogge1993,GeorgievSchirmer1995}. This is also seen in the case of Strichartz estimates for the Schr\"odinger equation \cite{Bourgain1993a,Bourgain1993b,BurqGerardTzvetkov2001}.
\end{remark}

\subsection{Step 3. Improved regularity and change of gauge.}

Next, we conformally transport the local solution to a domain of dependence on the Einstein cylinder, change the gauge to the cylindrical Lorenz gauge $\nabla_a A^a = 0$, and obtain the regularity $\mathcal{C}^0 H^{s-1}$ of the time derivatives of the potential and scalar field (\Cref{lem:local_solution_cylinder}). For the potential, this is done by using the regularity of the electric and magnetic fields and the gauge condition to express time derivatives in terms of spatial derivatives; for the scalar field, we return to the corresponding wave equation and use the conservation of energy for more regular solutions and the Aubin--Lions lemma. We prove that the change of gauge from Minkowskian to cylindrical Lorenz preserves regularity by using Sobolev composition estimates (\Cref{lem:composition_Sobolev_functions_1,lem:composition_Sobolev_functions_2,lem:composition_Sobolev_functions_3}) and product estimates in Sobolev spaces (\Cref{lem:product_Sobolev_estimates}).

\subsection{Step 4. Conformal patching on a strip.}

In \Cref{sec:two_copies_of_Minkowski} we cover a thin strip of the Einstein cylinder with two copies of Minkowski space (cf. \Cref{fig:wrapping_of_Minkowski_spaces}) and give a precise meaning to the priming operation used above as the interchange of the two antipodally placed Minkowski spaces. In \Cref{thm:local_strip_solution_cylinder} we prove the existence of a local-in-time finite energy solution on the strip by showing that on the overlap region the solutions obtained in \textbf{Step 3} and \textbf{(Step 3)'} are the same up to a sufficiently well-behaved gauge transformation, and prove the cylindrical conservation of energy by approximating the solution by smooth solutions and utilising domain of dependence properties and convergence in the Minkowskian energy norm. Here we also show that the solution obtained in \Cref{thm:local_strip_solution_cylinder} is a limit of more regular solutions in the \emph{cylindrical} energy norm.

\subsection{Step 5. Iteration.}

Finally, in \Cref{sec:global_solution} we extend the solution globally by inducting on the number of strips of the local solution. Here we must look after a residual gauge transformation which allows us to restart the argument on the next strip, and which, due to the loss of regularity in the potential, is rough. We observe that the regularity of the energy-carrying components $\mathbf{E}$, $\mathbf{B}$ and $\D_a \phi$ is nevertheless preserved due to the gauge-invariance of their $L^2$ norms.

\section{Notation}

We work in four dimensions and our spacetime metric signature is $(+,-,-,-)$. We denote the standard coordinates on the Einstein cylinder $\mathbb{R} \times \mathbb{S}^3$ by $(\tau, \zeta, \theta, \varphi)$, with metric $\d \tau^2 - \d \zeta^2 - \sin^2\zeta(\d \theta^2 + \sin^2 \theta \d \varphi^2)$. We denote the standard coordinates on Minkowski space $\mathbb{R}^{1+3}$ by $(t,r,\theta, \varphi)$ or $(t,x)$, $x \in \mathbb{R}^3$, with metric $\d t^2 - \d r^2 - r^2(\d \theta^2 + \sin^2 \theta \d \varphi^2) = \d t^2 - |\d x|^2$. Throughout the paper we distinguish between quantities on $\mathbb{R} \times \mathbb{S}^3$ and $\mathbb{R}^{1+3}$ by denoting quantities on Minkowski space with a tilde, and quantities on the Einstein cylinder plainly. For instance, $\tilde{\phi}$ denotes the scalar field on $\mathbb{R}^{1+3}$, and $\phi$ the scalar field on $\mathbb{R} \times \mathbb{S}^3$. We denote projections to spatial slices with respect to the $\tau$-coordinate on $\mathbb{R} \times \mathbb{S}^3$ and the $t$-coordinate on $\mathbb{R}^{1+3}$ in bold. For example, $\tilde{\boldsymbol{\nabla}}$ denotes the projection of the Levi-Civita connection $\tilde{\nabla}$ on $\mathbb{R}^{1+3}$ to the slices $\tilde{\Sigma}_t = \{ t = \text{const.} \}$. Due to the form of the metric on $\mathbb{R}^{1+3}$, $\tilde{\boldsymbol{\nabla}}$ coincides with the Levi-Civita connection on $\tilde{\Sigma}_t$. Similarly, untilded bold quantities denote projections to $\Sigma_\tau = \{ \tau = \text{const.} \} \subset \mathbb{R} \times \mathbb{S}^3$.

In Fourier space, we denote by $\xi \in \mathbb{R}^n$ the spatial Fourier frequencies dual to $x \in \mathbb{R}^n$. On Minkowski space, we denote by $\xi_0$ the temporal frequency dual to $t$. We use the shorthand $\langle \xi \rangle = (1+|\xi|^2)^{1/2}$.

Spaces of Schwartz functions and tempered distributions on $\mathbb{R}^n$ are denoted by $\mathscr{S}(\mathbb{R}^n)$ and $\mathscr{S}'(\mathbb{R}^n)$ respectively, where $'$ denotes the topological dual also more generally; $\mathscr{D}'(\mathbb{R}^n)$ denotes the space of distributions on $\mathbb{R}^n$ and $\mathscr{D}(
\mathbb{R}^n)$ the space of smooth compactly supported test functions. The standard Lebesgue spaces are denoted by $L^p$, and $L^2$-based Sobolev spaces by $H^s$ (see \Cref{sec:Sobolev_spaces}). Local Sobolev spaces are denoted by $H^s_{loc}(\Omega)$; that is, $f\in H^s_{loc}(\Omega)$ if and only if $f \in H^s(\Omega')$ for all $\Omega' \subset \subset \Omega$. Spaces of $k$ times continuously differentiable functions are denoted by $\mathcal{C}^k$. We write $f \in X \cdot Y$ to mean that $f = g h$ for some $g \in X$ and $h \in Y$. For Banach spaces $X$ and $Y$, $X \subset \subset Y$ denotes compact inclusion and $X \hookrightarrow Y$ continuous inclusion. The notation $f \la g$ means that there exists a constant $C>0$ such that $f \leq C g$.

\section{Preliminaries}

\subsection{Sobolev spaces of fractional order}
\label{sec:Sobolev_spaces}

We work with $L^2$-based Sobolev spaces of fractional order. As our analysis in the main text requires a careful treatment of regularity, below we include our definitions of the function spaces used and record a number of their properties.

\subsubsection{Sobolev spaces on the whole space}

Let $n \in \mathbb{N}$, $n \geq 3$, and let  $\mathscr{F}$ denote the Fourier transform on $\mathbb{R}^n$. For any given $s \in \mathbb{R}$ we define the fractional Sobolev\footnote{These are also found in the wider literature by the name of Bessel-potential spaces, denoted $H^s_2 = F^s_{2,2}$ \cite{RunstSickel1996,TriebelI,TriebelII}.} space 
\begin{equation} \label{definition_fractional_Sobolev_spaces} H^s(\mathbb{R}^n) \defeq \left\{ f \in \mathscr{S}'(\mathbb{R}^n) ~ : ~ \|f\|_{H^s} \defeq \| \mathscr{F}^{-1}(\langle \xi \rangle^s\mathscr{F}(f))\|_{L^2}  < \infty \right\}.
\end{equation}
The definition \eqref{definition_fractional_Sobolev_spaces} is equivalent to defining $H^s(\mathbb{R}^n)$ as the completion of $\mathscr{S}(\mathbb{R}^n) \ni f$  with respect to the norm $\| \langle \xi \rangle^s \mathscr{F}(f) \|_{L^2}$, or equivalently $\| (1-\Delta)^{\frac{s}{2}} f \|_{L^2}$, since for $s \in 2 \mathbb{N}$, $(1-\Delta)^{\frac{s}{2}} f = \mathscr{F}^{-1} (\langle \xi \rangle^{s} \mathscr{F}(f))$ in the classical sense, and for $s \notin 2\mathbb{N}$ it is precisely the definition of $(1-\Delta)^\frac{s}{2}$, where $\Delta$ is the Laplacian on $\mathbb{R}^n$ \cite[\S2]{RunstSickel1996}. It is immediate from the definition \eqref{definition_fractional_Sobolev_spaces} that the dual spaces of $H^s$ are given by
\[ (H^s(\mathbb{R}^n))' = H^{-s}(\mathbb{R}^n) \quad \forall s \in \mathbb{R}.\]
Moreover, we have the Sobolev embeddings
\begin{align}
    \label{Sobolev_embedding_continuous}
    & H^s(\mathbb{R}^n) \hookrightarrow \mathcal{C}^0(\mathbb{R}^n) \quad \text{for} ~~ s > \frac{n}{2}, \\
    \label{Sobolev_embedding_Lp}
    & H^s(\mathbb{R}^n) \hookrightarrow L^p(\mathbb{R}^n) \quad \text{for} ~~ s \geq n \left( \frac{1}{2} - \frac{1}{p} \right) ~~~ \text{and} ~~~ p \geq 2.
\end{align}

We also define, for $|s| < \frac{n}{2}$, the homogeneous fractional Sobolev space $\dot{H}^s(\mathbb{R}^n)$ as the completion of $\mathscr{S}(\mathbb{R}^n) \ni f$ with respect to the norm
\begin{equation}
    \label{homogeneous_Sobolev_norm_definition}
    \| |\xi|^s \mathscr{F}(f) \|_{L^2} = \| (-\Delta)^\frac{s}{2} f\|_{L^2},
\end{equation}
with $\|f\|_{\dot{H}^s} = \| \mathscr{F}^{-1}(|\xi|^s \mathscr{F}(f))|_{L^2}$. Note in particular that $1 \notin \dot{H}^1(\mathbb{R}^n)$ according to this definition, and, for instance, the Sobolev inequality
\begin{equation}
    \label{Sobolev_embedding_homogeneous_H1}
    \| f \|_{L^{\frac{2n}{n-2}}} \la \| f \|_{\dot{H}^1}
\end{equation} 
holds for all $f \in \dot{H}^1(\mathbb{R}^n)$. For $s \geq \frac{n}{2}$, this definition needs to modified\footnote{It is possible to define homogeneous Sobolev spaces for any $s \in \mathbb{R}$ by $\left\{ f \in \mathscr{S}'(\mathbb{R}^n) / \mathscr{P}(\mathbb{R}^n) \, : \, \| f \|_{\dot{H}^s} \defeq \| \mathscr{F}^{-1}(| \xi |^s \mathscr{F}(f)) \|_{L^2} < \infty \right\}$, where $\mathscr{P}(\mathbb{R}^n)$ is the space of polynomials on $\mathbb{R}^n$. If $|s| \geq \frac{n}{2}$, elements of this space consist of equivalence classes modulo non-trivial polynomials, and therefore the embedding \eqref{Sobolev_embedding_homogeneous_H1} and its analogues do not hold with respect to this definition. Furthermore, products of elements of this space are not well-defined \cite{Bourdaud2013,BahouriCheminDanchin2011,Grafakos2008,Grafakos2014,Soga1983}.} since \eqref{homogeneous_Sobolev_norm_definition} no longer separates polynomials of degree up to $s-\frac{n}{2}$. Using Plancherel's theorem, one has
\[ (\dot{H}^s(\mathbb{R}^n))' = \dot{H}^{-s}(\mathbb{R}^n) \quad \text{for} ~ |s| < \frac{n}{2}. \]

\subsubsection{Sobolev spaces on bounded Lipschitz domains}

Next, let $\Omega \subset \mathbb{R}^n$ be a bounded Lipschitz domain. We then define, for $s \in \mathbb{R}$,
\begin{equation}
    \label{definition_fractional_Sobolev_spaces_on_domain}
    H^s(\Omega) \defeq \left\{ f \in \mathscr{D}'(\Omega) \, : \, \exists g\in H^s(\mathbb{R}^n) ~ \text{s.t.} ~g|_{\Omega} = f \right\}
\end{equation}
with the norm
\[ \| f \|_{H^s(\Omega)} \defeq \inf_g \| g \|_{H^s}. \]
There exists a bounded linear extension operator $\operatorname{ext}: H^s(\Omega) \to H^s(\mathbb{R}^n)$ such that $(\operatorname{ext}f)(x) = f(x)$ for almost every $x \in \Omega$ (see \S5 \cite{DiNezzaPalatucciValdinoci2012} or \cite{Triebel2002}; also \cite{Stein1970,AdamsFournier2003,RunstSickel1996}), and the embeddings \eqref{Sobolev_embedding_continuous}, \eqref{Sobolev_embedding_Lp} remain true on $\Omega$ as a direct consequence of the definition \eqref{definition_fractional_Sobolev_spaces_on_domain}. The dual spaces of $H^s(\Omega)$ are:
\begin{align*}
    & (H^s(\Omega))' = H^{-s}(\Omega) \quad \text{for} ~~ -\frac{1}{2} < s< \frac{1}{2}, \\
    & (H^s_0(\Omega))' = H^{-s}(\Omega) \quad \text{for} ~~ s > \frac{1}{2} ~~~ \text{and} ~~~ s - \frac{1}{2} \notin \mathbb{N},
\end{align*}
where $H^s_0(\Omega)$ is the closure of $\mathscr{D}(\Omega)$ in $H^s(\Omega)$. Provided $s > \frac{1}{2}$ and $s -\frac{1}{2} \notin \mathbb{N}$, there exists a bounded linear extension-by-zero operator $\operatorname{ext}_0 : H^s_0(\Omega) \to H^s(\mathbb{R}^n)$ such that $(\operatorname{ext}_0f)(x) = f(x)$ for almost every $x \in \Omega$, and $(\operatorname{ext}_0 f )(x) = 0$ for $x \notin \Omega$. Moreover, provided $\frac{1}{2} < s < \frac{3}{2}$, there exists a bounded linear trace operator
\begin{equation}
    \label{trace_operator}
    \mathrm{tr}:H^s(\Omega) \to H^{s-\frac{1}{2}}(\partial \Omega)
\end{equation}
such that for $f \in H^s(\Omega) \cap \mathcal{C}^0(\bar{\Omega})$, $\operatorname{tr}f = f|_{\partial\Omega}$ \cite{Besovetal1977,Costabel1988,Zhonghai1996}. When the boundary $\partial \Omega$ of $\Omega$ is smooth, the upper bound $ s < \frac{3}{2}$ in \eqref{trace_operator} can be removed \cite{Tartar2007}.

\subsubsection{Sobolev spaces on compact manifolds} Now let $\mathcalboondox{N}$ be a smooth compact Riemannian manifold and $E \to \mathcalboondox{N}$ a Hermitian vector bundle\footnote{It is worth noting that if $\mathcalboondox{N}$ is contractible, as is the case for $\mathcalboondox{N}=\mathbb{S}^3 \simeq  \mathrm{SU}(2)$, then all vector bundles over $\mathcalboondox{N}$ are trivial.} over $\mathcalboondox{N}$. Choose a finite open cover of charts $\{ \Omega_i \}_i$ of $\mathcalboondox{N}$ such that $\Omega_i$ embeds smoothly into $\mathbb{R}^n$, $\phi_i : \Omega_i \to \mathbb{R}^n$, and $E$ is trivial on $\Omega_i$, and choose a partition of unity $\{ \rho_i \}_i $ subordinate to the cover $\{\Omega_i \}_i$. We then define, for $s \geq 0$, $H^s(\mathcalboondox{N};E)$ in the natural way by localization:
\begin{equation}
    \label{Sobolev_spaces_on_manifolds}
    H^s(\mathcalboondox{N};E) \defeq \left\{ f \in L^1_{\text{loc}}(\mathcalboondox{N};E) \, : \, \|f \|_{H^s(\mathcalboondox{N})} \defeq \left( \sum_i \| (\rho_i f) \circ \phi_i^{-1} \|^2_{H^s(\mathbb{R}^n)} \right)^{1/2} < \infty \right\}.
\end{equation}
The definition \eqref{Sobolev_spaces_on_manifolds} is independent of the choice of open cover $\{ \Omega_i \}_i$, the choice of partition of unity $\{ \rho_i \}_i$, and the Riemannian metric on $\mathcalboondox{N}$ \cite{Hebey1996,Aubin1998,Triebel2002}, and produces a Hilbert space for each $s$. For negative exponents, we define
\[ H^{-s}(\mathcalboondox{N};E) \defeq (H^s(\mathcalboondox{N};E))', \]
and suppress the vector bundle $E$ from now on.

\subsection{Composition estimates}

We record a pair of lemmas regarding the composition of a smooth real-valued function $G$ (with bounded derivative) with a Sobolev function $f \in H^s (\Omega)$ on a smooth bounded domain $\Omega \subset \mathbb{R}^n$. In our analysis this is needed due to the nonlinear change of the scalar field $\phi$, $\phi \rightsquigarrow \e^{-i\chi} \phi$, under a gauge transformation\footnote{The fact that $x \mapsto \e^{-ix}$ is complex-valued is not an obstruction as the real and imaginary parts may be considered separately.} by $\chi$. A somewhat subtle point is that when the Sobolev function $f$ is bounded, or $s > \frac{n}{2}$, the composition inherits the same regularity \cite{Moser1966}; however, when $s < \frac{n}{2}$, a small amount of regularity is (in general) lost \cite{RunstSickel1996} (for the threshold case $s=\frac{n}{2}$, $s>1$, see \cite{BrezisMironescu2001}; in this case the regularity is preserved provided the function $G$ itself and its higher derivatives are bounded). Denote by $\mathcal{C}_b^\infty(\mathbb{R}) = \bigcap_{m=0}^\infty \mathcal{C}^m_b(\mathbb{R})$ the space of smooth and bounded functions on $\mathbb{R}$ with bounded derivatives. 

\begin{lemma}[Theorem 2, \S5.3.6, \cite{RunstSickel1996}]
    \label{lem:composition_Sobolev_functions_1}
    Let $G' \in \mathcal{C}^\infty_b(\mathbb{R})$ and $G(0) = 0$. Suppose
    \[ 1 \leq s < \mu  \]
    and $f \in H^s \cap L^\infty$ on $\Omega$. Then there exists a constant $c>0$ such that
    \[ \| G(f) \|_{H^s} \leq c \|f\|_{H^s} \left( 1 + \| f \|_{L^\infty}^{\mu-1} \right). \]
\end{lemma}

\begin{lemma}[Theorem 3, \S5.3.6, \cite{RunstSickel1996}]
    \label{lem:composition_Sobolev_functions_2}
    Let $G' \in \mathcal{C}^\infty_b(\mathbb{R})$ and $G(0) = 0$. Suppose 
    \[ 1 < s < \frac{n}{2} \]
    and $f \in H^s$ on $\Omega$. Then there exists a constant $c > 0$ such that
    \[ \| G(f) \|_{H^{\varrho_1}} \leq c \left( \|f\|_{H^s} + \|f\|_{H^s}^{\varrho_1} \right), \]
    where
    \[ \varrho_1 = \frac{\frac{n}{2}}{\frac{n}{2}-s + 1}. \]
\end{lemma}

\noindent For completeness, we note that the case $0 < s \leq 1$ is in fact better than the case $1 < s < \frac{n}{2}$.

\begin{lemma}[Proposition 4.1, \cite{Taylor2007}, or Theorem 1 and Remark 2, \S5.3.6, \cite{RunstSickel1996}]
    \label{lem:composition_Sobolev_functions_3}
    Let $G' \in \mathcal{C}^\infty_b(\mathbb{R})$ and $G(0) = 0$. Suppose
    \[ 0 < s < 1 \]
    and $f \in H^s$ on $\Omega$. Then there exists a constant $c>0$ such that
    \[ \| G(f) \|_{H^s} \leq c \| f \|_{H^s}. \]
    Moreover, if in addition $G'' \in L^1(\mathbb{R})$, then the above estimate also holds for $s=1$.
\end{lemma}

\subsection{Product estimates in Sobolev spaces}

We frequently make use of the following product estimates in Sobolev spaces (see Theorem 2.2 in \cite{DAnconaFoschiSelberg2012}).
\begin{lemma} 
\label{lem:product_Sobolev_estimates}
Let $s_0$, $s_1$, $s_2 \in \mathbb{R}$. The product estimate
\[ \| fg \|_{H^{-s_0}(\mathbb{R}^n)} \la \| f \|_{H^{s_1}(\mathbb{R}^n)} \|g\|_{H^{s_2}(\mathbb{R}^n)} \]
holds for all $f, \, g \in \mathscr{S}(\mathbb{R}^n)$ if and only if
    \begin{align}
        \label{product_sobolev_estimates_condition_1}
        & s_0 + s_1 + s_2 \geq \frac{n}{2}, \\ 
        \label{product_sobolev_estimates_condition_2}
        & s_0 + s_1 \geq 0, \\
        \label{product_sobolev_estimates_condition_3}
        & s_0 + s_2 \geq 0, \\
        \label{product_sobolev_estimates_condition_4}
        & s_1 + s_2 \geq 0, \\
        \label{product_sobolev_estimates_condition_5}
        & \text{if \eqref{product_sobolev_estimates_condition_1} is an equality, then \eqref{product_sobolev_estimates_condition_2}--\eqref{product_sobolev_estimates_condition_4} must be strict.}
    \end{align}
\end{lemma}

By definition of restriction spaces \eqref{definition_fractional_Sobolev_spaces_on_domain} to bounded domains in $\mathbb{R}^n$, the same estimates hold with $\mathbb{R}^n$ replaced with $\Omega \subset \mathbb{R}^n$ bounded with Lipschitz boundary.

\subsection{Bourgain--Klainerman--Machedon dispersive Sobolev spaces}
\label{sec:dispersive_Sobolev_spaces}

For $s$, $b \in \mathbb{R}$, we denote by $X^{s,b} = X^{s,b}_\pm(\mathbb{R}^{1+3})$ the spaces\footnote{The spaces $X^{s,b}_\pm$ first appear in connection with questions of propagation of singularities of nonlinear wave equations in \cite{Bourgain1993a,Bourgain1993b,KlainermanMachedon1995a}, and are also sometimes called Fourier restriction spaces. See also \cite{KenigPonceVega1994}.} defined as the completion of the Schwartz space $\mathscr{S}(\mathbb{R}^{1+3})$ with respect to the norms
\[ \| u \|_{X^{s,b}_\pm} = \| \langle \xi \rangle^s \langle - \tau \pm |\xi| \rangle^b \hat{u}(\tau, \xi) \|_{L^2_{\tau,\xi}}, \]
where $\hat{u}(\tau, \xi)$ denotes the spacetime Fourier transform of $u(t,x)$. Let 
\begin{equation}
\label{free_wave_propagator}
U(\pm t) = \e^{\pm i |\nabla| t} \end{equation} 
be the free wave propagator defined by $\mathscr{F}_{x \to \xi}(U(\pm t) u(t,x)) = \e^{\pm i |\xi| t} \mathscr{F}_{x \to \xi}(u(t,x))$. Then the spaces $X^{s,b}_\pm$ are equivalently characterised as
\begin{equation} \label{propagator_characterisation_Bourgain_spaces} X^{s,b}_\pm(\mathbb{R}^{1+3}) = U(\pm t) H^b(\mathbb{R} ; H^s(\mathbb{R}^3)), \end{equation}
i.e. any $u \in X^{s,b}_\pm$ can be written as $u(t,x) = U(\pm t) f(t,x) $ for some $f(t,x) \in H^b(\mathbb{R}; H^s(\mathbb{R}^3))$. Indeed, this follows from the fact that if $u(t,x) = U(\pm t) f(t,x)$, then the spacetime Fourier transforms of $u(t,x)$ and $f(t,x)$ are related by $\hat{u}(\tau,\xi) = \hat{f}(\tau \mp |\xi|, \xi)$.  For $X \subset \mathbb{R}^{1+3}$, we denote the restriction spaces to $X$ by $X^{s,b}(X)$.
If $X= S_T = (-T, T) \times \mathbb{R}^3$, for $b > \frac{1}{2}$ one has the important embedding
\[ X^{s,b}_\pm(S_T) \hookrightarrow \mathcal{C}^0([-T,T]; H^s(\mathbb{R}^3)). \]
In addition, we denote by $H^{s,b}(\mathbb{R}^{1+3})$ the closely related wave-Sobolev spaces of Klainerman--Machedon, defined as the completion of $\mathscr{S}(\mathbb{R}^{1+3})$ with respect to the norm
\[  \| u \|_{H^{s,b}} = \| \langle \xi \rangle^s \langle |\tau| - |\xi| \rangle^b \hat{u}(\tau, \xi) \|_{L^2_{\tau,\xi}}. \]
It is straightforward to see, by splitting $\hat{u}(\tau,\xi)$ as $\hat{u}(\tau,\xi) = \hat{u}_-(\tau,\xi) + \hat{u}_+(\tau,\xi)$, where $\hat{u}_-(\tau,\xi) = 0 $ for $\tau >0$ and $\hat{u}_+(\tau,\xi) = 0 $ for $\tau < 0$, that any $u \in H^{s,b}$ may be written as $u = u_- + u_+$, where $u_\pm \in X^{s,b}_\pm$, with $\| u \|_{H^{s,b}} = \| u_- \|_{X_-^{s,b}} + \| u _+ \|_{X^{s,b}_+}$, i.e. $H^{s,b} \subset X^{s,b}_- + X^{s,b}_+$. One also has
\begin{align*}
    & \| u \|_{H^{s,b}} \leq \| u \|_{X_\pm^{s,b}} \qquad \text{for } b \geq 0, \\
    & \| u \|_{X_\pm^{s,b}} \leq \| u \|_{H^{s,b}} \qquad \text{for } b \leq 0, 
\end{align*}
i.e. for $b \geq 0$ we have the reverse inclusion $X^{s,b}_\pm \subset H^{s,b}$. Hence for $b \geq 0$, $H^{s,b} = X^{s,b}_- + X^{s,b}_+$. As before, we denote by $H^{s,b}(X)$ the restriction spaces to $X \subset \mathbb{R}^{1+3}$, defined exactly analogously to \eqref{definition_fractional_Sobolev_spaces_on_domain}.

\subsection{Product estimates in dispersive Sobolev spaces}

Analogous to the product estimates in Sobolev spaces of \Cref{lem:product_Sobolev_estimates}, one has the following lemma for product estimates in wave-Sobolev spaces on $\mathbb{R}^{1+3}$ (see Theorem 4.5, \cite{SelbergTesfahun2010}, or \cite{DAnconaFoschiSelberg2012}).

\begin{lemma}
    \label{lem:product_estimates_wave_Sobolev}
    Let $s_0, \, s_1, \, s_2, \, b_1, \, b_2 \in \mathbb{R}$ and $b_1, \, b_2 > \frac{1}{2}$. The product estimate
    \[ \| fg \|_{H^{-s_0,0}(\mathbb{R}^{1+3})} \la \| f \|_{H^{s_1,b_1}(\mathbb{R}^{1+3})} \| g \|_{H^{s_2,b_2}(\mathbb{R}^{1+3})}  \]
    holds for all $f, \, g \in \mathscr{S}(\mathbb{R}^{1+3})$ if
    \begin{align}
        & \label{416} s_0 + s_1 + s_2 \geq 2 - (b_1 + b_2), \\
        & \label{417} s_0 + s_1 + s_2 \geq \frac{3}{2} - b_1, \\
        & \label{418} s_0 + s_1 + s_2 \geq \frac{3}{2} - b_2, \\
        & \label{419} s_0 + s_1 + s_2 \geq 1, \\
        & \label{420} s_0 + 2(s_1 + s_2) \geq \frac{3}{2}, \\
        & \label{421} s_1 + s_2 \geq 0, \\
        & \label{422} s_0 + s_2 \geq 0, \\
        & \label{423} s_0 + s_1 \geq 0,
    \end{align}
    where \eqref{420} is strict if $s_0 \in \{ \frac{1}{2}, \frac{3}{2}, \frac{3}{2} - 2b_1, \frac{3}{2} - 2b_2, \frac{5}{2} - 2(b_1 + b_2) \}$, and \eqref{421}, \eqref{422}, \eqref{423} are strict if one of \eqref{416}, \eqref{417}, \eqref{418}, or \eqref{419} is an equality.
\end{lemma}

As in the case of \Cref{lem:product_Sobolev_estimates}, the same estimates hold with $\mathbb{R}^{1+3}$ replaced with a domain $X \subset \mathbb{R}^{1+3}$ with Lipschitz boundary, by definition of restriction spaces.

\subsection{Linear waves}

Let $S_T = (-T , T) \times \mathbb{R}^3$. For any $s \in \mathbb{R}$ and $b > \frac{1}{2}$ it is classical that for a given $F \in X^{s,b-1}_\pm(S_T)$ the linear initial value problem 
\[ (-i \partial_t \pm |\nabla| ) u = F \in X^{s,b-1}_\pm(S_T), \qquad u|_{t=0} = u_0 \in H^s(\mathbb{R}^3), \]
has a unique solution $u \in X^{s,b}_\pm(S_T)$ satisfying
\begin{equation}
    \label{basic_linear_Xsb_estimate}
    \| u \|_{X^{s,b}_\pm(S_T)} \la \| u_0 \|_{H^s(\mathbb{R}^3)} + \| F \|_{X^{s,b-1}_\pm(S_T)},
\end{equation} 
see e.g. \cite{KlainermanSelberg2002}. Moreover, the initial value problem for the linear wave equation 
\[ \Box_g u = f, \qquad (u, \partial_t u)|_{t = 0} = (u_0, u_1) \]
on a globally hyperbolic Lorentzian manifold $(\mathcalboondox{M} = \mathbb{R}_t \times \Sigma, g)$ satisfies the energy inequality \cite{Sogge1993}
    \begin{align}
    \begin{split}
    \label{energy_inequality}
    \|u(t,\cdot)\|_{H^{\sigma+1}(\Sigma_t)} + \| \partial_t u(t,\cdot) \|_{H^\sigma(\Sigma_t)} &\la \| u_0 \|_{H^{\sigma+1}(\Sigma_0)} + \| u_1 \|_{H^\sigma(\Sigma_0)} + \int_0^t \| f(t', \cdot) \|_{H^\sigma(\Sigma_t)} \, \d t' 
    \end{split}
    \end{align}
for all $\sigma \in \mathbb{R}$, where $\Sigma_t = \{ t \} \times \Sigma$.

\subsection{Conformal embedding of $\mathbb{M}$ into $\mathbb{R} \times \mathbb{S}^3$}
\label{sec:conformal_embedding}

For notational completeness we recap briefly the conformal embedding of Minkowski space $\mathbb{M} = \mathbb{R}^{1+3}$, with the standard metric $\eta_{ab} = \d t^2 - \d r^2 - r^2 g_{\mathbb{S}^2}$, into the Einstein cylinder $\mathbb{R}\times \mathbb{S}^3$ with metric $g_{ab} = \d \tau^2 - \d \zeta^2 - \sin^2\zeta g_{\mathbb{S}^2}$. This is achieved using the conformal factor 
\begin{equation}
    \label{standard_conformal_factor} \Omega = 2 \langle t-r \rangle^{-1} \langle t+r \rangle^{-1}
\end{equation}
and the relationships between the Minkowskian and cylindrical coordinates
\begin{align}
    \label{tau_definition_in_physical_coordinates}
    & \tau = \arctan(t+r) + \arctan(t-r), \\
    \label{zeta_definition_in_physical_coordinates}
    & \zeta = \arctan(t+r) - \arctan(t-r),
\end{align}
with $(\theta, \varphi)$ unchanged. In these coordinates the conformal factor \eqref{standard_conformal_factor} takes the form
\begin{equation}
    \label{standard_conformal_factor_cylinder_coordinates}
    \Omega = 2 \cos \left( \frac{\tau+\zeta}{2} \right) \cos \left( \frac{\tau - \zeta}{2} \right),
\end{equation}
and if $\eta_{ab}$ is multiplied by $\Omega^2$, one obtains exactly $g_{ab}$,
\[ \Omega^2 \eta_{ab} = g_{ab}. \]
The ``physical'' range of the Minkowskian coordinates $(t, r, \theta, \varphi)$ is represented in $\mathbb{R} \times \mathbb{S}^3$ by the set
\[ \iota_\Omega(\mathbb{M}) \defeq \{ (\tau, \zeta) \, : \, |\tau| + \zeta < \pi, ~ \zeta \geq 0 \} \times \mathbb{S}^2. \]
We denote this conformal embedding by
\begin{equation}
    \label{conformal_embedding}
    \iota_\Omega : \mathbb{M} \to \iota_\Omega (\mathbb{M}) \hookrightarrow \mathbb{R} \times \mathbb{S}^3,
\end{equation}
and frequently simply write\footnote{We occasionally abuse notation more generally, and write $X$ to denote $\iota_\Omega(X) \subset \mathbb{R} \times \mathbb{S}^3$ for subsets $X \subset \mathbb{M}$.} $\mathbb{M}$ to mean $\iota_\Omega(\mathbb{M})$. The time-zero Minkowskian and cylindrical slices coincide in $\mathbb{R} \times \mathbb{S}^3$, i.e. $\{ \tau = 0 \} = \iota_\Omega(\{ t= 0 \}) \cup i^0$. Moreover, the Minkowskian Killing vector field $\tilde{T}^a = (\partial_t)^a$ and the cylindrical Killing vector field $T^a = (\partial_\tau)^a$ differ by
\begin{align} \label{relation_between_Killing_fields} \tilde{T}^a &= \left( \cos^2\left( \frac{\tau-\zeta}{2} \right) + \cos^2\left( \frac{\tau+\zeta}{2} \right) \right)T^a - \sin \tau \sin \zeta Z^a, \\
T^a &= \frac{1}{2}(1+t^2 +r^2)\tilde{T}^a + rt R^a,
\end{align}
where
\[ Z^a \partial_a = \partial_\zeta \quad \text{and} \quad R^a \partial_a = \partial_r. \]
In particular, we have $\frac{1}{2} \langle r \rangle^{2} \tilde{T}^a |_{t=0} = T^a|_{\tau =0}$, i.e. they are initially parallel.

\section{The Maxwell-Scalar Field System}

\subsection{Basics and conformal symmetry}

The conformally invariant Maxwell-scalar field system describes the motion of a massless charged spin-0 particle interacting with an electromagnetic field. The system is given, on a general spacetime $(\mathcalboondox{M}, g)$, by the equations
\begin{align}  \label{general_Maxwell_equation} & \nabla^b F_{ab} = \Im ( \bar{\phi} \D_a \phi ), \\
\label{general_scalar_equation} & \D^a \D_a \phi + \frac{1}{6} \mathrm{R} \phi = 0, \\
& \label{general_Bianchi_identity}
\nabla_{[a} F_{bc]} = 0,
\end{align}
where $F_{ab}:\mathcalboondox{M} \to \mathbb{R}$ is a 2-form called the \emph{Maxwell field strength}, $\phi:\mathcalboondox{M} \to \mathbb{C}$ is a complex scalar field, $\mathrm{D}_a = \nabla_a + i A_a$ is the \emph{gauge covariant derivative}, $A_a : \mathcalboondox{M} \to \mathbb{R}$ is a 1-form called the \emph{Maxwell potential}, and $\mathrm{R}$ is the scalar curvature of $(\mathcalboondox{M},g)$. The 2-form $F_{ab}$ is given in terms of $A_a$ by $F_{ab} = 2 \nabla_{[a} A_{b]}$, which makes equation \eqref{general_Bianchi_identity} trivial (or, conversely, \eqref{general_Bianchi_identity} and the Poincar\'e lemma imply the existence of the potential $A_a$). The system is conformally invariant, i.e. massless, due to the presence of the conformal mass term $\mathrm{R}/6$ in \eqref{general_scalar_equation}. That is, for a smooth $\Omega>0$ the conformal rescaling
\begin{equation} \label{conformal_symmetry} g_{ab} = \Omega^2 \tilde{g}_{ab}, \qquad A_a = \tilde{A}_a, \qquad F_{ab} = \tilde{F}_{ab}, \qquad \phi = \Omega^{-1} \tilde{\phi}
\end{equation}
is a symmetry of \eqref{general_Maxwell_equation}--\eqref{general_Bianchi_identity}, with the tilded quantities satisfying the same equations if and only if \eqref{general_Maxwell_equation}--\eqref{general_Bianchi_identity} hold. The Lagrangian for the system \eqref{general_Maxwell_equation}--\eqref{general_Bianchi_identity} is 
\begin{equation}
    \label{general_Lagrangian}
    \mathcal{L} = - \frac{1}{4} F_{ab} F^{ab} + \frac{1}{2} \mathrm{D}_a \phi \overline{\mathrm{D}^a \phi} - \frac{1}{12} \mathrm{R} |\phi|^2,
\end{equation}
which also gives rise to the stress-energy tensor
\begin{equation}
    \label{general_stress_tensor}
    \mathbf{T}_{ab} = - F_{ac}F_b{}^c + \frac{1}{4} g_{ab} F_{cd} F^{cd} + \overline{\mathrm{D}_{(a} \phi } \mathrm{D}_{b)} \phi - \frac{1}{2} g_{ab} \overline{\mathrm{D}_c \phi }\mathrm{D}^c \phi + \frac{1}{12} g_{ab} \mathrm{R} |\phi|^2.
\end{equation}
When the equations \eqref{general_Maxwell_equation}--\eqref{general_Bianchi_identity} are satisfied, $\mathbf{T}_{ab}$ satisfies 
\[ \nabla^a \mathbf{T}_{ab} = \frac{1}{12} |\phi|^2 \nabla_b \mathrm{R}, \]
and so is conserved when $\mathrm{R}$ is constant and $\mathbf{T}_{ab}$ is sufficiently smooth.

\subsection{Gauge symmetry}

In addition to the conformal symmetry \eqref{conformal_symmetry}, \eqref{general_Maxwell_equation}--\eqref{general_Bianchi_identity} exhibit the \emph{gauge symmetry}
\begin{equation}
    \label{general_gauge_transformations}
    A_a \rightsquigarrow A_a + \nabla_a \chi, \qquad \phi \rightsquigarrow \e^{-i\chi} \phi.
\end{equation}
Under \eqref{general_gauge_transformations} the fields $F_{ab}$ and $\D_a \phi$ transform as
\[ F_{ab} \rightsquigarrow F_{ab}, \qquad \D_a \phi \rightsquigarrow \e^{-i\chi} \D_a \phi, \]
which in particular makes the Lagrangian \eqref{general_Lagrangian} and the stress tensor \eqref{general_stress_tensor} manifestly gauge invariant. Geometrically, this is the statement that the system \eqref{general_Maxwell_equation}--\eqref{general_Bianchi_identity} is a theory on a principal bundle $P \to \mathcalboondox{M}$ with fibre $\mathrm{U}(1)$. The gauge covariant derivative $\mathrm{D}_a$ is then a connection on $P$, represented by the real\footnote{Note that we choose $A_a$ to be real, i.e. the factor of $i$ in $\mathrm{D}_a = \nabla_a + iA_a$ comes from the Lie algebra of $\mathrm{U}(1)$ being $\mathfrak{u}(1) = i \mathbb{R}$.} 1-form $A_a$ on $\mathcalboondox{M}$ in any trivialisation of $P$, and $iF_{ab} = [\mathrm{D}_a, \mathrm{D}_b]$ its curvature. The scalar field $\phi$ is interpreted as a section of the complex line bundle over $\mathcalboondox{M}$ associated to $P$ by the representation $\e^{i\chi}$ of $\mathrm{U}(1)$, and the gauge transformations \eqref{general_gauge_transformations} correspond to a change of local trivialization. When written in terms of the potential $A_a$, the system \eqref{general_Maxwell_equation}--\eqref{general_Bianchi_identity} takes the form
\begin{align}
    \label{general_Maxwell_equation_potential}
    & \Box A_a - \nabla_a (\nabla_b A^b) + \mathrm{R}_{ab} A^b = -\operatorname{Im}(\bar{\phi} \mathrm{D}_a \phi), \\
    \label{general_scalar_equation_potential}
    & \Box \phi + 2 i A_a \nabla^a \phi + \left( \frac{1}{6} \mathrm{R} - A_a A^a + i \nabla_a A^a \right) \phi = 0.
\end{align}
Note that in Lorenz gauge $\nabla_a A^a = 0$ on $\mathcalboondox{M}$ the equations \eqref{general_Maxwell_equation_potential} and \eqref{general_scalar_equation_potential} become wave equations for $(A_a, \phi)$.

\subsection{Time slicing}

Let $\tau$ be a time function on $\mathcalboondox{M}$. Then the metric $g_{ab}$ may be decomposed as
\[ g = N^2 \d \tau^2 - h, \]
where $N = (\nabla_a \tau \nabla^a \tau)^{-1/2}$ is the lapse function and $h_{ab}$ is a Riemannian metric on the level surfaces $\Sigma_\tau$ of $\tau$ with future-oriented unit normal $T^a = N \nabla^a \tau$. The flow along $T^a$ allows to identify all hypersurfaces $\Sigma_\tau$ with a single $3$-manifold $\Sigma = \Sigma_0$ and thus defines the product structure $\mathcalboondox{M} = \mathbb{R}_\tau \times \Sigma$. This also defines the vector field $\partial_\tau$ as
\[ \partial_\tau = N T^a \partial_a .\]
Since we do not need anything more general, we assume that $N=1$ and that $h_{ab}$ is independent of $\tau$. In particular, this implies that the extrinsic curvature of $\Sigma_\tau$ vanishes for all $\tau$, and that $\partial_\tau$ is a Killing vector. Given this decomposition, we then define the \emph{electric} and \emph{magnetic} fields by $\mathbf{E}_a = T^b F_{ba}$ and $\mathbf{B}_a = \frac{1}{2} \varepsilon_{abc} F^{bc}$, where $\varepsilon_{abc}$ is the volume form of $\Sigma$. We also denote by $\boldsymbol{\nabla}_a = h_a{}^b \nabla_b$ the projection of $\nabla_b$ to $\Sigma$, which, since $N=1$, is also the Levi-Civita connection of $(\Sigma, h)$. Similarly, we denote by $\boldsymbol{\mathrm{D}}_a = h_a{}^b \mathrm{D}_b$ the projection of $\mathrm{D}_a$ to $\Sigma$. The system \eqref{general_Maxwell_equation}--\eqref{general_Bianchi_identity} then takes the form
\begin{align}
\label{general_MKG_system_with_projections}
\tag{MSF}
\begin{split}
    & \partial_\tau \mathbf{E} + \boldsymbol{\nabla} \times \mathbf{B} = - \operatorname{Im}(\bar{\phi} \boldsymbol{\mathrm{D}} \phi), \\
    & \mathrm{D}^a \mathrm{D}_a \phi + \frac{1}{6} \mathrm{R} \phi = 0, \\
    & \boldsymbol{\nabla} \cdot \mathbf{E} = - \operatorname{Im}(\bar{\phi} \mathrm{D}_0 \phi), \\
    & \boldsymbol{\nabla} \cdot \mathbf{B} = 0,
\end{split}
\end{align}
where $(\boldsymbol{\nabla} \times \mathbf{B} )_a = \varepsilon_{abc} \boldsymbol{\nabla}^b \mathbf{B}^c$, $\mathrm{D}_0 = T^a \mathrm{D}_a$, and $\boldsymbol{\nabla} \cdot \mathbf{E} = h^{ab} \boldsymbol{\nabla}_a \mathbf{E}_b$. The first two equations in \eqref{general_MKG_system_with_projections} are dynamical, whereas the third and fourth equations are constraints, and, for sufficiently smooth solutions, are propagated. The system \eqref{general_MKG_system_with_projections} is manifestly time-reversible, since $(\tau, A_0, \mathbf{E}) \mapsto (-\tau,-A_0,-\mathbf{E})$ is a symmetry. We note also that contracting \eqref{general_stress_tensor} with $T^a T^b$ gives rise to the conservation of energy
\begin{equation} \label{conservation_of_energy_in_general} \mathscr{E}(\tau) = \mathscr{E}(0)
\end{equation}
for all $\tau \in \mathbb{R}$, where
\[ \mathscr{E}(\tau) = \frac{1}{2} \int_{\Sigma_\tau} \left( | \mathbf{E} |^2 + |\mathbf{B}|^2 + |\mathrm{D}_0 \phi |^2 + |\boldsymbol{\mathrm{D}} \phi |^2 + \frac{\mathrm{R}}{6} |\phi|^2 \right) \dvol_{\Sigma}. \]
Here
\[ \mathrm{D}_0 \phi = \partial_\tau \phi + i A_0 \phi, \]
where $A_0 = (\partial_\tau)^a A_a$, and
\[ \boldsymbol{\mathrm{D}} \phi = \boldsymbol{\nabla} \phi + i \mathbf{A} \phi, \]
where $\mathbf{A}_a = h_a^{\phantom{a}b} A_b$ is the projection of $A_b$ to $\Sigma$.

\subsubsection{Einstein cylinder}
In the case of the Einstein cylinder $(\R_\tau \times \mathbb{S}^3, g)$, the metric is given by $g_{ab} = \d \tau^2 - h_{ab}$, where $h_{ab}$ is the standard round metric on $\mathbb{S}^3$. The system \eqref{general_MKG_system_with_projections} assumes the same form with $\mathrm{R} = 6$, or, in terms of the potential, \eqref{general_Maxwell_equation_potential}--\eqref{general_scalar_equation_potential} become
\begin{align}
    \label{cylinder_Maxwell_equation_potential}
    & \Box A_a - \nabla_a (\nabla_b A^b) -2 \mathbf{A}_a = -\operatorname{Im}(\bar{\phi} \mathrm{D}_a \phi), \\
    \label{cylinder_scalar_equation_potential}
    & \Box \phi + 2 i A_a \nabla^a \phi + \left( 1 - A_a A^a + i \nabla_a A^a \right) \phi = 0.
\end{align}
In any gauge which expresses $\nabla_a A^a$ in terms of zeroth-order derivatives of $A_a$ (and perhaps $\phi$)\footnote{A trivial example of such a gauge is the Lorenz gauge on the Einstein cylinder, $\nabla_a A^a = 0$. The Lorenz gauge on Minkowski space, $\tilde{\nabla}_a \tilde{A}^a = 0$, is another example, when the Minkowski spacetime is conformally embedded in the Einstein cylinder. See \Cref{sec:conformal_embedding}.}, \eqref{cylinder_Maxwell_equation_potential}--\eqref{cylinder_scalar_equation_potential} are a system of coupled nonlinear wave equations. In particular, in the Lorenz gauge $\nabla_a A^a = 0$ on $\mathbb{R} \times \mathbb{S}^3$, the system reads
\begin{align}
\label{cylinder_Lorenz_gauge_Maxwell_equation_potential}
    & \Box A_a -2 \mathbf{A}_a = -\operatorname{Im}(\bar{\phi} \mathrm{D}_a \phi), \\
    \label{cylinder_Lorenz_gauge_scalar_equation_potential}
    & \Box \phi + 2 i A_a \nabla^a \phi + \left( 1 - A_a A^a \right) \phi = 0,
\end{align}
and the gauge condition $\nabla_a A^a = 0$ replaces the constraints. Indeed, for sufficiently smooth solutions the quantity $\lambda = \nabla_a A^a $ satisfies, as a consequence of 
\eqref{cylinder_Lorenz_gauge_Maxwell_equation_potential} and \eqref{cylinder_Lorenz_gauge_scalar_equation_potential}, the linear wave equation
\[ \Box \lambda + |\phi|^2 \lambda = 0, \]
so the condition $\lambda = 0$ is propagated provided $(\lambda, \partial_\tau \lambda)|_{\tau = 0} = (0,0)$. But in fact $\partial_\tau \lambda = - \boldsymbol{\nabla} \cdot \mathbf{E} - \operatorname{Im}(\bar{\phi} \D_0 \phi) = 0$ is satisfied trivially provided the initial data satisfies the constraints, so that $\lambda|_{\tau = 0} = 0$ is a sufficient condition for the Lorenz gauge condition to be propagated.

We have in the case of the Einstein cylinder that $T^a = \partial_\tau$, so the conserved energy becomes
\begin{equation}
    \label{cylinder_energy}
    \mathscr{E}(\tau) = \frac{1}{2} \int_{\{ \tau \} \times \mathbb{S}^3} \left( | \mathbf{E} |^2 + |\mathbf{B}|^2 + |\mathrm{D}_0 \phi |^2 + |\boldsymbol{\mathrm{D}} \phi |^2 + |\phi|^2 \right) \dvol_{\mathbb{S}^3}.
\end{equation}

\begin{definition}
    \label{defn_finite_energy_data_on_cylinder}
    We say that $(\mathbf{E}_0, \mathbf{B}_0, U, \phi_0)$, where $U = (U_0, \mathbf{U})$, constitutes a set of finite energy initial data for the system \eqref{general_MKG_system_with_projections} on $\mathbb{R} \times \mathbb{S}^3$ if 
    \begin{equation}
        \label{finite_energy_initial_data_cylinder}
    (\mathbf{E}_0, \mathbf{B}_0, U, \phi_0) \in L^2(\mathbb{S}^3) \times L^2(\mathbb{S}^3) \times L^2(\mathbb{S}^3) \times L^2(\mathbb{S}^3),
    \end{equation} 
    and the constraints
    \begin{align}
    \label{first_initial_constraint_cylinder}
    & \boldsymbol{\nabla} \cdot \mathbf{E}_0 = -\operatorname{Im}(\bar{\phi}_0 U_0), \\
    \label{second_initial_constraint_cylinder}
    & \boldsymbol{\nabla} \cdot \mathbf{B}_0 = 0
    \end{align}
are satisfied in the sense of distributions.
\end{definition}

\subsubsection{Minkowski space}

In the case of Minkowski space $(\mathbb{M}, \eta)$ the metric is the standard metric $\eta_{ab} = \d t^2 - \delta_{ab}$, where $\delta_{ab}$ is the 3-dimensional Euclidean metric. Here the timelike vector field is $\tilde{T}^a = \partial_t$, and the system \eqref{general_MKG_system_with_projections} becomes
\begin{align}
\label{Minkowski_MKG_system_with_projections}
\tag{MSF$\mathbb{M}$}
\begin{split}
    & \partial_t \tilde{\mathbf{E}} + \tilde{\boldsymbol{\nabla}} \times \tilde{\mathbf{B}} = - \operatorname{Im}\big(\bar{\tilde{\phi}} \tilde{\boldsymbol{\mathrm{D}}} \tilde{\phi} \big), \\
    & \tilde{\mathrm{D}}^a \tilde{\mathrm{D}}_a \tilde{\phi} = 0, \\
    & \tilde{\boldsymbol{\nabla}} \cdot \tilde{\mathbf{E}} = - \operatorname{Im}\big(\bar{\tilde{\phi}} \tilde{\mathrm{D}}_0 \phi \big), \\
    & \tilde{\boldsymbol{\nabla}} \cdot \tilde{\mathbf{B}} = 0.
\end{split}
\end{align}
In Lorenz gauge $\tilde{\nabla}_a \tilde{A}^a = 0$ these equations become
\begin{align}
    \label{Minkowski_Maxwell_equation_potential}
    & \widetilde{\Box} \tilde{A}_a = - \operatorname{Im}\big( \bar{\tilde{\phi}} \tilde{\mathrm{D}}_a \tilde{\phi} \big) \eqdef \mathcalboondox{N}(\tilde{A}_a, \tilde{\phi}), \\
    \label{Minkowski_scalar_equation_potential}
    & \widetilde{\Box} \tilde{\phi} = -2i \tilde{A}_a \tilde{\nabla}^a \tilde{\phi} + \tilde{A}_a \tilde{A}^a \tilde{\phi} \eqdef \mathcalboondox{M}(\tilde{A}_a, \tilde{\phi}),
\end{align}
and the gauge condition $\tilde{\nabla}_a \tilde{A}^a = 0$ replaces the constraints.

Furthermore, by differentiating \eqref{general_Bianchi_identity} and contracting with $\partial_t$ or projecting to $\tilde{\Sigma}_t = \{ t = \text{const.} \}$, one derives gauge-invariant wave equations for $\tilde{\mathbf{E}}$ and $\tilde{\mathbf{B}}$:
\begin{align}
    \label{Minkowski_wave_equation_electric_field}
    & \widetilde{\Box} \tilde{\mathbf{E}} = -\operatorname{Im} \Big( \partial_t \tilde{\phi} \tilde{\boldsymbol{\nabla}}\bar{\tilde{\phi}} - \tilde{\boldsymbol{\nabla}} \tilde{\phi} \partial_t\bar{\tilde{\phi}} \Big) - \partial_t(\tilde{\mathbf{A}} |\tilde{\phi}|^2) + \tilde{\boldsymbol{\nabla}} ( \tilde{A}_0 |\tilde{\phi}|^2) \eqdef \boldsymbol{\mathcalboondox{P}}(\tilde{A}_a, \tilde{\phi}), \\
    \label{Minkowski_wave_equation_magnetic_field}
    & \widetilde{\Box} \tilde{\mathbf{B}}  = \operatorname{Im}\Big( \tilde{\boldsymbol{\nabla}} \tilde{\phi} \times \tilde{\boldsymbol{\nabla}} \bar{\tilde{\phi}} \Big) + \tilde{\boldsymbol{\nabla}} \times (\tilde{\mathbf{A}} |\tilde{\phi}|^2 ) \eqdef \boldsymbol{\mathcalboondox{Q}}(\tilde{A}_a \tilde{\phi}).
\end{align}

The conserved energy on Minkowski space is given by
\begin{equation}
    \label{conserved_energy_Minkowski}
    \tilde{\mathscr{E}}(t) = \frac{1}{2} \int_{\{ t\} \times \mathbb{R}^3} \left( |\tilde{\mathbf{E}}|^2 + |\tilde{\mathbf{B}}|^2 + |\tilde{\mathrm{D}}_0 \tilde{\phi}|^2 + |\tilde{\boldsymbol{\mathrm{D}}} \tilde{\phi}|^2 \right) \dvol_{\mathbb{R}^3}.
\end{equation}
In addition to $\tilde{\mathscr{E}}(0) < \infty$, on $\mathbb{M}$ we will need the assumption\footnote{In fact, this assumption is needed already in the original proof by Klainerman and Machedon in the Coulomb gauge \cite{KlainermanMachedon1994}.} that initially the scalar field $\tilde{\phi}$ is in $L^2(\mathbb{R}^3)$. This makes no further restriction on our assumptions on the data on $\mathbb{R} \times \mathbb{S}^3$, but is a technical assumption needed to apply the theorems of \cite{SelbergTesfahun2010}. We therefore define finite energy initial data for the system \eqref{Minkowski_MKG_system_with_projections} on Minkowski space as follows.

\begin{definition} \label{defn:finite_energy_initial_data_Minkowski}
    We say that $(\tilde{\mathbf{E}}_0, \tilde{\mathbf{B}}_0, \tilde{U}, \tilde{\phi}_0)$, where $\tilde{U} = (\tilde{U}_0, \tilde{\mathbf{U}})$, constitutes a set of finite energy initial data for the system \eqref{Minkowski_MKG_system_with_projections} if
    \begin{equation}
        \label{finite_energy_initial_data_Minkowski}
        (\tilde{\mathbf{E}}_0, \tilde{\mathbf{B}}_0, \tilde{U}, \tilde{\phi}_0) \in L^2(\mathbb{R}^3) \times L^2(\mathbb{R}^3) \times L^2(\mathbb{R}^3) \times L^2(\mathbb{R}^3),
    \end{equation} 
    and the constraints
    \begin{align}
        \label{Minkowski_constraint_1}
        & \tilde{\boldsymbol{\nabla}} \cdot \tilde{\mathbf{E}}_0 = - \operatorname{Im}\big(\bar{\tilde{\phi}}_0 \tilde{U}_0\big), \\
        \label{Minkowski_constraint_2}
        & \tilde{\boldsymbol{\nabla}} \cdot \tilde{\mathbf{B}}_0 = 0
    \end{align}
    are satisfied in the sense of distributions.
\end{definition}

In fact, we will use the observation of Selberg \& Tesfahun (\cite{SelbergTesfahun2010}, Remarks 2.1 and 2.2, and Lemma 2.1) that the residual gauge freedom in Lorenz gauge may be used to reduce the conditions $(\tilde{U}, \tilde{\phi}_0) \in L^2(\mathbb{R}^3) \times L^2(\mathbb{R}^3)$ to the simpler requirement $(\tilde{\phi}_1, \tilde{\phi}_0) \in L^2(\mathbb{R}^3) \times H^1(\mathbb{R}^3)$.

\begin{lemma} 
\label{lem:equivalence_of_Minkowskian_initial_data_Lorenz_gauge_and_temporal_initially}
Let $(\tilde{\mathbf{E}}_0, \tilde{\mathbf{B}}_0, \tilde{U}, \tilde{\phi}_0)$ be a set of finite energy initial data for the system \eqref{Minkowski_MKG_system_with_projections} in the sense of \Cref{defn:finite_energy_initial_data_Minkowski}. Then there exists a gauge which leaves the condition $\tilde{\nabla}_a \tilde{A}^a = 0$ unchanged, sets
\begin{equation} 
\label{Minkowski_residual_initial_gauge}
(\tilde{a}_0,\dot{\tilde{a}}_0) = (\tilde{A}_0, \partial_t \tilde{A}_0)|_{t=0} = (0,0), 
\end{equation}
and reduces the initial data set to
\[ (\tilde{\mathbf{E}}_0, \tilde{\mathbf{B}}_0, \tilde{\phi}_1, \tilde{\phi}_0 ) \in L^2(\mathbb{R}^3) \times L^2(\mathbb{R}^3) \times L^2(\mathbb{R}^3) \times H^1(\mathbb{R}^3), \]
where $\tilde{\phi}_1 = \partial_t \tilde{\phi}|_{t=0}$, with the estimates 
\begin{align*}
    &\| \dot{\tilde{\mathbf{a}}} \|_{L^2(\mathbb{R}^3)} = \| \tilde{\mathbf{E}}_0 \|_{L^2(\mathbb{R}^3)}, \\
    & \|\tilde{\mathbf{a}} \|_{\dot{H}^1(\mathbb{R}^3)} = \| \tilde{\mathbf{B}}_0 \|_{L^2(\mathbb{R}^3)}, \\
    & \| \tilde{\phi}_1 \|_{L^2(\mathbb{R}^3)} = \|\tilde{U}_0\|_{L^2(\mathbb{R}^3)}, \\
    & \|\tilde{\phi}_0 \|_{H^1(\mathbb{R}^3)} \leq 2 \|\tilde{\mathbf{U}} \|_{L^2(\mathbb{R}^3)} + C (1+\|\tilde{\mathbf{B}}_0 \|^2_{L^2(\mathbb{R}^3)}) \|\tilde{\phi}_0\|_{L^2(\mathbb{R}^3)},
\end{align*}
and the constraint \eqref{Minkowski_constraint_1} replaced with
\begin{equation}
\label{Minkowski_electric_field_constraint_under_initial_gauge}
    \tilde{\boldsymbol{\nabla}} \cdot \tilde{\mathbf{E}}_0 = - \operatorname{Im}(\bar{\tilde{\phi}}_0 \tilde{\phi}_1).
\end{equation} 
\end{lemma}

\begin{proof}
    The residual gauge freedom in Lorenz gauge is given by $\tilde{A}_a \rightsquigarrow \tilde{A}_a + \tilde{\nabla}_a \chi$, $\tilde{\phi} \rightsquigarrow \e^{-i \chi} \tilde{\phi}$, where $\chi$ satisfies
    \[ \widetilde{\Box}\chi = 0, \qquad (\chi, \partial_t \chi)|_{t=0} = (\chi_0, \chi_1). \]
    We set $\chi_1 = - \tilde{a}_0$ and, using the Lorenz gauge condition $\partial_t \tilde{A}_0 = \tilde{\boldsymbol{\nabla}} \cdot \tilde{\mathbf{A}}$, $\chi_0 = - \tilde{\boldsymbol{\Delta}}^{-1} \tilde{\boldsymbol{\nabla}} \cdot \tilde{\mathbf{a}}$, where $\tilde{\mathbf{a}} = \tilde{\mathbf{A}}|_{t=0}$ (here $u = \tilde{\boldsymbol{\Delta}}^{-1}f$ denotes any solution $u$ to $\tilde{\boldsymbol{\Delta}}u = f$). Then clearly $\tilde{a}_0 \rightsquigarrow 0$ and $\dot{\tilde{a}}_0 \rightsquigarrow \dot{\tilde{a}}_0 + \partial_t^2 \chi|_{t=0} = \dot{\tilde{a}}_0 + \tilde{\boldsymbol{\Delta}} \chi_0 = 0$. In this gauge
    \[ \tilde{\boldsymbol{\nabla}} \cdot \tilde{\mathbf{a}} = 0, \qquad \tilde{\mathbf{E}}_0  = \dot{\tilde{\mathbf{a}}}, \quad \text{and} \quad \tilde{\mathbf{B}}_0 = \tilde{\boldsymbol{\nabla}} \times \tilde{\mathbf{a}}, \]
    so $\| \dot{\tilde{\mathbf{a}}} \|_{L^2(\mathbb{R}^3)} = \| \tilde{\mathbf{E}}_0 \|_{L^2(\mathbb{R}^3)}$ and, by integrating by parts,
    \[ \| \tilde{\mathbf{a}} \|_{\dot{H}^1(\mathbb{R}^3)} = \| \tilde{\boldsymbol{\nabla}} \times \tilde{\mathbf{a}} \|_{L^2(\mathbb{R}^3)} + \| \tilde{\boldsymbol{\nabla}} \cdot \tilde{\mathbf{a}} \|_{L^2(\mathbb{R}^3)} = \| \tilde{\mathbf{B}}_0 \|_{L^2(\mathbb{R}^3)}.  \]
    Finally, writing $\tilde{U} = (\tilde{U}_0, \tilde{\mathbf{U}})$, we have $\tilde{U}_0 = \tilde{\phi}_1 + i \tilde{a}_0 \phi_0 = \tilde{\phi}_1 \in L^2(\mathbb{R}^3)$, and 
    \begin{equation}
    \label{spatial_components_initial_data_gauge_covariant_derivative_scalar_field}
    \tilde{\mathbf{U}} = \tilde{\boldsymbol{\nabla}} \tilde{\phi}_0 +i \tilde{\mathbf{a}} \tilde{\phi}_0,
    \end{equation}
    for which we have the estimate 
    \[ \| \tilde{\boldsymbol{\nabla}} \tilde{\phi}_0 \|_{L^2(\mathbb{R}^3)} \leq 2 \| \tilde{\mathbf{U}} \|_{L^2(\mathbb{R}^3)} + C \|\tilde{\mathbf{a}} \|^2_{\dot{H}^1(\mathbb{R}^3)} \|\tilde{\phi}_0 \|_{L^2(\mathbb{R}^3)} \]
    due to Selberg and Tesfahun \cite{SelbergTesfahun2010} (Lemma 2.1), i.e. $\tilde{\phi}_0 \in H^1(\mathbb{R}^3)$. Putting this together with the above gives the stated estimates and shows that the assumption $\tilde{\mathscr{E}}(0) + \|\tilde{\phi}_0\|_{L^2(\mathbb{R}^3)} < \infty$ in this gauge is equivalent to prescribing $(\tilde{\mathbf{E}}_0, \tilde{\mathbf{B}}_0, \tilde{\phi}_1, \tilde{\phi}_0 ) \in L^2(\mathbb{R}^3) \times L^2(\mathbb{R}^3) \times L^2(\mathbb{R}^3) \times H^1(\mathbb{R}^3)$, with the constraints \eqref{Minkowski_electric_field_constraint_under_initial_gauge} \& \eqref{Minkowski_constraint_2}.
\end{proof}

\subsection{Gauge fixing on $\mathbb{R} \times \mathbb{S}^3$}

We next observe that the Lorenz gauge on the Einstein cylinder $\nabla_a A^a = 0$ permits an analogue of \Cref{lem:equivalence_of_Minkowskian_initial_data_Lorenz_gauge_and_temporal_initially} on $\mathbb{R} \times \mathbb{S}^3$. 

\begin{lemma} \label{lem:initial_residual_gauge_on_cylinder}
    Let $(\mathbf{E}_0, \mathbf{B}_0, U, \phi_0) \in L^2(\mathbb{S}^3)^4$ be a set of finite energy initial data for \eqref{general_MKG_system_with_projections} on $\mathbb{R} \times \mathbb{S}^3$ in the sense of \Cref{defn_finite_energy_data_on_cylinder}. Then there exists a gauge which leaves the condition $\nabla_a A^a = 0$ unchanged, sets 
    \[ (a_0, \dot{a}_0) = (A_0, \partial_\tau A_0)|_{\tau = 0} = (0,0), \]
    and reduces the initial data to the set
    \[ (\mathbf{E}_0, \mathbf{B}_0, \phi_1, \phi_0) \in L^2(\mathbb{S}^3) \times L^2(\mathbb{S}^3) \times L^2(\mathbb{S}^3) \times H^1(\mathbb{S}^3), \]
    where $\phi_1 = \partial_\tau \phi |_{\tau = 0}$, with the estimates on the initial data
    \begin{align*}
        & \| \mathbf{a} \|_{H^1(\mathbb{S}^3)} \la \| \mathbf{B}_0 \|_{L^2(\mathbb{S}^3)}, \\
        & \| \dot{\mathbf{a}} \|_{L^2(\mathbb{S}^3)} \la \| \mathbf{E}_0 \|_{L^2(\mathbb{S}^3)}, \\
        & \| \phi_1 \|_{L^2(\mathbb{S}^3)} \la \| U_0 \|_{L^2(\mathbb{S}^3)}, \\
        & \|\phi_0 \|_{H^1(\mathbb{S}^3)} \la \| \mathbf{U} \|_{L^2(\mathbb{S}^3)} + \| \phi_0 \|_{L^2(\mathbb{S}^3)} \left(1 + \| \mathbf{B}_0 \|_{L^2(\mathbb{S}^3)} \right)^2,
    \end{align*}
    and the constraints 
    \[ \boldsymbol{\nabla} \cdot \mathbf{E}_0 = - \operatorname{Im}(\bar{\phi}_0 \phi_1) \quad \text{and} \quad \boldsymbol{\nabla} \cdot \mathbf{B}_0 = 0. \]
\end{lemma}

\begin{proof}
    The gauge transformations which fix $\nabla_a A^a = 0$ are given by solutions $\chi$ to the wave equation $\Box \chi = 0$. Choose the initial data for $\chi$ to be 
    \[ (\chi, \partial_\tau \chi)|_{\tau = 0} = (-\boldsymbol{\Delta}^{-1} \dot{a}_0, -a_0), \]
    where $\boldsymbol{\Delta}^{-1} f$ denotes any solution $u$ to $\boldsymbol{\Delta} u = f$ on $\mathbb{S}^3$ (the kernel of $\boldsymbol{\Delta}$ is the space of constant functions on $\mathbb{S}^3$; for concreteness, one may choose for instance $u$ in such a way that $\dashint u = 0$). Then
    \[ a_0 \rightsquigarrow 0, \qquad \dot{a}_0 \rightsquigarrow \dot{a}_0 + \partial_\tau^2 \chi|_{\tau = 0} = \dot{a}_0 + \boldsymbol{\Delta}\chi|_{\tau=0} = 0, \]
    and we have from $\mathbf{E}_0 = \dot{\mathbf{a}}$ and $\mathbf{B}_0 = \boldsymbol{\nabla} \times \mathbf{a}$
    \[ \| \dot{\mathbf{a}} \|_{L^2(\mathbb{S}^3)} = \| \mathbf{E}_0 \|_{L^2(\mathbb{S}^3)} \]
    and, by integrating by parts, using that now $\boldsymbol{\nabla} \cdot \mathbf{a} = 0$, and plugging in the Ricci curvature of $\mathbb{S}^3$,
    \[ \| \boldsymbol{\nabla} \mathbf{a} \|^2_{L^2(\mathbb{S}^3)} + 2 \| \mathbf{a} \|^2_{L^2(\mathbb{S}^3)} = \| \mathbf{B}_0 \|^2_{L^2(\mathbb{S}^3)}. \]
    Moreover, $U_0 = \phi_1$ in this gauge, and in particular $\| \phi_1 \|_{L^2(\mathbb{S}^3)} = \| U_0 \|_{L^2(\mathbb{S}^3)}$. Finally, to control the $L^2$ norm of $\boldsymbol{\nabla} \phi_0$ we have the estimate
    \begin{align*}
        \| \boldsymbol{\nabla} \phi_0 \|_{L^2(\mathbb{S}^3)} &\leq \| \mathbf{U}\|_{L^2(\mathbb{S}^3)} + \| \phi_0 \|_{L^3(\mathbb{S}^3)} \| \mathbf{a} \|_{L^6(\mathbb{S}^3)} \\
        & \leq \| \mathbf{U} \|_{L^2(\mathbb{S}^3)} + \| \phi_0 \|^{1/2}_{L^6(\mathbb{S}^3)} \| \phi_0 \|^{1/2}_{L^2(\mathbb{S}^3)} \| \mathbf{a} \|_{L^6(\mathbb{S}^3)} \\
        & \leq \| \mathbf{U} \|_{L^2(\mathbb{S}^3)} + C \left( \| \boldsymbol{\nabla} \phi_0 \|^{1/2}_{L^2(\mathbb{S}^3)} + \| \phi_0 \|^{1/2}_{L^2(\mathbb{S}^3)} \right)  \| \phi_0 \|^{1/2}_{L^2(\mathbb{S}^3)} \| \mathbf{a} \|_{H^1(\mathbb{S}^3)} \\
        & \leq \| \mathbf{U} \|_{L^2(\mathbb{S}^3)} + \frac{1}{2} \| \boldsymbol{\nabla} \phi_0 \|_{L^2(\mathbb{S}^3)} + C \| \phi_0 \|_{L^2(\mathbb{S}^3)}\left(1+\|\mathbf{a}\|_{H^1(\mathbb{S}^3)}\right)^2,
    \end{align*}
    where in the second line we used the interpolation inequality $\|f\|_{L^r} \leq \|f\|^t_{L^p} \|f\|^{1-t}_{L^q}$ for exponents $r^{-1} = tp^{-1} + (1-t)q^{-1}$, in the third line the Sobolev inequality $H^1(\mathbb{S}^3) \hookrightarrow L^6(\mathbb{S}^3)$, and Young's inequality in the fourth line. Moving the $\boldsymbol{\nabla} \phi_0$ term to the left-hand side then gives the required estimate.
\end{proof}

Finally, the following lemma expresses the Minkowskian residual gauge fixing conditions \eqref{Minkowski_residual_initial_gauge} in terms of quantities on $\mathbb{R} \times \mathbb{S}^3$. 

\begin{lemma} \label{lem:obtaining_Minkowski_residual_gauge_on_cylinder}
With respect to the conformal embedding $\iota_\Omega : \mathbb{M} \hookrightarrow \mathbb{R} \times \mathbb{S}^3$ described in \Cref{sec:conformal_embedding}, on the Einstein cylinder the residual gauge fixing conditions \eqref{Minkowski_residual_initial_gauge} are equivalent to 
    \begin{equation} \label{residual_gauge_fixing_cylinder_one_copy} 
    (a_0, \dot{a}_0 ) = \left(  0, \Gamma a_Z \right),
    \end{equation}
    where $a_Z = (Z^a A_a)|_{\tau = 0}$ and $\Gamma = \tan \frac{\zeta}{2}$.
\end{lemma}

\begin{proof}
    Contracting $A_a = \tilde{A}_a$ with $T^a = \partial_\tau$, differentiating with respect to $\tau$, and using the expression \eqref{relation_between_Killing_fields}, one derives after restricting to $\{t = 0 \} = \{ \tau = 0\}$ the equations
    \begin{align*}
        & \tilde{a}_0 = 2 a_0 \cos^2 \frac{\zeta}{2} , \\
        & \dot{\tilde{a}}_0 = 4 \dot{a}_0 \cos^4 \frac{\zeta}{2} - 4 a_Z \sin \frac{\zeta}{2} \cos^3 \frac{\zeta}{2}.
    \end{align*}
    The condition \eqref{residual_gauge_fixing_cylinder_one_copy} follows by setting $(\tilde{a}_0, \dot{\tilde{a}}_0) = (0,0)$.
\end{proof}

\section{Main Result}

Our main result is the following.

\begin{theorem} \label{thm:main_theorem}
Let $(\mathbf{E}_0, \mathbf{B}_0, U, \phi_0) \in L^2(\mathbb{S}^3) \times L^2(\mathbb{S}^3) \times L^2(\mathbb{S}^3) \times L^2(\mathbb{S}^3)$ be finite energy initial data for the Maxwell-scalar field system \eqref{general_MKG_system_with_projections} on $\mathbb{R} \times \mathbb{S}^3$ in the sense of \Cref{defn_finite_energy_data_on_cylinder}. Then for any $\delta > 0$ there exists a unique solution
\[ \mathbf{E}, \, \mathbf{B}, \, \D_a \phi \in \mathcal{C}^0 (\R_\tau; L^2 (\mathbb{S}^3)) \]
and
\[ \phi \in \mathcal{C}^0(\mathbb{R}_\tau; H^{1-\delta}(\mathbb{S}^3)) \cap \mathcal{C}^1(\mathbb{R}_\tau; H^{-\delta}(\mathbb{S}^3)) \]
to \eqref{general_MKG_system_with_projections} on $\mathbb{R} \times \mathbb{S}^3$ with respect to a potential
\[ A_a \in \mathcal{C}^0(\mathbb{R}_\tau;H^{\frac{1}{2}-\delta}(\mathbb{S}^3)) \cap \mathcal{C}^1(\mathbb{R}_\tau; H^{-\frac{1}{2}-\delta}(\mathbb{S}^3)) \]
satisfying the Lorenz gauge condition $\nabla_a A^a = 0$ in the sense of distributions. The solution satisfies the initial condition $(\mathbf{E}, \mathbf{B}, \D_a \phi, \phi)|_{\tau=0} =  (\mathbf{E}_0, \mathbf{B}_0, U, \phi_0)$ and the conservation of energy
\[ \mathscr{E}(\tau) = \mathscr{E}(0) \quad \forall \tau \in \mathbb{R}. \]
\end{theorem}

\begin{remark}
We show, moreover, that the solution obtained in \Cref{thm:main_theorem} is a limit in the cylinder energy norm \eqref{cylinder_energy} of smoother solutions on $\mathbb{R} \times \mathbb{S}^3$. See \Cref{rmk:limit_of_smooth_solutions_cylinder}.
\end{remark}

\section{Theorem of Selberg--Tesfahun}

Our proof of \Cref{thm:main_theorem} relies in an essential way on the following.

\begin{theorem}[Selberg--Tesfahun, \cite{SelbergTesfahun2010}]
    \label{Selberg_Tesfahun_theorem}
    Let $(\tilde{\mathbf{E}}_0, \tilde{\mathbf{B}}_0, \tilde{U}, \tilde{\phi}_0)$ be finite energy initial data for \eqref{Minkowski_MKG_system_with_projections} in the sense of \Cref{defn:finite_energy_initial_data_Minkowski} such that $(\tilde{a}_0, \dot{\tilde{a}}_0) = (\tilde{A}_0, \partial_t \tilde{A}_0)|_{t=0} = (0,0)$ (cf. \Cref{lem:equivalence_of_Minkowskian_initial_data_Lorenz_gauge_and_temporal_initially}). Then:
    \begin{enumerate}[(i)]
        \item there exist real numbers $T>0$ and $\epsilon>0$ sufficiently small such that for all $\delta > 0$ sufficiently small there exists a unique local solution to the system \eqref{Minkowski_MKG_system_with_projections} given by
        \begin{align*}
            & \tilde{\phi} \in H^{1,\frac{1}{2}+\epsilon}(S_T), ~~ \partial_t \tilde{\phi} \in H^{0,\frac{1}{2}+\epsilon}(S_T), \\
            & \tilde{\mathbf{E}}, ~ \tilde{\mathbf{B}} \in H^{0,\frac{1}{2}+\epsilon}(S_T), \\
            & \tilde{A}_a \in \dot{H}^{1,\frac{1}{2}+\epsilon}(S_T) + H^{1-\delta,\frac{1}{2}+\epsilon}(S_T)
        \end{align*}
        satisfying the Lorenz gauge $\tilde{\nabla}_a \tilde{A}^a = 0$, the conservation of energy $\tilde{\mathscr{E}}(t) = \tilde{\mathscr{E}}(0)$, where $\tilde{\mathscr{E}}$ is given by \eqref{conserved_energy_Minkowski}, and the initial condition $(\tilde{\mathbf{E}}, \tilde{\mathbf{B}}, \tilde{\D}_a\tilde{\phi}, \tilde{\phi})|_{t=0} = (\tilde{\mathbf{E}}_0, \tilde{\mathbf{B}}_0, \tilde{U}, \tilde{\phi}_0)$, where $S_T = (-T,T) \times \mathbb{R}^3$, and $T$ depends (continuously) only on the size of the initial energy\footnote{And the size of the initial $L^2$ norm of the scalar field, $\|\tilde{\phi}_0\|_{L^2(\mathbb{R}^3)}$, cf. \Cref{rmk:Klein_Gordon_mass}.} $\tilde{\mathscr{E}}(0)$, and
        \smallskip
        \item the above local solution extends to a unique global solution of the system \eqref{Minkowski_MKG_system_with_projections},
        \begin{align*}
            & \tilde{\phi} \in \mathcal{C}^0(\mathbb{R}_t;H^1(\mathbb{R}^3))\cap \mathcal{C}^1(\mathbb{R}_t;L^2(\mathbb{R}^3)), \\
            & \tilde{\mathbf{E}}, \, \tilde{\mathbf{B}} \in \mathcal{C}^0(\mathbb{R}_t;L^2(\mathbb{R}^3)),
        \end{align*}
        relative to a real-valued 4-potential $\tilde{A}_a$ such that
        \begin{align*}
            \tilde{A}_a \in \mathcal{C}^0(\mathbb{R}_t;\dot{H}^1(\mathbb{R}^3) + H^{1-\delta}(\mathbb{R}^3)) \cap \mathcal{C}^1(\mathbb{R}_t;H^{-\delta}(\mathbb{R}^3))
        \end{align*}
        satisfying the Lorenz gauge $\tilde{\nabla}_a \tilde{A}^a = 0$. Moreover, the solution satisfies the conservation of energy $\tilde{\mathscr{E}}(t) = \tilde{\mathscr{E}}(0)$ $\forall t \in \mathbb{R}$.
    \end{enumerate}
\end{theorem}

We make several remarks.

\begin{remark}
    The loss $\delta$ in \Cref{Selberg_Tesfahun_theorem} is arbitrarily small. This is used extensively in our argument.
\end{remark}

\begin{remark} 
\label{rmk:Klein_Gordon_mass}
\Cref{Selberg_Tesfahun_theorem} is proven explicitly in \cite{SelbergTesfahun2010} in the case of positive Klein--Gordon mass $m>0$, which results in the additional term $m^2 \tilde{\phi}$ in the wave equation for $\tilde{\phi}$ in \eqref{Minkowski_MKG_system_with_projections}. The positivity of the mass $m$ serves to control the $L^2$ norm of the scalar field in the energy $\tilde{\mathscr{E}}(t)$. Since our constructions rely on the conformal invariance of the equations, we in fact need the theorem for $m=0$. We remark that \Cref{Selberg_Tesfahun_theorem} extends with only cosmetic modifications to the proof to the case $m=0$, provided the initial assumption that $\tilde{\phi}_0 \in L^2(\mathbb{R}^3)$ is satisfied as a replacement; this is the reason we define the set of finite energy initial data on Minkowski space as in \Cref{defn:finite_energy_initial_data_Minkowski}. Indeed, for the local well-posedness (part (i) of \Cref{Selberg_Tesfahun_theorem}), the only change in the proof of \cite{SelbergTesfahun2010} is in estimate (4.6) (p. 1042), where the operator $\langle \tilde{\boldsymbol{\nabla}} \rangle_m$ is replaced with $| \tilde{\boldsymbol{\nabla}} |$. So there is only a change at low frequencies $|\xi| < 1$. But in this regime the estimate can be reduced to $L^1$ (instead of $L^2$) in space, which can be done just using H\"older and Sobolev inequalities. For the global solution (part (ii) of \Cref{Selberg_Tesfahun_theorem}), suppose that $T_{max.}$ is the maximal time of existence of the local solution; by the conservation of energy $\tilde{\mathscr{E}}(t) = \tilde{\mathscr{E}}(0)$ and the following estimate for the $L^2$ norm of $\tilde{\phi}$ at a later time $t$, 
\begin{align*} \frac{\partial}{\partial t} \| \tilde{\phi}(t,\cdot) \|^2_{L^2(\mathbb{R}^3)}  &= 2 \operatorname{Re} \int_{\mathbb{R}^3} \tilde{\phi} \overline{\partial_t \tilde{\phi}} \dvol_{\mathbb{R}^3} \\
& =2 \operatorname{Re}\int_{\mathbb{R}^3} \tilde{\phi} \overline{\tilde{\mathrm{D}}_0 \tilde{\phi}} \dvol_{\mathbb{R}^3} \leq 2 \|\tilde{\phi}(t,\cdot) \|_{L^2(\mathbb{R}^3)} \tilde{\mathscr{E}}(0),
\end{align*}
which consequently implies
\[ \|\tilde{\phi}(t,\cdot)\|_{L^2(\mathbb{R}^3)} \leq \|\tilde{\phi}_0\|_{L^2(\mathbb{R}^3)} + \tilde{\mathscr{E}}(0)t < \infty, \]
it then follows that the local solution may be extended beyond $t= T_{max.}$, yielding a contradiction and a global solution.
\end{remark}

\begin{remark}
    Note that the uniqueness statements in parts (i) and (ii) of \Cref{Selberg_Tesfahun_theorem} refer to the respective function spaces in parts (i) and (ii), i.e. in wave-Sobolev spaces $H^{s,b}$ with exponents $s$ and $b$ as stated in part (i), and in the natural energy spaces $\mathcal{C}^0 H^s \cap \mathcal{C}^1 H^{s-1}$ in part (ii). The former is sometimes called \emph{conditional} uniqueness, and the latter \emph{unconditional} uniqueness.
\end{remark}

\begin{remark}
    \label{rmk:continuous_dependence_on_data} The solution obtained in \Cref{Selberg_Tesfahun_theorem} enjoys continuous dependence on the data and persistence of higher regularity. That is, the solution $(\tilde{\phi}, \tilde{\mathbf{E}}, \tilde{\mathbf{B}})$ obtained in part (ii) is a limit, in $(\mathcal{C}^0_t H^1_x \cap \mathcal{C}^1_t L^2_x)\times \mathcal{C}^0_t L^2_x \times \mathcal{C}^0_t L^2_x$, of smooth solutions with smooth data. This in particular implies convergence in the (Minkowskian) energy norm, see \cite{SelbergTesfahun2010}, p. 1048.
\end{remark}

\begin{remark}
    For later use, we note the splitting of the (local) potential $\tilde{A}_a$ obtained in \Cref{Selberg_Tesfahun_theorem} into its homogeneous and inhomogeneous (i.e. low- and high-frequency) parts. We write
    \[ \tilde{A}_a = \tilde{A}_a^{(0)} + \tilde{A}_a^{\text{inh.}}, \]
    where $\tilde{A}_a^{(0)} \in \dot{H}^{1,\frac{1}{2}+\epsilon}(S_T)$ is supported on low spatial frequencies $|\xi| \leq 1$, and $\tilde{A}_a^{\text{inh.}} \in H^{1-\delta,\frac{1}{2}+\epsilon}(S_T)$ on high frequencies $|\xi| \geq 1$. 
\end{remark}

\section{Spacetime Estimates for Nonlinearities}
\label{sec:spacetime_estimates_for_nonlinearities}

The strategy of our proof requires extracting $L^2$-based spacetime Sobolev regularity for the right-hand sides $\mathcalboondox{M}(\tilde{A}_a, \tilde{\phi})$, $\mathcalboondox{N}(\tilde{A}_a, \tilde{\phi})$, $\boldsymbol{\mathcalboondox{P}}(\tilde{A}_a, \tilde{\phi})$ and $\boldsymbol{\mathcalboondox{Q}}(\tilde{A}_a, \tilde{\phi})$ of the equations \eqref{Minkowski_Maxwell_equation_potential}, \eqref{Minkowski_scalar_equation_potential}, \eqref{Minkowski_wave_equation_electric_field} and \eqref{Minkowski_wave_equation_magnetic_field}. Sufficient regularity in these spaces is a consequence of the null structure in $\mathcalboondox{M}$, $\boldsymbol{\mathcalboondox{P}}$ and $\boldsymbol{\mathcalboondox{Q}}$, and null form estimates in $X^{s,b}_\pm$ spaces which follow from the well-known estimates of Foschi--Klainerman \cite{FoschiKlainerman2000} for free waves. In this section we derive explicit expressions for $\mathcalboondox{M}$, $\boldsymbol{\mathcalboondox{P}}$ and $\boldsymbol{\mathcalboondox{Q}}$ in terms of classical null forms, modulo more regular terms. There is no null structure in $\mathcalboondox{N}$, and we obtain only rough estimates for this equation, but these are sufficient for our purposes. Throughout this section we assume that we have the local solutions $\tilde{\phi}$, $\tilde{A}_a$ and $\tilde{\mathbf{E}}$, $\tilde{\mathbf{B}}$ on $S_T \subset \mathbb{M}$, as given by \Cref{Selberg_Tesfahun_theorem}, where $S_T = (-T,T) \times\mathbb{R}^3$.

\subsection{Regularity of $\mathcalboondox{N}$}

\begin{lemma}
    \label{lem:regularity_of_N}
    For the local solution $\tilde{A}_a$, $\tilde{\phi}$ of \eqref{Minkowski_Maxwell_equation_potential}--\eqref{Minkowski_scalar_equation_potential} given by \Cref{Selberg_Tesfahun_theorem}, we have for some $T>0$
    \[ \mathcalboondox{N}(\tilde{A}_a, \tilde{\phi}) \in L^{\frac{3}{2}}(S_T) \hookrightarrow H^{-\frac{2}{3}}(S_T). \]
\end{lemma}

\begin{proof}
    We have $\mathcalboondox{N}(\tilde{A}_a, \tilde{\phi}) = - \operatorname{Im}(\bar{\tilde{\phi}} \tilde{\D}_a \tilde{\phi})$, where $\tilde{\phi} \in H^{1,\frac{1}{2}+\epsilon}(S_T)$ and $\tilde{\nabla}_a \tilde{\phi} \in H^{0,\frac{1}{2}+\epsilon}(S_T)$. The gauge covariant derivative of $\tilde{\phi}$ is given by $\tilde{\D}_a \tilde{\phi} = \tilde{\nabla}_a \tilde{\phi} + i \tilde{A}_a \tilde{\phi}$, where
    \[ \tilde{A}_a = \tilde{A}_a^{(0)} + \tilde{A}_a^{\text{inh.}}, \]
    with
    \[ \tilde{A}_a^{(0)} \in \dot{H}^{1,\frac{1}{2}+\epsilon}(S_T), \qquad \tilde{A}_a^{\text{inh.}} \in H^{1-\delta,\frac{1}{2} + \epsilon}(S_T). \]
    For the high-frequency part of the potential, we therefore have the estimate (using \Cref{lem:product_estimates_wave_Sobolev})
    \[ \| \tilde{A}^{\text{inh.}} \tilde{\phi} \|_{H^{1-\delta,0}(S_T)} \la \| \tilde{A}^{\text{inh.}} \|_{H^{1-\delta,\frac{1}{2}+\epsilon}(S_T)} \| \tilde{\phi} \|_{H^{1,\frac{1}{2}+\epsilon}(S_T)}. \]
    For the low-frequency part $\tilde{A}_a^{(0)}$, using $\dot{H}^{1,\frac{1}{2}+\epsilon}(S_T) \hookrightarrow L^\infty_t \dot{H}^1_x (S_T) \hookrightarrow L^\infty_t L^6_x (S_T)$ (the second embedding using \eqref{Sobolev_embedding_homogeneous_H1}), the embedding $H^1(\mathbb{R}^3) \hookrightarrow L^6(\mathbb{R}^3)$, and the interpolation inequality $\| u \|_{L^3(\mathbb{R}^3)} \la \| u \|_{L^2(\mathbb{R}^3)}^{1/2} \| u \|_{L^6(\mathbb{R}^3)}^{1/2} \la \|u\|_{H^1(\mathbb{R}^3)}$, we have
    \begin{align*} \| \tilde{A}^{(0)} \tilde{\phi} \|_{L^2(S_T)} &\leq \| \tilde{A}^{(0)} \tilde{\phi} \|_{L^\infty_tL^2_x(S_T)} \\
    &\la \| \tilde{A}^{(0)}\|_{L^\infty_t L^6_x(S_T)}\|\tilde{\phi}\|_{L^\infty_t L^3_x(S_T)} \\
    & \la \| \tilde{A}^{(0)}\|_{\dot{H}^{1,\frac{1}{2}+\epsilon}(S_T)} \| \tilde{\phi} \|_{L^\infty_t H^1_x(S_T)} \\
    &\la \| \tilde{A}^{(0)} \|_{\dot{H}^{1,\frac{1}{2}+\epsilon}(S_T)} \| \tilde{\phi}\|_{H^{1,\frac{1}{2}+\epsilon}(S_T)}.
    \end{align*}
    We therefore have
    \[ \tilde{\D}_a \tilde{\phi} = \tilde{\nabla}\tilde{\phi} + i \tilde{A}_a \tilde{\phi} \in H^{0,\frac{1}{2}+\epsilon}(S_T) + H^{1-\delta,0}(S_T) + L^2(S_T) \subset L^2(S_T). \]
    Hence 
    \[ \mathcalboondox{N}(\tilde{A}_a, \tilde{\phi}) \in H^{1,\frac{1}{2}+\epsilon}(S_T) \cdot L^2 (S_T) \hookrightarrow L^2_t (H^1_x \cdot L^2_x) \subset L^2_t L^{\frac{3}{2}}_x \subset L^{\frac{3}{2}}(S_T). \]
    Finally, we note that $H^{\frac{2}{3}}(S_T) \hookrightarrow L^3(S_T)$, so $L^{\frac{3}{2}}(S_T) \hookrightarrow H^{-\frac{2}{3}}(S_T)$.
    \end{proof}

\subsection{Null structure in $\mathcalboondox{M}$}

While it is absent in general, a key observation of \cite{SelbergTesfahun2010} is that in Lorenz gauge there is in fact null structure in $\mathcalboondox{M}(\tilde{A}_a, \tilde{\phi})$; recall that
\[ \mathcalboondox{M}(\tilde{A}_a, \tilde{\phi}) = - 2i \tilde{A}_a \tilde{\nabla}^a \tilde{\phi} + \tilde{A}_a \tilde{A}^a \tilde{\phi}, \]
and the classical null forms $Q_0$, $Q_{0i}$ and $Q_{ij}$ on $\mathbb{M}$ are given by
\begin{align*}
    Q_0(\tilde{\phi}, \tilde{\psi}) &= \partial_t \tilde{\phi} \, \partial_t \tilde{\psi} - \tilde{\boldsymbol{\nabla}}\tilde{\phi} \cdot \!\tilde{\boldsymbol{\nabla}} \tilde{\phi}, \\
    Q_{0i}(\tilde{\phi},\tilde{\psi}) & = \partial_t \tilde{\phi} \, \tilde{\nabla}_i \tilde{\psi} - \tilde{\nabla}_i \tilde{\phi} \,\partial_t \tilde{\psi}, \\
    Q_{ij}(\tilde{\phi}, \tilde{\psi})&= \tilde{\nabla}_i\tilde{\phi} \tilde{\nabla}_j \tilde{\psi} - \tilde{\nabla}_j \tilde{\phi} \tilde{\nabla}_i \tilde{\psi}.
\end{align*}

\begin{lemma}
    \label{lem:null_structure_in_M}
    There exist $\tilde{v} \in H^{2-\delta,\frac{1}{2}+\epsilon}(S_T)$ and $\tilde{\mathbf{w}} \in H^{2,\frac{1}{2}+\epsilon}(S_T)$ such that
    \begin{equation}
        \label{null_structure_M}
        -\frac{1}{2i} \mathcalboondox{M}(\tilde{A}_a, \tilde{\phi}) - Q_0(\tilde{v}, \tilde{\phi}) - \varepsilon^{kij} Q_{ij}(\tilde{\phi}, \tilde{w}_k) \in L^2(S_T).
    \end{equation}
\end{lemma}

\begin{proof}
We split $\tilde{A}_a^{(0)} + \tilde{A}_a^{\text{inh.}} \in \dot{H}^{1,\frac{1}{2}+\epsilon}(S_T) + H^{1-\delta,\frac{1}{2}+\epsilon}(S_T)$ and consider first the cubic term in $\mathcalboondox{M}(\tilde{A}_a, \tilde{\phi})$. We claim this is in $L^2(S_T)$. The inhomogeneous $\times$ inhomogeneous part can be controlled by two applications of \Cref{lem:product_estimates_wave_Sobolev}:
\begin{align*} \| (\tilde{A}^{\text{inh.}})^2 \tilde{\phi} \|_{L^2(S_T)} &\la \| (\tilde{A}^{\text{inh.}})^2 \|_{H^{\frac{1}{2}+\epsilon, 0}(S_T)} \| \tilde{\phi} \|_{H^{1, \frac{1}{2}+\epsilon}(S_T)} \\
&\la \| \tilde{A}^{\text{inh.}} \|^2_{H^{1-\delta, \frac{1}{2}+\epsilon}(S_T)} \| \tilde{\phi} \|_{H^{1, \frac{1}{2} +\epsilon}(S_T)}.
\end{align*}
For the inhomogeneous $\times$ low-frequency part, one similarly obtains
\begin{align*}
    \| \tilde{A}^{\text{inh.}}\tilde{A}^{(0)}\tilde{\phi}\|_{L^2(S_T)} & \la \| \tilde{A}^{(0)}\|_{\dot{H}^{1,\frac{1}{2}+\epsilon}(S_T)} \| \tilde{A}^{\text{inh.}}\tilde{\phi}\|_{H^{\frac{1}{2}+\epsilon,0}(S_T)} \\
    & \la \| \tilde{A}^{(0)}\|_{\dot{H}^{1,\frac{1}{2}+\epsilon}(S_T)} \| \tilde{A}^{\text{inh.}} \|_{H^{{1-\delta},\frac{1}{2}+\epsilon}(S_T)} \|\tilde{\phi} \|_{H^{1,\frac{1}{2}+\epsilon}(S_T)}.
\end{align*}
For the low-frequency $\times$ low-frequency part, the estimate is even simpler: using $H^{1,\frac{1}{2}+\epsilon}(S_T) \hookrightarrow L^\infty_t L^6_x(S_T)$ and $\dot{H}^1(\mathbb{R}^3) \hookrightarrow L^6(\mathbb{R}^3)$,
\begin{align*} \| (\tilde{A}^{(0)})^2 \tilde{\phi} \|_{L^2(S_T)} & \la \| (\tilde{A}^{(0)})^2\|_{L^\infty_t L^3_x (S_T)} \| \tilde{\phi} \|_{L^\infty_t L^6_x(S_T)} \\
& \la \| \tilde{A}^{(0)} \|^2_{L^\infty_t L^6_x(S_T)} \| \tilde{\phi} \|_{H^{1,\frac{1}{2}+\epsilon}(S_T)} \\
& \la \| \tilde{A}^{(0)} \|^2_{\dot{H}^{1,\frac{1}{2}+\epsilon}(S_T)} \| \tilde{\phi} \|_{H^{1,\frac{1}{2}+\epsilon}(S_T)}.
\end{align*}
Now for the differentiated term in $\mathcalboondox{M}(\tilde{A}_a, \tilde{\phi})$, we split the spatial part $\tilde{\mathbf{A}}$ of the potential into its divergence-free and curl-free parts,
\[ \tilde{\mathbf{A}} = - \tilde{\boldsymbol{\Delta}}^{-1} \tilde{\boldsymbol{\nabla}} \times \tilde{\boldsymbol{\nabla}} \times \tilde{\mathbf{A}} + \tilde{\boldsymbol{\Delta}}^{-1} \tilde{\boldsymbol{\nabla}} (\tilde{\boldsymbol{\nabla}} \cdot \tilde{\mathbf{A}}) \eqdef \tilde{\mathbf{A}}^{\mathrm{df}} + \tilde{\mathbf{A}}^{\mathrm{cf}} \]
and hence rewrite
\[ -\tilde{A}_a \tilde{\nabla}^a \tilde{\phi} = ( - \tilde{A}_0 \partial_t \tilde{\phi} + \tilde{\mathbf{A}}^{\mathrm{cf}} \cdot \tilde{\boldsymbol{\nabla}}\tilde{\phi} ) + \tilde{\mathbf{A}}^{\mathrm{df}} \cdot \tilde{\boldsymbol{\nabla}} \tilde{\phi} \eqdef P_1 + P_2, \]
where we claim that both $P_1$ and $P_2$ are null forms. Indeed,
\[ P_2 = \tilde{\mathbf{A}}^{\mathrm{df}} \cdot \tilde{\boldsymbol{\nabla}} \tilde{\phi} = - ( \tilde{\boldsymbol{\nabla}} \times \tilde{\mathbf{w}} ) \cdot \tilde{\boldsymbol{\nabla}} \tilde{\phi} = - \varepsilon^{kij} Q_{ij}(\tilde{\phi}, \tilde{w}_k), \]
where $\tilde{\boldsymbol{\Delta}} \tilde{\mathbf{w}} = \tilde{\mathbf{B}} \in H^{0,\frac{1}{2}+\epsilon}(S_T)$, to which there exists\footnote{This solution $\tilde{\mathbf{w}}$ is, of course, not unique; it is sufficient to require that $\tilde{\mathbf{w}} \in H^{0,\frac{1}{2}+\epsilon}(S_T)$ to obtain the stated regularity, since $\| u \|_{H^2} \la \|u\|_{L^2} + \| \tilde{\boldsymbol{\Delta}} u \|_{L^2}$.} a solution such that $\tilde{\mathbf{w}} \in H^{2,\frac{1}{2}+\epsilon}(S_T)$ that we denote by $\tilde{\mathbf{w}} = \tilde{\boldsymbol{\Delta}}^{-1} \tilde{\mathbf{B}} = \tilde{\boldsymbol{\Delta}}^{-1} \tilde{\boldsymbol{\nabla}} \times \tilde{\mathbf{A}} $. Further, since in Lorenz gauge $\partial_t \tilde{A}_0 = \tilde{\boldsymbol{\nabla}} \cdot \tilde{\mathbf{A}}$ and hence $\tilde{\mathbf{A}}^{\mathrm{cf}} = \tilde{\boldsymbol{\Delta}}^{-1} \tilde{\boldsymbol{\nabla}} \partial_t \tilde{A}_0$,
\[ P_1 = - \tilde{A}_0 \partial_t \tilde{\phi} + \tilde{\mathbf{A}}^{\mathrm{cf}} \cdot \tilde{\boldsymbol{\nabla}} \tilde{\phi} = - \tilde{A}_0 \partial_t \tilde{\phi} + \tilde{\boldsymbol{\nabla}} \tilde{\boldsymbol{\Delta}}^{-1} \partial_t \tilde{A}_0 \cdot \tilde{\boldsymbol{\nabla}} \tilde{\phi}. \]
Put $\tilde{v} \defeq \tilde{\boldsymbol{\Delta}}^{-1} \partial_t \tilde{A}_0$ (i.e., as above, choose a solution $\tilde{v}$ to $\tilde{\boldsymbol{\Delta}} \tilde{v} = \partial_t \tilde{A}_0 \in H^{-\delta,\frac{1}{2}+\epsilon}(S_T)$ such that $\tilde{v} \in H^{2-\delta,\frac{1}{2}+\epsilon}(S_T)$). Then, using the equation for $\tilde{A}_0$,
\begin{align*} \partial_t \tilde{v} = \tilde{\boldsymbol{\Delta}}^{-1} \partial_t^2 \tilde{A}_0 = \tilde{A}_0 + \tilde{\boldsymbol{\Delta}}^{-1} \big(\! -\operatorname{Im} \big( \bar{\tilde{\phi}} \tilde{\mathrm{D}}_0 \tilde{\phi} \big) \big),
\end{align*}
so that
\[ P_1 = - Q_0(\tilde{v}, \tilde{\phi}) + \partial_t \tilde{\phi} \tilde{\boldsymbol{\Delta}}^{-1} \big( \! \operatorname{Im} \big( \bar{\tilde{\phi}} \partial_t \tilde{\phi} \big) \big) + \partial_t \tilde{\phi} \tilde{\boldsymbol{\Delta}}^{-1} \big( \tilde{A}_0 |\tilde{\phi}|^2 \big).  
\]
Since $\partial_t \tilde{\phi} \in H^{0,\frac{1}{2}+\epsilon}(S_T)$, we have $\bar{\tilde{\phi}} \partial_t \tilde{\phi} \in H^{-\eta,0}(S_T)$ for any $\eta > 0$ using $H^{1,\frac{1}{2}+\epsilon} \cdot H^{0,\frac{1}{2}+\epsilon} \hookrightarrow H^{-\eta, 0}$. Moreover, the term $\tilde{A}_0 |\tilde{\phi}|^2$ is in $L^2(S_T)$ by a similar argument as for $\tilde{A}^2 \tilde{\phi}$,
\begin{align}
\label{L2_estimate_A_phi_squared}
\begin{split}
\| \tilde{A}^{\text{inh.}} |\tilde{\phi}|^2 \|_{L^2(S_T)} &\la \| |\tilde{\phi}|^2 \|_{H^{\frac{1}{2}+\epsilon+\delta,0}(S_T)} \| \tilde{A}^{\text{inh.}} \|_{H^{1-\delta,\frac{1}{2}+\epsilon}(S_T)} \\
& \la \| \tilde{\phi}\|^2_{H^{1,\frac{1}{2}+\epsilon}(S_T)} \| \tilde{A}^{\text{inh.}} \|_{H^{1-\delta,\frac{1}{2}+\epsilon}(S_T)},
\end{split}
\end{align}
and
\begin{align}
\label{L2_estimate_A_phi_squared_2}
\begin{split}
    \| \tilde{A}^{(0)} |\tilde{\phi}|^2 \|_{L^2(S_T)} & \la \| \tilde{A}^{(0)} \|_{L^\infty_t L^6_x (S_T)} \| |\tilde{\phi}|^2 \|_{L^\infty_t L^3_x(S_T)} \\ 
    & \la \| \tilde{A}^{(0)} \|_{\dot{H}^{1,\frac{1}{2}+\epsilon}(S_T)} \| \tilde{\phi} \|_{H^{1,\frac{1}{2}+\epsilon}(S_T)}.
\end{split}
\end{align}
Hence we have $\operatorname{Im}\big( \bar{\tilde{\phi}} \tilde{\mathrm{D}}_0 \tilde{\phi} \big) \in H^{-\eta, 0}(S_T)$ (for any $\eta >0$), and so
\[ \partial_t \tilde{\phi} \tilde{\boldsymbol{\Delta}}^{-1}\big(\!\operatorname{Im}\big( \bar{\tilde{\phi}} \tilde{\mathrm{D}}_0 \tilde{\phi} \big)\big) \in H^{0,\frac{1}{2}+\epsilon} \cdot L^2_t H^{2-\eta}_x \hookrightarrow L^\infty_tL^2_x \cdot L^2_t L^\infty_x \subset L^2(S_T). \]
Hence up to a term which is in $L^2(S_T)$, $P_1$ is equal to the classical null form $-Q_0(\tilde{v}, \tilde{\phi})$, and \eqref{null_structure_M} follows.
\end{proof}

\subsection{Null structure in $\boldsymbol{\mathcalboondox{P}}$ and $\boldsymbol{\mathcalboondox{Q}}$}

\begin{lemma}
    \label{lem:null_structure_P_and_Q}
    The nonlinearities $\boldsymbol{\mathcalboondox{P}}(\tilde{A}_a, \tilde{\phi})$ and $\boldsymbol{\mathcalboondox{Q}}(\tilde{A}_a, \tilde{\phi})$ are given by 
    \begin{align*} \boldsymbol{\mathcalboondox{P}}(\tilde{A}_a, \tilde{\phi}) &+ \operatorname{Im}\big(Q_{0i}(\tilde{\phi},\bar{\tilde{\phi}}) \big) \in \dot{H}^{-1}(S_T), \\
    \boldsymbol{\mathcalboondox{Q}}(\tilde{A}_a, \tilde{\phi}) &- \epsilon^{ijk}\operatorname{Im} \big( Q_{jk}(\tilde{\phi},\bar{\tilde{\phi}}) \big) \in \dot{H}^{-1}(S_T), \end{align*}
    where $\tilde{\phi} \in H^{1,\frac{1}{2}+\epsilon}(S_T)$ is the scalar field obtained in \Cref{Selberg_Tesfahun_theorem}.
\end{lemma}

\begin{proof}
This follows directly from the expressions for $\boldsymbol{\mathcalboondox{P}}(\tilde{A}_a, \tilde{\phi})$ and $\boldsymbol{\mathcalboondox{Q}}(\tilde{A}_a, \tilde{\phi})$. We have
\begin{equation}
\label{null_structure_P}
\boldsymbol{\mathcalboondox{P}}(\tilde{A}_a, \tilde{\phi}) = - \operatorname{Im} \big( Q_{0i}(\tilde{\phi},\bar{\tilde{\phi}}) \big) - \partial_t(\tilde{\mathbf{A}} |\tilde{\phi}|^2) + \tilde{\boldsymbol{\nabla}} ( \tilde{A}_0 |\tilde{\phi}|^2),
\end{equation}
where by \eqref{L2_estimate_A_phi_squared} and \eqref{L2_estimate_A_phi_squared_2}, the last two terms in \eqref{null_structure_P} are in $\dot{H}^{-1}(S_T)$. Similarly,
\begin{equation}
\label{null_structure_Q}
\boldsymbol{\mathcalboondox{Q}}(\tilde{A}_a, \tilde{\phi}) = \epsilon^{ijk}\operatorname{Im} \big( Q_{jk}(\tilde{\phi},\bar{\tilde{\phi}}) \big) + \boldsymbol{\tilde{\nabla}} \times ( \tilde{\mathbf{A}} |\tilde{\phi}|^2 ),
\end{equation}
where the second term is in $\dot{H}^{-1}(S_T)$.
\end{proof}

\subsection{Regularity of the null nonlinearities}

We control the null nonlinearities using the following estimates for null forms in $H^{s,b}$ spaces. The first of these appears in Masmoudi--Nakanishi \cite{MasmoudiNakanishi2003}, Lemma 2.3, with the most important underlying estimates due to Foschi and Klainerman \cite{FoschiKlainerman2000}. By definition of restrictions spaces, these estimates hold also on $S_T$.

\begin{lemma}[Masmoudi--Nakanishi] Let $b> \frac{1}{2}$ and $s_1, \, s_2\leq 2 \leq s_1 + s_2$. Then
    \begin{equation} \label{Masmoudi_Nakanishi_estimate}
        \| Q_{ij}(\tilde{u}, \tilde{v}) \|_{L^2_t \dot{H}^{s_1 + s_2 - 3}_x} \la \| \tilde{u} \|_{H^{s_1,b}} \| \tilde{v} \|_{H^{s_2,b}}.
    \end{equation}
\end{lemma}

\begin{remark}
Note that the above estimate is valid for the endpoint case $s_1 + s_2 = 2$. The situation is slightly more delicate in the case of $Q_0$ and $Q_{0i}$: for these null forms the endpoint estimate in $L^2_t H^{-1}_x$ is prohibited. Although in all cases we will in fact be interested in the weaker estimates in $H^{-1}$ in spacetime, for completeness we record the slight differences between the null forms involving time derivatives and the ones without. For $Q_0$, we can actually get away with an asymmetric estimate in $L^2_t \dot{H}^{s-2}_x$, since in \eqref{null_structure_M} the null form $Q_0(\tilde{v},\tilde{\phi})$ contains the more regular $\tilde{v}$, which ensures that the required estimate stays away from endpoints. For $Q_{0i}$, however, we will need the full estimate in $H^{-1}$ in spacetime.
\end{remark}

\begin{lemma} \label{lem:Hsb_estimate_Q0} Let $b> \frac{1}{2}$. Then for $1<s \leq 2$
    \[ \| Q_0(\tilde{u}, \tilde{v}) \|_{L^2_t \dot{H}^{s-2}_x} \la \| \tilde{u} \|_{H^{1,b}} \| \tilde{v} \|_{H^{s,b}}. \]    
\end{lemma}

\begin{proof} Without loss of generality, assume $\tilde{u} \in X^{1,b}_+$ and $\tilde{v} \in X^{s,b}_-$, the proof being the same for the other signs. First assume $\tilde{u}$ and $\tilde{v}$ are free waves, i.e. $\tilde{u}(t,x) = U(t) \tilde{u}(0,x)$ and $\tilde{v}(t,x) = U(-t) \tilde{v}(0,x)$, where $U(t)$ is the operator defined in \eqref{free_wave_propagator}. Then $\widetilde{\Box} \tilde{u} = 0 = \widetilde{\Box} \tilde{v}$, and in the notation of \cite{FoschiKlainerman2000}, the null form $Q_0(\tilde{u}, \tilde{v})$ corresponds to $D_+ D_- (\tilde{u} \tilde{v})$, where $D_+$ is the operator with symbol $|\tau| + |\xi|$ and $D_-$ is the operator with symbol $||\tau|-|\xi||$. Applying Corollary 13.3 of \cite{FoschiKlainerman2000} with $\beta_+ = 0 = \beta_-$, $\beta_0 = s - 2$, $\alpha_1 = 1$, and $\alpha_2 = s$ gives the estimate for free waves
\begin{equation} \label{Q0_estimate_free_waves} \| Q_0(\tilde{u}, \tilde{v}) \|_{L^2_t\dot{H}^{s-2}_x} \la \| \tilde{u}(0,\cdot) \|_{H^1} \| \tilde{v}(0,\cdot) \|_{H^s} . \end{equation}
To extend this to $\tilde{u} \in X^{1,b}_+$ and $\tilde{v} \in X^{s,b}_-$, note that by the characterisation \eqref{propagator_characterisation_Bourgain_spaces}, there exist $\tilde{f} \in H^b_t H^1_x$, $\tilde{g} \in H^b_t H^s_x$ such that $\tilde{u} = U(t) \tilde{f}$ and $\tilde{v} = U(-t) \tilde{g}$. Using the embedding $H^b(\mathbb{R}) \subset L^\infty(\mathbb{R})$, we have
\begin{align*}
    \| Q_0 (\tilde{u}, \tilde{v}) \|_{L^2_t \dot{H}^{s-2}_x} & \la \| Q_0 (U(t)\tilde{f}(t,x), U(-t)\tilde{g}(t,x) ) \|_{L^2_t \dot{H}^{s-2}_x} \\
    & \la \| Q_0 ( U(t) \tilde{f}(t',x), U(-t) \tilde{g}(t',x) ) \|_{L^\infty_{t'} L^2_t \dot{H}^{s-2}_x} \\
    & \la \| Q_0 ( U(t) \tilde{f}(t',x), U(-t) \tilde{g}(t',x) ) \|_{H^b_{t'} L^2_t \dot{H}^{s-2}_x} \\
 \eqref{Q0_estimate_free_waves} \to  & \la \left\| \| \tilde{f}(t',\cdot) \|_{H^{1}} \| \tilde{g}(t',\cdot) \|_{H^s} \right\|_{H^b_{t'}} \\
    & \la \| U(-t') \tilde{u}(t',x) \|_{H^b_{t'} H^1_x} \| U(t') \tilde{v}(t',x) \|_{H^b_{t'} H^s_x} \\
    & \la \| \tilde{u} \|_{X^{1,b}_+} \| \tilde{v} \|_{X^{s,b}_-},
\end{align*}
where in the penultimate line we used the fact that $H^b(\mathbb{R})$ is a Banach algebra for $b> \frac{1}{2}$, i.e. $\| fg \|_{H^b} \la \|f \|_{H^b} \|g \|_{H^b}$.
\end{proof}

\begin{lemma} \label{lem:Hsb_estimate_Q0i} Let $b > \frac{1}{2}$. Then

    \[ \| Q_{0i}(\tilde{u},\tilde{v}) \|_{\dot{H}^{-1}_{t,x}} \la \| \tilde{u} \|_{H^{1,b}} \| \tilde{v} \|_{H^{1,b}}. \]

\end{lemma}

\begin{proof} The estimate
\[ \| Q_{0i} (\tilde{u}, \tilde{v}) \|_{\dot{H}^{-1}_{t,x}} \la \| \tilde{u}(0,\cdot)\|_{H^{1}} \|\tilde{v}(0,\cdot)\|_{H^{1}} \]
for free waves $\tilde{u}(t,x) = U(t) \tilde{u}(0,x)$, $\tilde{v}(t,x) = U(-t)\tilde{v}(0,x)$ follows from Corollary 13.5 of \cite{FoschiKlainerman2000} with $\beta_+ = -1$, $\beta_- = 0$, $\alpha_1 = 1$, and $\alpha_2 = 1$. We then extend this estimate to $X^{s,b}_\pm$ in the same way as in the proof of \Cref{lem:Hsb_estimate_Q0}.
\end{proof}

\begin{lemma} \label{lem:regularity_of_nonlinearities} For the local solution $\tilde{A}_a$, $\tilde{\phi}$ of \eqref{Minkowski_Maxwell_equation_potential}--\eqref{Minkowski_scalar_equation_potential} given by \Cref{Selberg_Tesfahun_theorem}, we have for some $T>0$
\[ \mathcalboondox{M}(\tilde{A}_a, \tilde{\phi}) \in H^{-\delta}(S_T), \]
\[ \boldsymbol{\mathcalboondox{Q}}(\tilde{A}_a, \tilde{\phi)}, ~ \boldsymbol{\mathcalboondox{P}}(\tilde{A}_a, \tilde{\phi}) \in H^{-1}(S_T). \]
\end{lemma}

\begin{proof} We first treat $\boldsymbol{\mathcalboondox{Q}}(\tilde{A}_a, \tilde{\phi})$. Noting that $\tilde{\phi} \in H^{1,\frac{1}{2}+\epsilon}(S_T)$, estimate \eqref{Masmoudi_Nakanishi_estimate} with $s_1=s_2=1$ and $b=\frac{1}{2}+\epsilon$ implies
\[ \| |\tilde{\boldsymbol{\nabla}}|^{-1} Q_{ij}(\tilde{\phi},\bar{\tilde{\phi}}) \|_{L^2(S_T)} \la \| \tilde{\phi} \|^2_{H^{1,\frac{1}{2}+\epsilon}(S_T)}, \]
i.e. $Q_{ij}(\tilde{\phi},\bar{\tilde{\phi}}) \in L^2_t\dot{H}^{-1}_x(S_T) \subset L^2_t H^{-1}_x(S_T)$. Using \eqref{null_structure_Q}, this immediately gives $\boldsymbol{\mathcalboondox{Q}}(\tilde{A}_a, \tilde{\phi}) \in \dot{H}^{-1}(S_T) \subset H^{-1}(S_T)$. To treat $\mathcalboondox{M}(\tilde{A}_a, \tilde{\phi})$, we note that the $Q_{ij}(\tilde{\phi}, \tilde{w}_k)$ term in \eqref{null_structure_M} is in $L^2(S_T)$ using estimate \eqref{Masmoudi_Nakanishi_estimate} with $b=\frac{1}{2} + \epsilon$, $s_1 =1$, and $s_2 = 2$. The final term $Q_0(\tilde{v}, \tilde{\phi})$, on the other hand, is not in $L^2(S_T)$, but by \Cref{lem:Hsb_estimate_Q0} with $s=2-\delta$ and $b=\frac{1}{2}+\epsilon$ is in $L^2_t \dot{H}^{-\delta}_x(S_T) \subset H^{-\delta}(S_T)$. Finally, for $\boldsymbol{\mathcalboondox{P}}(\tilde{A}_a, \tilde{\phi})$ \Cref{lem:Hsb_estimate_Q0i} and \eqref{null_structure_P} give $\boldsymbol{\mathcalboondox{P}}(\tilde{A}_a, \tilde{\phi}) \in H^{-1}(S_T)$.
\end{proof}

\begin{remark}
    Note that, as a consequence of the arbitrarily small loss of regularity $\delta$ in the potential in \Cref{Selberg_Tesfahun_theorem}, we can only ensure that the nonlinearity $\mathcalboondox{M}(\tilde{A}_a, \tilde{\phi})$ is in $H^{-\delta}(S_T)$ rather than $L^2(S_T)$. This is the source of the small loss of regularity in $\phi$ in \Cref{thm:main_theorem}.
\end{remark}

\section{Localization to Minkowski Space}
\label{sec:localization_to_Minkowski}

\subsection{Induced data and their modification} \label{sec:induced_data_and_modification}

We now turn to our construction for obtaining localized finite energy initial data on Minkowski space (in the sense of \Cref{defn:finite_energy_initial_data_Minkowski}) from global finite energy initial data (in the sense of \Cref{defn_finite_energy_data_on_cylinder}) on the Einstein cylinder.

Let us first describe how a set of finite energy initial data on the Einstein cylinder induces a conformally related set of initial data (which may not be finite energy) on Minkowski space. Consider initial data $(\mathbf{E}_0, \mathbf{B}_0, U, \phi_0)$ for the Maxwell-scalar field system on $\mathbb{R} \times \mathbb{S}^3$ in the sense of \Cref{defn_finite_energy_data_on_cylinder}, and write 
\[ r_+^2 = \frac{1}{2} (1+r^2) . \] 
The conformal embedding $\iota_\Omega : \mathbb{M} \hookrightarrow \mathbb{R} \times \mathbb{S}^3$ described in \Cref{sec:conformal_embedding} reduces to a compactification of $\mathbb{R}^3$ via stereographic projection on $\{t=0\} \iff \{\tau =0\}$. The data $(\mathbf{E}_0, \mathbf{B}_0, U, \phi_0)$ on $\mathbb{S}^3$ therefore induces conformally related data on the standard initial hypersurface $\mathbb{R}^3$ in Minkowski space as follows. 

\subsubsection{Electromagnetic data}

Under the conformal embedding $\iota_\Omega$, on $\{t=0\} \iff \{\tau = 0\}$ the measure on $\mathbb{S}^3$ becomes related to the measure on $\mathbb{R}^3$ by $\dvol_{\mathbb{S}^3} = r_+^{-6} \dvol_{\mathbb{R}^3}$, and the initial electric and magnetic fields on $\mathbb{M}$ and $\mathbb{R} \times \mathbb{S}^3$ satisfy
\[ \mathbf{E}_0 = r^2_+ \tilde{\mathbf{E}}_0 \qquad \text{and} \qquad \mathbf{B}_0 = r^{2}_+ \tilde{\mathbf{B}}_0. \]
Indeed, the vector fields $\partial_\tau$ on $\mathbb{R}\times \mathbb{S}^3$ and $\partial_t$ on $\mathbb{M}$ are parallel on $\{ t=0 \} \iff \{ \tau = 0 \}$,
\[ r^2_+ \left. \frac{\partial}{\partial t} \right|_{t=0} = \left. \frac{\partial}{\partial \tau} \right|_{\tau=0},
\]
and $\Omega^{-1}|_{t=0} = r^2_+$. This puts $\tilde{\mathbf{E}}_0$ and $\tilde{\mathbf{B}}_0$ in weighted $L^2$ spaces on $\mathbb{R}^3$,
\[ \int_{\mathbb{S}^3} |\mathbf{E}_0|^2 \dvol_{\mathbb{S}^3} = \int_{\mathbb{R}^3} r^2_+ |\tilde{\mathbf{E}}_0|^2 \dvol_{\mathbb{R}^3}, \qquad \int_{\mathbb{S}^3} |\mathbf{B}_0|^2 \dvol_{\mathbb{S}^3} = \int_{\mathbb{R}^3} r^2_+ |\tilde{\mathbf{B}}_0|^2 \dvol_{\mathbb{R}^3}, \]
\emph{where the inner product on the left-hand side is taken with respect to the standard round metric $h$ on $\mathbb{S}^3$, and on the right-hand side with respect to the Euclidean metric $\delta$ on $\mathbb{R}^3$, $h^{ij} = r^4_+ \delta^{ij}$}. In particular, the resulting Minkowskian Maxwell data is $\tilde{\mathbf{E}}_0$, $\tilde{\mathbf{B}}_0 \in L^2(\mathbb{R}^3)$.

\begin{remark}
\label{rmk:electric_charge_zero}

Of course, in fact one has more decay, $\tilde{\mathbf{E}}_0, \, \tilde{\mathbf{B}}_0 \in r_+^{-1} L^2(\mathbb{R}^3)$, where the additional factor of $r$ imposes the usual effect of conformally induced Maxwell data being necessarily chargeless at spatial infinity. Indeed, if $\tilde{\mathbf{E}}_0 \in r^{-1}_+ L^2(\mathbb{R}^3)$, then
\[ \int_{\mathbb{R}^3} r^2_+ | \tilde{\mathbf{E}}_0 |^2 \dvol_{\mathbb{R}^3} = \int_0^\infty \int_{\mathbb{S}^2} r^2_+ r^2 |\tilde{\mathbf{E}}_0|^2 \dvol_{\mathbb{S}^2} \d r < \infty, \]
so in particular (assuming\footnote{This limit may fail to exist up to, for instance, bumps of shrinking support going out to infinity. What is, instead, true, is that if $f \in L^1([0,\infty)_r)$, then for every $\epsilon > 0$ there exists a measurable set $A_\epsilon$ with $|A_\epsilon|< \epsilon$ such that $\lim_{r \to \infty, \, r\notin A_\epsilon} f(r) = 0$.} $\lim_{r \to \infty} \int_{\mathbb{S}^2} r^4 |\tilde{\mathbf{E}}_0|^2 \dvol_{\mathbb{S}^2}$ exists)
\[ \int_{\mathbb{S}^2} r^2 | \langle\tilde{\mathbf{E}}_0, \mathbf{e}_r\rangle| \dvol_{\mathbb{S}^2} \leq |\mathbb{S}^2|^{1/2} \left( \int_{\mathbb{S}^2} r^4 |\tilde{\mathbf{E}}_0|^2 \dvol_{\mathbb{S}^2} \right)^{1/2} \to 0 \quad \text{as} \quad r \to \infty. \]
By approaching $\tilde{\mathbf{E}}_0$ in $r_+^{-1} L^2(\mathbb{R}^3)$ by approximants belonging to $\mathscr{S}(\mathbb{R}^3)$, we obtain that the electric charge at spatial infinity is therefore zero:
\begin{equation} \label{electric_charge_zero} \tilde{\mathfrak{q}} \defeq \int_{\mathbb{R}^3} \tilde{\boldsymbol{\nabla}} \cdot \tilde{\mathbf{E}}_0 \dvol_{\mathbb{R}^3} = \lim_{r \to \infty} \int_{\mathbb{S}^2} r^2 \langle \tilde{\mathbf{E}}_0, \mathbf{e}_r \rangle \dvol_{\mathbb{S}^2} = 0. 
\end{equation}
\end{remark}

\subsubsection{Scalar field data}

To see the induced data for the scalar field, we recall that by the first part of \Cref{lem:equivalence_of_Minkowskian_initial_data_Lorenz_gauge_and_temporal_initially} we may additionally impose
\begin{equation} \label{residual_gauge_fixing}
\tilde{a}_0 = 0 = \dot{\tilde{a}}_0,
\end{equation}
where $\tilde{a}_0 = \tilde{A}_0|_{t=0}$, $\dot{\tilde{a}}_0 = \partial_t \tilde{A}_0|_{t=0}$, and $\tilde{A}_0 = (\partial_t)^a \tilde{A}_a$. Then, because $\partial_t \Omega|_{t=0} = 0$, the component $U_0$ of the initial data becomes
\[ U_0 =r^4_+ \partial_t \tilde{\phi}|_{t=0} = r^4_+ \tilde{\phi}_1. \]
Therefore the condition $U_0 \in L^2(\mathbb{S}^3)$ is
\begin{equation}
    \label{initial_norm_phi_1}
    \int_{\mathbb{S}^3} |U_0|^2 \dvol_{\mathbb{S}^3} = \int_{\mathbb{R}^3} r^2_+ |\tilde{\phi}_1| \dvol_{\mathbb{R}^3} < \infty,
\end{equation} 
which similarly implies $\tilde{\phi}_1 \in L^2(\mathbb{R}^3)$ (in fact $\tilde{\phi}_1 \in r^{-1}_+ L^2(\mathbb{R}^3)$). For $\phi_0 = \phi|_{\tau = 0}$, however, the rescaling is different, $\phi_0 = r^2_+ \tilde{\phi}_0$, and we have
\begin{equation} 
\label{initial_norm_phi_0}
\int_{\mathbb{S}^3} | \phi_0 |^2 \dvol_{\mathbb{S}^3} = \int_{\mathbb{R}^3} r^{-2}_+ |\tilde{\phi}_0|^2 \dvol_{\mathbb{R}^3} < \infty.
\end{equation}
 The remaining components of $U$ transform as
\begin{align*} \mathbf{U} &= \boldsymbol{\nabla} \phi_0 + i \mathbf{a} \phi_0 = \tilde{\boldsymbol{\nabla}} (r^2_+ \tilde{\phi}_0 ) + i \tilde{\mathbf{a}} r^2_+ \tilde{\phi}_0 = r^2_+ \tilde{\mathbf{U}} + \tilde{\phi}_0 r \mathbf{e}_r,
\end{align*}
so that
\begin{align*}
    \|\tilde{\mathbf{U}} \|^2_{L^2(\mathbb{R}^3)} &= \int_{\mathbb{R}^3} \langle \tilde{\mathbf{U}}, \tilde{\mathbf{U}} \rangle_\delta \dvol_{\mathbb{R}^3} \\
    &= \int_{\mathbb{R}^3} r^{-4}_+ |\mathbf{U} - r \tilde{\phi}_0 \mathbf{e}_r |^2_{\delta} \dvol_{\mathbb{R}^3} \\
    & \la \int_{\mathbb{R}^3} r^{-4}_+ |\mathbf{U} |^2_{\delta} \dvol_{\mathbb{R}^3} + \int_{\mathbb{R}^3} r^{-2}_+ |\tilde{\phi}_0|^2 \dvol_{\mathbb{R}^3} \\
    & \la \int_{\mathbb{S}^3} r^{-2}_+ |\mathbf{U}|^2_{h} \dvol_{\mathbb{S}^3} + \int_{\mathbb{S}^3} |\phi_0|^2 \dvol_{\mathbb{S}^3} \\
    & \leq \| \mathbf{U} \|^2_{L^2(\mathbb{S}^3)} + \| \phi_0 \|^2_{L^2(\mathbb{S}^3)} < \infty.
\end{align*}

\subsubsection{Modification of the scalar field data near infinity}

We therefore obtain the induced set of data
\begin{equation} 
\label{conformally_induced_data_on_Minkowski}
(\tilde{\mathbf{E}}_0, \tilde{\mathbf{B}}_0, \tilde{\phi}_1, \tilde{\mathbf{U}}, \tilde{\phi}_0) = (r^{-2}_+ \mathbf{E}_0, r^{-2}_+ \mathbf{B}_0, r^{-4}_+ U_0, r^{-2}_+ \mathbf{U} - r^{-4}_+ \phi_0 r \mathbf{e}_r, r^{-2}_+ \phi_0)
\end{equation}
on $\mathbb{M}$, which satisfies the constraints \eqref{Minkowski_constraint_1} and \eqref{Minkowski_constraint_2} (by the conformal invariance of the equations, or a direct calculation). Of these data, we saw above that $\tilde{\mathbf{E}}_0 \in L^2(\mathbb{R}^3)$, $\tilde{\mathbf{B}}_0 \in L^2(\mathbb{R}^3)$, $\tilde{\phi}_1 \in L^2(\mathbb{R}^3)$, and $\tilde{\mathbf{U}} \in L^2(\mathbb{R}^3)$. However, \eqref{initial_norm_phi_0} \emph{does not} imply that $\tilde{\phi}_0 \in L^2(\mathbb{R}^3)$, only that $\tilde{\phi}_0 \in r_+L^2(\mathbb{R}^3)$. Because of this, the induced data \eqref{conformally_induced_data_on_Minkowski} on $\mathbb{R}^3$ does not have finite energy on Minkowski space in the sense of \Cref{defn:finite_energy_initial_data_Minkowski}. Our strategy is to modify this data to return to the finite energy regime while preserving the constraint equations; we do this by a delicate\footnote{As a first idea, one might attempt to simply cut off the data outside a large ball on $\mathbb{R}^3$. This, however, does not preserve the non-local constraints \eqref{Minkowski_constraint_1}, \eqref{Minkowski_constraint_2}.} change of the data outside a large ball, in effect undoing the conformal change of the scalar field at infinity on the initial surface.

\begin{lemma}
    \label{lem:modification_of_induced_Minkowskian_data}
    Given a set of finite energy initial data $(\mathbf{E}_0, \mathbf{B}_0, U, \phi_0)$ on the Einstein cylinder, consider the conformally induced initial data \eqref{conformally_induced_data_on_Minkowski} on $\{0\}_t \times \mathbb{R}^3 \subset \mathbb{M}$ under the embedding $\mathbb{M} \hookrightarrow \R \times \mathbb{S}^3$. There exists a set of modified finite energy data $(\tilde{\mathbf{E}}_0, \tilde{\mathbf{B}}_0, \tilde{\psi}_1, \tilde{\mathbf{V}}, \tilde{\psi}_0)$ on $\{0\}_t \times \mathbb{R}^3$ such that
    \[ (\tilde{\mathbf{E}}_0, \tilde{\mathbf{B}}_0, \tilde{\psi}_1, \tilde{\mathbf{V}}, \tilde{\psi}_0) \equiv (\tilde{\mathbf{E}}_0, \tilde{\mathbf{B}}_0, \tilde{\phi}_1, \tilde{\mathbf{U}}, \tilde{\phi}_0) \quad \text{ on } \quad B_R = \{ x \in \mathbb{R}^3 \, : \, r= \| x \| \leq R  \} \]
    for a given fixed $R \gg 1$. In particular, $(\tilde{\mathbf{E}}_0, \tilde{\mathbf{B}}_0, \tilde{\psi}_1, \tilde{\mathbf{V}}, \tilde{\psi}_0)$ satisfy the constraint equations \eqref{Minkowski_constraint_1}, \eqref{Minkowski_constraint_2} and \eqref{spatial_components_initial_data_gauge_covariant_derivative_scalar_field} on $\mathbb{R}^3$, and have the regularity
    \begin{align*}
        & \tilde{\mathbf{E}}_0, \, \tilde{\mathbf{B}}_0, \, \tilde{\mathbf{V}}, \, \tilde{\psi}_1 \in L^2(\mathbb{R}^3), \\
        & \tilde{\psi}_0 \in H^1(\R^3).
    \end{align*}
\end{lemma}

\begin{proof}
    We only wish to modify the scalar field data while preserving the constraint equation
    \[
        \tilde{\boldsymbol{\nabla}} \cdot \tilde{\mathbf{E}}_0 = -  \operatorname{Im}(\bar{\tilde{\phi}}_0 \tilde{\phi}_1).
    \]
    Given $R \gg 1$, consider a function $\xi$ such that
    \[ \xi \in \mathcal{C}^\infty (\R^3) \, ,~ \xi >0 \, ,~ \xi \equiv 1 \mbox{ for } r \leq R \, ,~ \xi \equiv r_+ \mbox{ for } r \geq 2R \, . \]
    Then put
    \begin{equation} \label{modification_of_initial_scalar_field_data} \tilde{\psi}_0 \defeq \frac{1}{\xi} \tilde{\phi}_0, \, \, \qquad \tilde{\psi}_1 \defeq \xi \tilde{\phi}_1 \, .\end{equation}
    Then, by \eqref{initial_norm_phi_1} and \eqref{initial_norm_phi_0}, $\tilde{\psi}_0$, $\tilde{\psi}_1 \in L^2(\mathbb{R}^3)$, they agree with $\tilde{\phi}_0$ and $\tilde{\phi}_1$ respectively on $B_R$, and, since $\bar{\tilde{\psi}}_0 \tilde{\psi}_1 = \xi^{-1} \xi \bar{\tilde{\phi}}_0 \tilde{\phi}_1 =  \bar{\tilde{\phi}}_0 \tilde{\phi}_1$, satisfy
    \[ \tilde{\boldsymbol{\nabla}} \cdot \tilde{\mathbf{E}}_0 = - \operatorname{Im}(\bar{\tilde{\psi}}_0 \tilde{\psi}_1 ). \]
    The constraint \eqref{Minkowski_constraint_2}, of course, remains unchanged. To satisfy the constraint \eqref{spatial_components_initial_data_gauge_covariant_derivative_scalar_field}, we simply set
    \begin{align*} \tilde{\mathbf{V}} &= \tilde{\boldsymbol{\nabla}} \tilde{\psi}_0 + i \tilde{\mathbf{a}} \tilde{\psi}_0 \\
    & = \tilde{\boldsymbol{\nabla}} ( \xi^{-1} \tilde{\phi}_0 ) + i \tilde{\mathbf{a}} \xi^{-1} \tilde{\phi}_0 \\
    & = \xi^{-1} \tilde{\mathbf{U}} - \xi^{-2} (\tilde{\boldsymbol{\nabla}} \xi)\tilde{\phi}_0,
    \end{align*}
    which of course equals $\tilde{\mathbf{U}}$ in $B_R$, and satisfies the estimate
    \begin{align*}
        \| \tilde{\mathbf{V}} \|^2_{L^2(\mathbb{R}^3)} & \la \| \tilde{\mathbf{U}} \|^2_{L^2(\mathbb{R}^3)} + \int_{\mathbb{R}^3} \xi^{-4} |\tilde{\boldsymbol{\nabla}} \xi |^2 |\tilde{\phi}_0|^2 \dvol_{\mathbb{R}^3} < \infty,
    \end{align*}
    in which the second integral is finite as a consequence of the fact that for $r \geq 2R$, $\xi^{-4} |\tilde{\boldsymbol{\nabla}} \xi|^2 = r^{-4}_+ r^2 \la r^{-2}_+$, and the condition \eqref{initial_norm_phi_0}. Using \Cref{lem:equivalence_of_Minkowskian_initial_data_Lorenz_gauge_and_temporal_initially}, we therefore have that $\tilde{\psi}_0 \in H^1(\mathbb{R}^3)$, and the set
    \begin{equation} \label{IDTilda}
    (\tilde{\mathbf{E}}_0, \tilde{\mathbf{B}}_0, \tilde{\psi}_1, \tilde{\psi}_0) \in L^2(\mathbb{R}^3)\times L^2(\mathbb{R}^3) \times L^2(\mathbb{R}^3) \times L^2(\mathbb{R}^3) \times H^1(\mathbb{R}^3)
    \end{equation}
    defines genuine finite energy data for the Maxwell-scalar field system \eqref{Minkowski_MKG_system_with_projections} (equivalently \eqref{Minkowski_Maxwell_equation_potential}--\eqref{Minkowski_scalar_equation_potential}) on Minkowski space in the sense of \Cref{defn:finite_energy_initial_data_Minkowski}, which agrees with $(\tilde{\mathbf{E}}_0, \tilde{\mathbf{B}}_0, \tilde{\phi}_1, \tilde{\phi}_0)$ in $B_R = \{ r \leq R \}$.
\end{proof}

\subsection{Local solution}
\label{sec:local_solution}

We can therefore apply \Cref{Selberg_Tesfahun_theorem} to the data constructed in \Cref{lem:modification_of_induced_Minkowskian_data} to obtain a unique solution $(\tilde{\psi}', \tilde{\mathbf{E}}' , \tilde{\mathbf{B}}')$ to \eqref{Minkowski_Maxwell_equation_potential}--\eqref{Minkowski_scalar_equation_potential} and a corresponding potential $\tilde{A}'_a$ in Lorenz gauge such that
\begin{align*}
    & \tilde{\psi}' \in \mathcal{C}^0 (\R_t ; H^1(\R^3)) \cap \mathcal{C}^1 (\R_t ; L^2 (\R^3)),\\
    & \tilde{\mathbf{E}}' , \, \tilde{\mathbf{B}}' \in \mathcal{C}^0 (\R_t ; L^2 (\R^3)), \\
    & \tilde{A}'_a \in \mathcal{C}^0(\R_t ; \dot{H}^1(\R^3) + H^{1-\delta}(\R^3) ) \cap \mathcal{C}^1(\R_t ; H^{-\delta}(\R^3)) , \\
    & (\tilde{\psi}', \partial_t \tilde{\psi}')|_{t=0} = (\tilde{\psi}_0, \tilde{\psi}_1), \\
    & (\tilde{\mathbf{E}}', \tilde{\mathbf{B}}')|_{t=0} = (\tilde{\mathbf{E}}_0, \tilde{\mathbf{B}}_0),
\end{align*}
and moreover
    \begin{align}
        \label{wave_Sobolev_regularity_1}
        & \tilde{\psi}' \in H^{1,\frac{1}{2}+\epsilon}(S_T), ~ \partial_t \tilde{\psi}' \in H^{0,\frac{1}{2}+\epsilon}(S_T), \\
        \label{wave_Sobolev_regularity_2}
        & \tilde{\mathbf{E}}', \, \tilde{\mathbf{B}}' \in H^{0,\frac{1}{2}+\epsilon}(S_T), \\
        \label{wave_Sobolev_regularity_3}
        & \tilde{A}'_a \in \dot{H}^{1,\frac{1}{2}+\epsilon}(S_T) + H^{1-\delta,\frac{1}{2}+\epsilon}(S_T).
    \end{align}
Here the primes---and the use of the letter $\tilde{\psi}$ to denote the scalar field---represent the fact that the solution $(\tilde{\psi}', \tilde{A}'_a, \tilde{\mathbf{E}}', \tilde{\mathbf{B}}')$ arises from the modified set of initial data and is not meaningful outside of $D^+(B_R)$, where $D^+(B_R)$ denotes the future domain of dependence of $B_R$. On the other hand, by the manifest finite speed of propagation for the system \eqref{Minkowski_Maxwell_equation_potential}--\eqref{Minkowski_scalar_equation_potential}, we have that in $D^+(B_R)$ the fields $(\tilde{\psi}', \tilde{A}'_a, \tilde{\mathbf{E}}', \tilde{\mathbf{B}}')$ depend only on $(\tilde{\mathbf{E}}_0, \tilde{\mathbf{B}}_0, \tilde{\phi}_1, \tilde{\phi}_0)|_{B_R}$.

\section{Local Change of Foliation}
\label{sec:change_of_foliation}

\subsection{Change of foliation for linear wave equations on manifolds}

Consider a smooth globally hyperbolic spacetime $(\mathcalboondox{M}, g)$, $\mathcalboondox{M} = \mathbb{R} \times S$, where $S$ is a smooth $n$-dimensional compact manifold without boundary and $g$ a smooth Lorentzian metric on $\mathcalboondox{M}$, and denote by $\nabla$ the Levi-Civita connection associated to $g$. Suppose that $\tau$ is a smooth time function on $\mathcalboondox{M}$ whose values span $\mathbb{R}$, and $\{ S_\tau \}_\tau$ the associated uniformly spacelike foliation of $\mathcalboondox{M}$ by the level hypersurfaces of $\tau$, where for each $\tau$ the hypersurface $S_\tau$ is diffeomorphic to $S$. Hence for each $\tau$, $S_\tau \simeq S$ is a smooth compact spacelike Cauchy hypersurface in $\mathcalboondox{M}$. In particular, on $S_\tau$ we can define all Sobolev spaces (cf. \Cref{sec:Sobolev_spaces}), all canonically isomorphic to the corresponding Sobolev spaces on $S$. Define the future-oriented timelike vector field $\partial_\tau$ by
\[ \partial_\tau \defeq g ( \nabla \tau, \nabla \tau)^{-1} \nabla \tau. \]
We then have the following lemmas, whose proofs are postponed until the appendix.

\begin{lemma} \label{lem:foliation_change_positive_s}
    Let $s \in [0, \infty)$ and $\{ S_\tau \}_\tau$ be any smooth uniformly spacelike foliation of $\mathcalboondox{M}$, and suppose $f \in H^s_{\text{loc}}(\mathcalboondox{M})$. Then, given $(u_0, u_1) \in H^{s+1}(S_0) \times H^s(S_0)$, if the Cauchy problem on $\mathcalboondox{M}$
    \begin{align*}
        \begin{cases}
            \Box_g u = f, \\
            (u, \partial_\tau u)|_{S_0} = (u_0, u_1)
        \end{cases}
    \end{align*}
    has a solution $u$, then this solution has the regularity
        \[ u \in \mathcal{C}^0 (\R_\tau ;H^{s+1} (S_\tau)) \cap \mathcal{C}^1 (\R_\tau ;H^s(S_\tau)) \, . \]
\end{lemma}

\begin{lemma} \label{lem:foliation_change_negative_s}
    Let $s \in [-1, 0)$ and $\{S_\tau \}$ be any smooth uniformly spacelike foliation of $\mathcalboondox{M}$, and suppose $f \in H^s_{\text{loc}}(\mathcalboondox{M})$. Then, given $(u_0, u_1) \in H^{s+1}(S_0) \times H^s(S_0)$, if the Cauchy problem on $\mathcalboondox{M}$
    \begin{align*}
        \begin{cases}
            \Box_g u = f, \\
            (u, \partial_\tau u)|_{S_0} = (u_0, u_1)
        \end{cases}
    \end{align*}
    has a solution $u$, then this solution has the regularity
    \[ u \in \mathcal{C}^0(\mathbb{R}_\tau; H^{s+1}(S_\tau)). \]
\end{lemma}

\subsection{Change of foliation of the local solution}

Consider the solution  $(\tilde{\psi}', \tilde{A}'_a, \tilde{\mathbf{E}}', \tilde{\mathbf{B}}')$ on $\mathbb{M}$ obtained in \Cref{sec:local_solution} and let
\[ \tilde{\Sigma}_\tau = \{ \tau = \text{const.} \} \cap D^+(B_R),  \]
where $\tau$ is given\footnote{In Minkowskian coordinates, $\partial_\tau = \frac{1}{2}(1+t^2+r^2)\partial_t + rt \partial_r$.} by \eqref{tau_definition_in_physical_coordinates}. Extend $\tilde{\Sigma}_\tau$ to a foliation of $\mathbb{M}$, still denoting it $\tilde{\Sigma}_\tau$, so that $\tilde{\Sigma}_\tau$ coincides with $\tilde{\Sigma}_t$ outside $D^+(B_{R'})$ for some $R' > R$, and interpolates smoothly as a uniformly spacelike foliation in the region $D^+(B_{R'}) \setminus D^+(B_R)$. We now denote the solution in $D^+(B_R)$ and to the past of $\tilde{\Sigma}_{\tau_*}$, for some $\tau_* > 0$ to be chosen, by removing the primes and returning to the notation $\tilde{\phi}$ for the scalar field, i.e.
\[ (\tilde{\psi}', \tilde{A}'_a, \tilde{\mathbf{E}}', \tilde{\mathbf{B}}')|_{D^+(B_R) \cap J^-(\tilde{\Sigma}_{\tau_*})} = (\tilde{\phi}, \tilde{A}_a, \tilde{\mathbf{E}}, \tilde{\mathbf{B}}), \]
and discard the part of the primed solution outside $D^+(B_R)$. In what follows we now write
\[ \tilde{\mathcal{R}}_{R,\tau_*} \defeq D^+(B_R) \cap J^-(\tilde{\Sigma}_{\tau_*}) \]
and prove the following regularity of the local solution with respect to the foliation $\tilde{\Sigma}_\tau$.

\begin{figure}[h]
    \centering
    \begin{tikzpicture}
        \node[inner sep=0pt] (curved) at (0,0)
        {\includegraphics[width=.421\textwidth]{{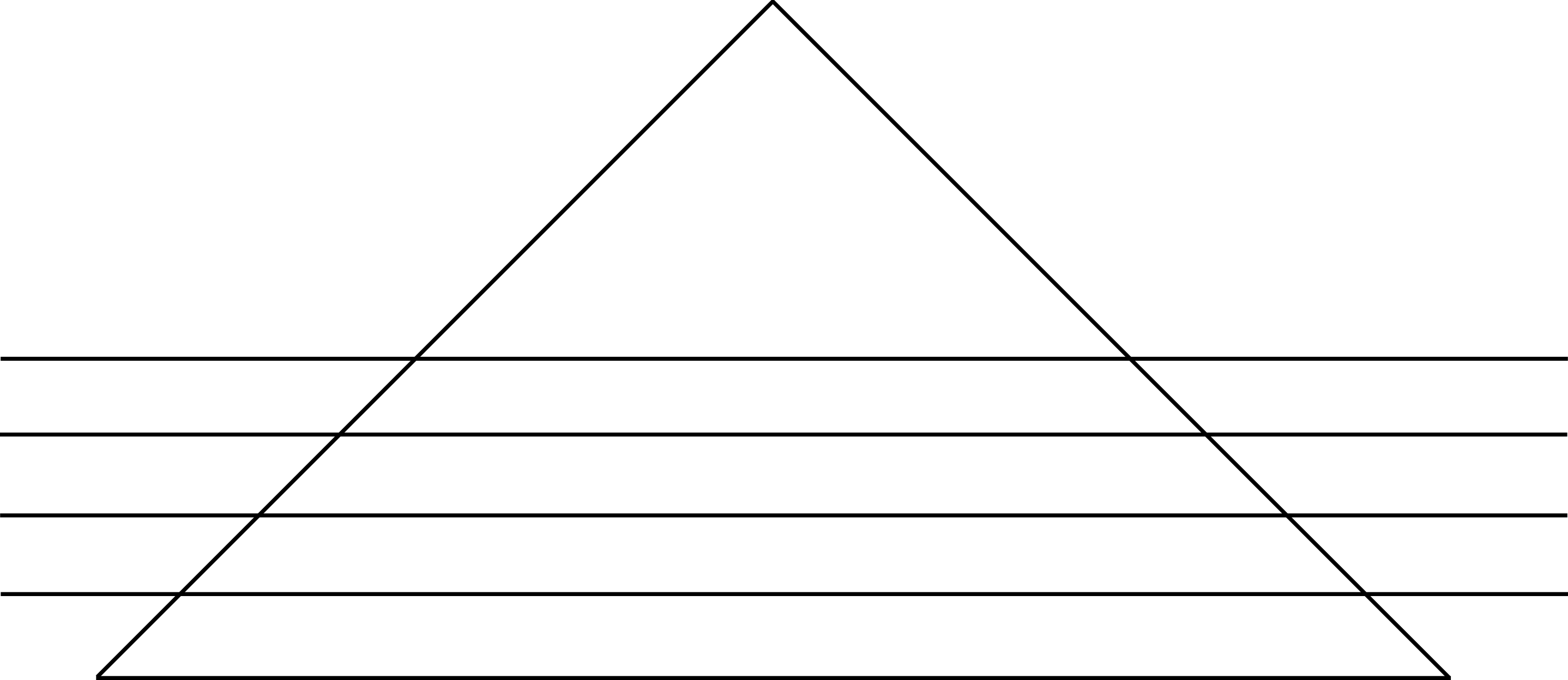}}};
        \node[inner sep=0pt] (straight) at (7.5,0)
        {\includegraphics[width=.485\textwidth]{{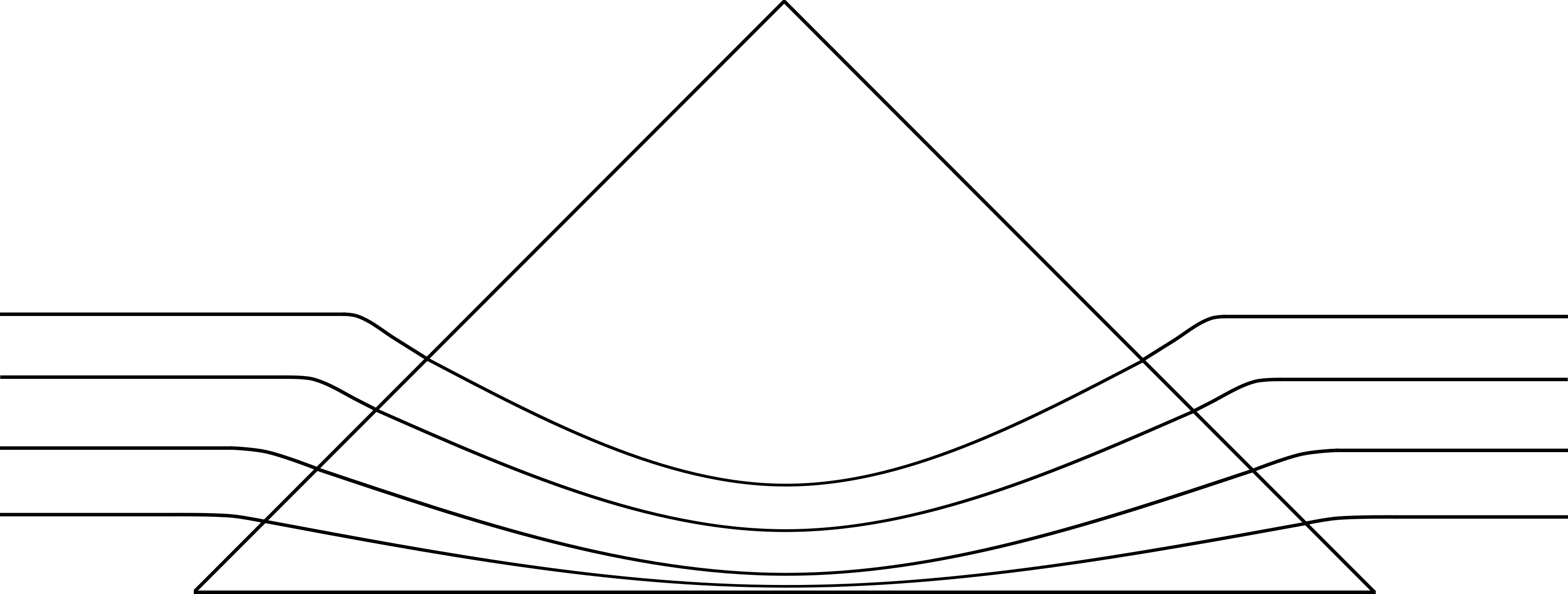}}};
        \draw[->] (2.3,0.5) .. controls (3,0.9) and (4,0.9) .. (4.7,0.5);
        \node[label={[shift={(0,-2.1)}]$B_R$}] {};
        \node[label={[shift={(7.5,-2.1)}]$B_R$}] {};
        \node[label={[shift={(0,-0.1)}]$D^+(B_R)$}] {};
        \node[label={[shift={(7.5,-0.1)}]$D^+(B_R)$}] {};
        \node[label={[shift={(-2,-0.2)}]$\tilde{\Sigma}_t$}] {};
        \node[label={[shift={(7.5,-1)}]$\tilde{\Sigma}_\tau$}] {};
        \node[label={[shift={(3.5,0.7)}]foliation change}] {};
    \end{tikzpicture}
    \caption{Change of local foliation of the solution on Minkowski space.}
    \label{fig:change_of_foliation_Minkowski}
\end{figure}
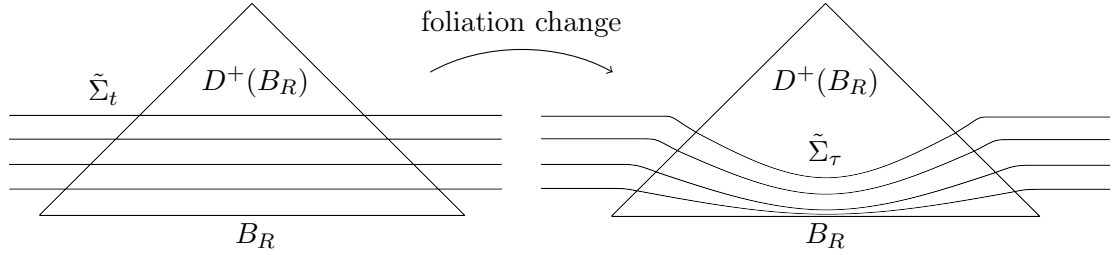

\begin{lemma}
\label{lem:regularity_with_respect_to_hyperboloidal_foliation}
    There exists a $\tau_*>0$ such that in $\tilde{\mathcal{R}}_{R,\tau_*}$ the solution \[(\tilde{\psi}', \tilde{A}'_a, \tilde{\mathbf{E}}', \tilde{\mathbf{B}}')|_{\tilde{\mathcal{R}}_{R,\tau_*}} \eqdef (\tilde{\phi}, \tilde{A}_a, \tilde{\mathbf{E}}, \tilde{\mathbf{B}})\] has the regularity
    \begin{align*}
    & \tilde{\phi} \in \mathcal{C}^0 ([0,\tau_*] ; H^{1-\delta}(\tilde{\Sigma}_\tau)), \\
    & \tilde{\mathbf{E}} , \, \tilde{\mathbf{B}} \in \mathcal{C}^0 ([0,\tau_*] ; L^2 (\tilde{\Sigma}_\tau)), \\
    & \tilde{A}_a \in \mathcal{C}^0([0,\tau_*];H^{\frac{1}{2}-\delta}(\tilde{\Sigma}_\tau)),
    \end{align*}
    and
    \begin{align*} \tilde{A}_a, \, \tilde{\phi} &\in H^{1-\delta}(\tilde{\mathcal{R}}_{R,\tau_*}).
    \end{align*}
\end{lemma}

\begin{proof}
We first treat the spacetime regularity of the potential and the scalar field. We split
    \[ \tilde{A}'_a = \tilde{A}^{(0)}_a + \tilde{A}^{\text{inh.}}_a, \]
    where in $\tilde{\mathcal{R}}_{R,\tau_*}$
    \[ \tilde{A}^{(0)}_a \in \mathcal{C}^0_t \dot{H}^1_x \cap \mathcal{C}^1_t H^{-\delta}_x, \quad \tilde{A}^{\text{inh.}}_a \in \mathcal{C}^0_t H^{1-\delta}_x \cap \mathcal{C}^1_t H^{-\delta}_x \]
    with respect to the standard foliation $\tilde{\Sigma}_t$. In particular, by the boundedness of $\tilde{\mathcal{R}}_{R,\tau_*}$
    and the embedding \eqref{Sobolev_embedding_homogeneous_H1}, we have $\| \tilde{A}_a^{(0)} \|_{L^2_x} \la \| \tilde{A}_a^{(0)} \|_{L^6_x} \la \|\tilde{A}_a^{(0)}\|_{\dot{H}^1_x}$, so
    \[ \tilde{A}^{(0)}_a \in L^2_t H^1_x, \quad \partial_t \tilde{A}^{(0)}_a \in L^2_t H^{-\delta}_x, \quad \tilde{A}^{\text{inh.}}_a \in L^2_t H^{1-\delta}_x, \quad \partial_t \tilde{A}^{\text{inh.}}_a \in L^2_t H^{-\delta}_x, \]
    and hence $\tilde{A}_a \in L^2_t H^{1-\delta}_x$ and $\partial_t \tilde{A}_a \in L^2_t H^{-\delta}_x$. It follows immediately from the Fourier space definitions of (Riemannian, spacetime) $H^s$ spaces, and of the definition of the corresponding restriction spaces to $\tilde{\mathcal{R}}_{R,\tau_*}$ \eqref{definition_fractional_Sobolev_spaces_on_domain}, that
    \begin{equation} \label{local_spacetime_regularity_potential} \tilde{A}_a \in H^{1-\delta}(\tilde{\mathcal{R}}_{R,\tau_*}),
    \end{equation}
    since
    \begin{equation} \label{simple_Fourier_argument} (1+|\xi_0|^2 + |\xi|^2)^{1-\delta} = \langle \xi \rangle^{2(1-\delta)} \left( 1+ \frac{|\xi_0|^2}{\langle \xi \rangle^2} \right)^{1-\delta} \leq \langle \xi \rangle^{2(1-\delta)} + |\xi_0|^2 \langle \xi \rangle^{-2\delta}, \end{equation}
    as $\delta \in (0,1)$. The spacetime regularity of $\tilde{\phi}$ follows in the same way.

    Now for the regularity with respect to the foliation $\tilde{\Sigma}_\tau$, recall that by \Cref{Selberg_Tesfahun_theorem}, the local solution in $S_T$ has the wave-Sobolev regularity stated in \eqref{wave_Sobolev_regularity_1}--\eqref{wave_Sobolev_regularity_3}. Let $\tau_*$ be sufficiently small such that $\sup_{D^+(B_R)\cap \{ \tau = \tau_*\} } t < T$. By \Cref{lem:regularity_of_nonlinearities}, 
    \begin{align*} \widetilde{\Box} \tilde{\mathbf{E}}' =\boldsymbol{\mathcalboondox{P}}(\tilde{A}'_a, \tilde{\psi}'), ~ \widetilde{\Box} \tilde{\mathbf{B}}' = \boldsymbol{\mathcalboondox{Q}}(\tilde{A}'_a, \tilde{\psi}') & \in H^{-1}(S_T) \subset H^{-1}(\tilde{\mathcal{R}}_{R,\tau_*}), \\
    \widetilde{\Box} \tilde{\psi}' = \mathcalboondox{M}(\tilde{A}'_a, \tilde{\psi}') &\in H^{-\delta}(S_T) \subset H^{-\delta}(\tilde{\mathcal{R}}_{R,\tau_*}),
    \end{align*}
    so applying\footnote{Note that in order to apply \Cref{lem:foliation_change_positive_s} and \Cref{lem:foliation_change_negative_s}, we need to extend all relevant local Sobolev functions to Sobolev functions on a spatially compact spacetime, e.g. $\mathbb{R}\times \mathbb{S}^3$. This is easy using the definition \eqref{Sobolev_spaces_on_manifolds}.} \Cref{lem:foliation_change_negative_s} we deduce that $\tilde{\mathbf{E}}, \, \tilde{\mathbf{B}} \in \mathcal{C}^0([0,\tau_*];L^2(\tilde{\Sigma}_\tau))$ and $\tilde{\phi} \in \mathcal{C}^0([0,\tau_*];H^{1-\delta}(\tilde{\Sigma}_\tau))$. For the potential, since $\tilde{A}'_a \in \mathcal{C}^0_t(\dot{H}^1_x + H^{1-\delta}_x) \cap \mathcal{C}^1_t(H^{-\delta}_x)$, again by the boundedness of $\tilde{\mathcal{R}}_{R,\tau_*}$ and the embedding \eqref{Sobolev_embedding_homogeneous_H1}, we have $\tilde{A}_a \in \mathcal{C}^0_t H^{1-\delta}_x \cap \mathcal{C}^1_t H^{-\delta}_x$ in $\tilde{\mathcal{R}}_{R,\tau_*}$, so \Cref{lem:regularity_of_N} and \Cref{lem:foliation_change_negative_s} similarly give $\tilde{A}_a \in \mathcal{C}^0([0,\tau_*];H^{\frac{1}{3}}(\tilde{\Sigma}_\tau))$. But now the spatial regularity of $\tilde{A}_a$ can be improved using \eqref{local_spacetime_regularity_potential} and the trace theorem \eqref{trace_operator}: using the continuity of the operator $\operatorname{tr}: H^s(\tilde{\mathcal{R}}_{R,\tau_*}) \to H^{s-\frac{1}{2}}(\tilde{\Sigma}_\tau)$, $\tilde{A}_a \mapsto \tilde{A}_a|_{\tilde{\Sigma}_\tau}$ (for $s > \frac{1}{2}$), we have 
\[ \sup_\tau \| \tilde{A}_a|_{\tilde{\Sigma}_\tau} \|_{H^{\frac{1}{2}-\delta}(\tilde{\Sigma}_\tau)} \la \| \tilde{A}_a \|_{H^{1-\delta}(\mathcal{R}_{R,\tau_*})}, \]
and by approximating $\tilde{A}_a$ by smooth functions, we obtain $\tilde{A}_a \in \mathcal{C}^0([0,\tau_*];H^{\frac{1}{2}-\delta}(\tilde{\Sigma}_\tau))$.
\end{proof}

\section{Solution on the Einstein Cylinder}

\subsection{Conformal transport of local solution and improved regularity}
\label{sec:conformal_transport_of_solution}
By conformality, under $\iota_\Omega$ (cf. \Cref{sec:conformal_embedding}) the domain $D^+(B_R) \subset \mathbb{M}$ maps to the future domain of dependence of the image of $B_R$, which we still denote by $D^+(B_R)$. By construction of the foliation $\tilde{\Sigma}_\tau$, in $D^+(B_R)$ the leaves $\tilde{\Sigma}_\tau$ are mapped to the leaves $\Sigma_\tau$, as shown in \Cref{fig:change_of_foliation_cylinder}, where the leaves $\Sigma_\tau$ are now pieces of the standard foliation of $\mathbb{R} \times \mathbb{S}^3$ by 3-spheres contained in $D^+(B_R)$. Further, $J^-(\tilde{\Sigma}_{\tau_*}) = J^-(\Sigma_{\tau_*})$, and hence the region $\tilde{\mathcal{R}}_{R,\tau_*} = D^+(B_R) \cap J^-(\tilde{\Sigma}_{\tau_*}) \subset \mathbb{M}$ is mapped to the region  
\[ \mathcal{R}_{R,\tau_*} \defeq D^+(B_R) \cap J^-(\Sigma_{\tau_*}) \subset \mathbb{R} \times \mathbb{S}^3,\] 
as shown in \Cref{fig:image_of_partial_slab}.

    \begin{figure}[h]
    \hspace{-1.8cm}
    \begin{tikzpicture}
        \node[inner sep=0pt] (curved) at (0,0)
        {\includegraphics[width=.25\textwidth]{{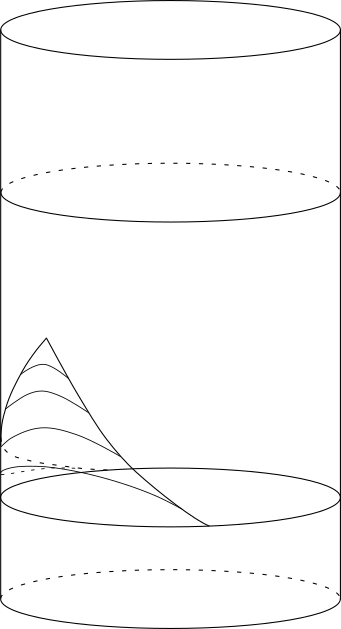}}};
        \node[inner sep=0pt] (straight) at (8,0)
        {\includegraphics[width=.25\textwidth]{{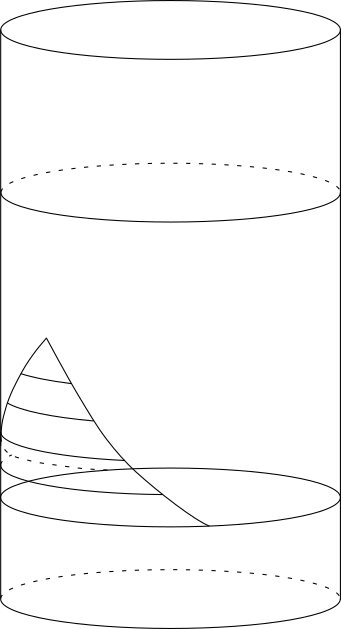}}};
        \draw[->] (2.8,0) .. controls (3.5,0.4) and (4.5,0.4) .. (5.2,0);
        \node[label={[shift={(4,0.2)}]foliation change}] {};
        \node[label={[shift={(-2.6,-1.65)}]$\Sigma_t$}] {};
        \node[label={[shift={(-2,-2.5)}]$\left\{\begin{array}{c} \\ \\ \\  \\\end{array}\right.$}] {};
        \node[label={[shift={(5.4,-1.65)}]$\Sigma_\tau$}] {};
        \node[label={[shift={(6,-2.5)}]$\left\{\begin{array}{c} \\ \\ \\ \\ \end{array}\right.$}] {};
    \end{tikzpicture}
    \caption{By construction, the new foliation agrees with the standard foliation of $\mathbb{R} \times \mathbb{S}^3$ in $D^+(B_R)$ after the conformal embedding of $\mathbb{M}$ into $\mathbb{R} \times \mathbb{S}^3$.}
    \label{fig:change_of_foliation_cylinder}
    \end{figure}
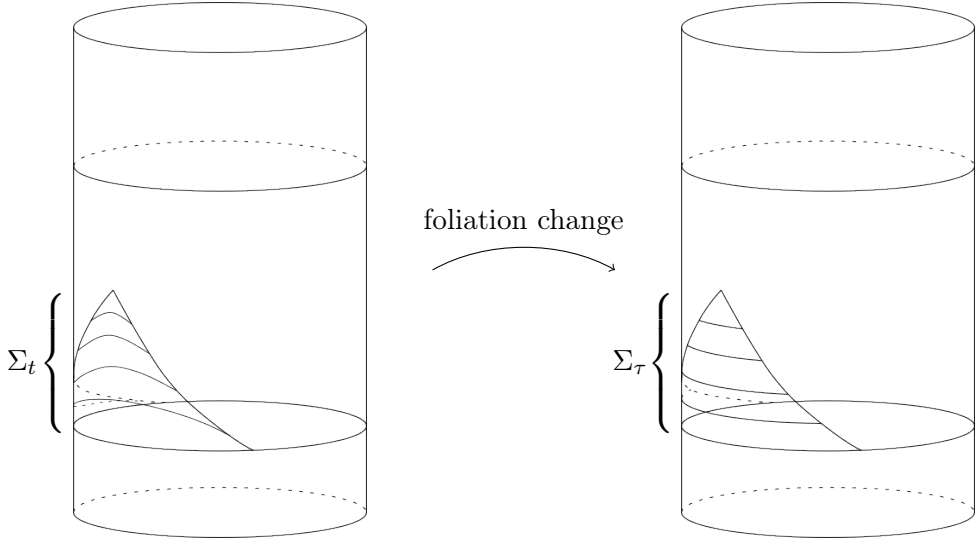

    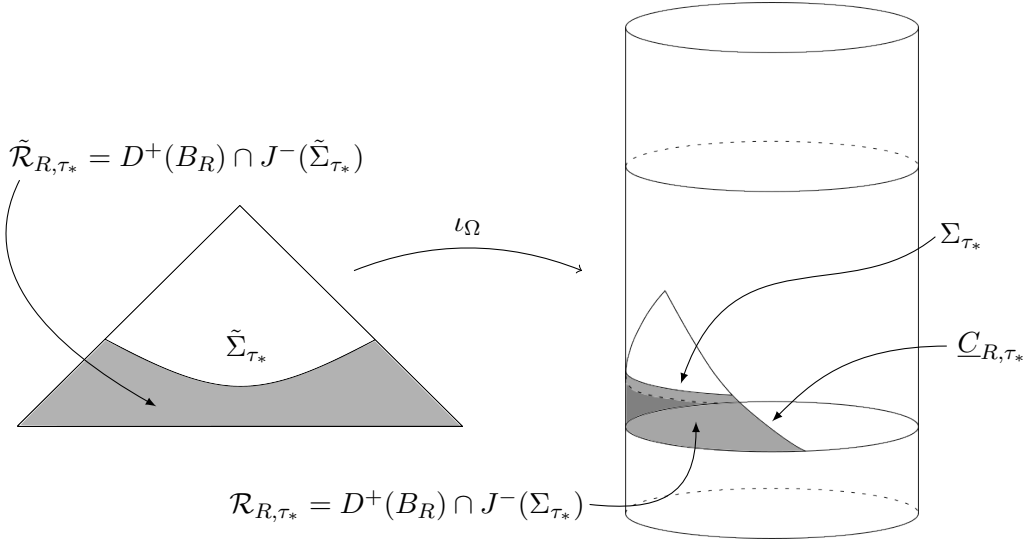
\begin{figure}[h]
    \centering
    \begin{tikzpicture}
        \node[inner sep=0pt] (curved) at (0,-0.6)
        {\includegraphics[width=.38\textwidth]{{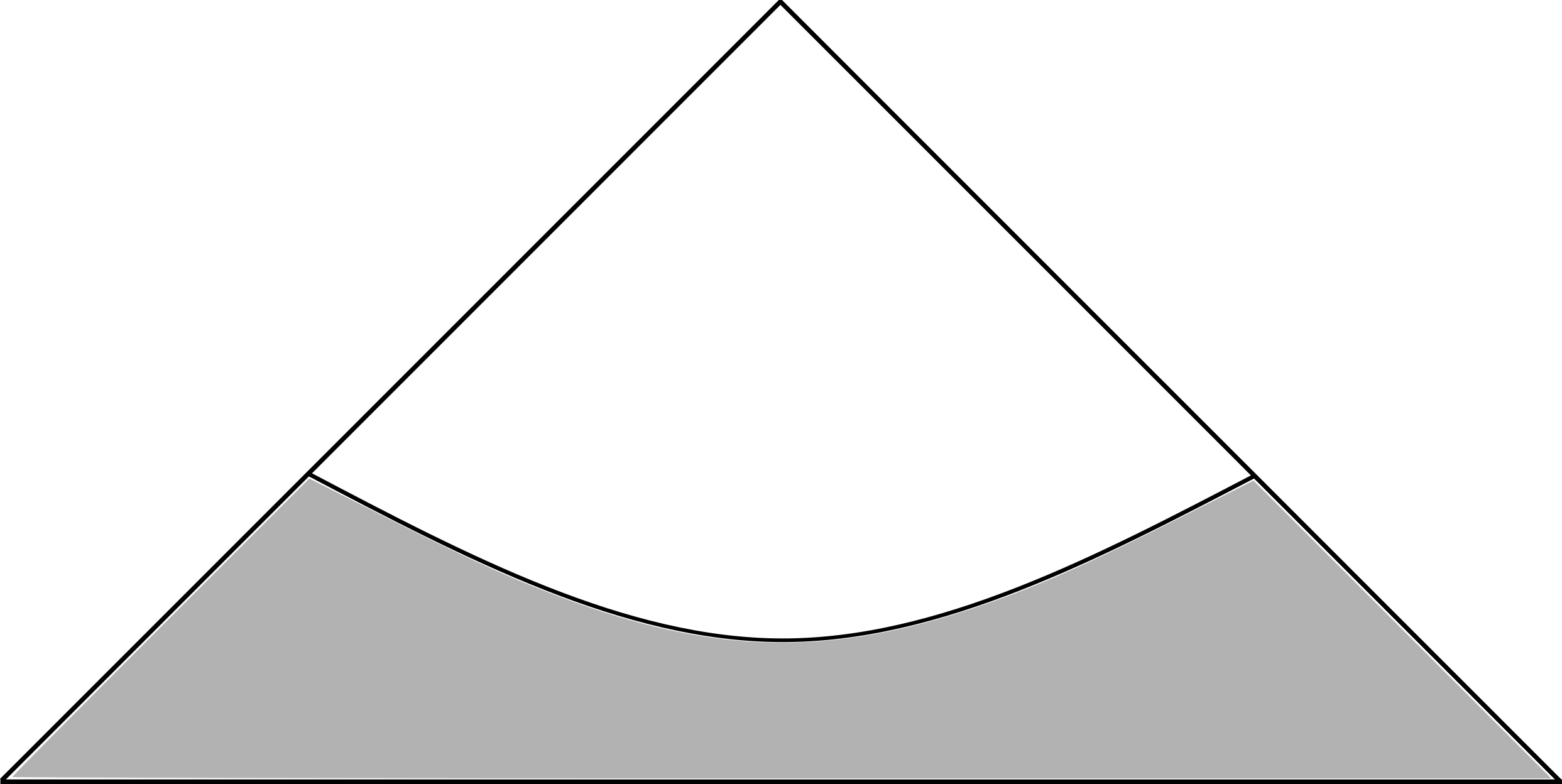}}};
        \node[inner sep=0pt] (straight) at (7,0)
        {\includegraphics[width=.25\textwidth]{{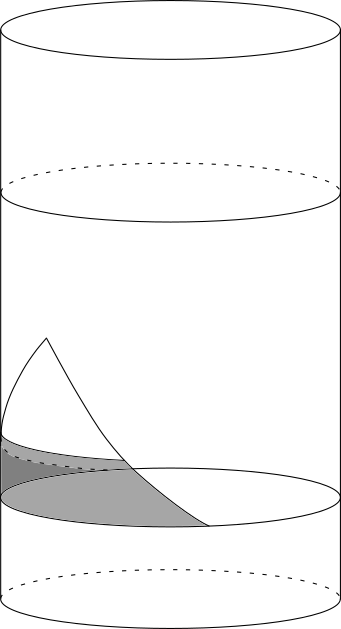}}};
        \draw[->] (1.5,0) .. controls (2.5,0.4) and (3.5,0.4) .. (4.5,0);
        \node[label={[shift={(3,0.2)}]$\iota_\Omega$}] {};
        \node[label={[shift={(0.1,-1.5)}]$\tilde{\Sigma}_{\tau_*}$}] {};
        \node[label={[shift={(9.5,0)}]$\Sigma_{\tau_*}$}] {};
        \draw[-latex] (9.15,0.45) .. controls (8,-0.5) and (6.5,0.5) .. (5.8,-1.5);
        \node[label={[shift={(-0.7,1)}]$\tilde{\mathcal{R}}_{R,\tau_*} = D^+(B_R) \cap J^-(\tilde{\Sigma}_{\tau_*})$}] {};
        \draw[-latex] (-2.9,1.2) .. controls (-3.5,0) and (-2.5,-1) .. (-1.1,-1.8);
        \node[label={[shift={(2.2,-3.6)}]$\mathcal{R}_{R,\tau_*} = D^+(B_R) \cap J^-(\Sigma_{\tau_*})$}] {};
        \draw[-latex] (4.6,-3.1) .. controls (5.5,-3.2) and (6,-3) .. (6,-2);
        \node[label={[shift={(9.9,-1.5)}]$\underline{
        C}_{R,\tau_*}$}] {};
        \draw[-latex] (9.3,-1) .. controls (8,-1) and (7.5,-1.5) .. (7,-2);
    \end{tikzpicture}
    \caption{The image of $\tilde{\mathcal{R}}_{R,\tau_*}$ in $\mathbb{R} \times \mathbb{S}^3$ under the conformal embedding $\mathbb{M} \hookrightarrow \mathbb{R} \times \mathbb{S}^3$.}
    \label{fig:image_of_partial_slab}
    \end{figure}
    
In $\mathcal{R}_{R,\tau_*}$ we now denote by
\begin{equation} \label{conformally_transported_local_solution_definition} A_a \defeq (\iota_\Omega)_* \tilde{A}_a, \quad F_{ab} \defeq (\iota_\Omega)_* \tilde{F}_{ab} \quad \text{and} \quad
\phi \defeq \Omega^{-1} (\iota_\Omega)_* \tilde{\phi} \end{equation}
the conformally transported fields obtained in \Cref{sec:local_solution}. The potential and electromagnetic field $A_a$ and $F_{ab}$ are related in the usual way, and the fields $(\phi, A_a, F_{ab})$ solve \eqref{general_Maxwell_equation}--\eqref{general_Bianchi_identity} in $\mathcal{R}_{R,\tau_*}$ (by the conformal invariance of \eqref{general_Maxwell_equation}--\eqref{general_Bianchi_identity} according to the rescaling \eqref{conformal_symmetry}). Moreover, $A_a$, $F_{ab}$ and $\phi$, being defined in $\mathcal{R}_{R,\tau_*}$, depend only on $(\tilde{\mathbf{E}}_0, \tilde{\mathbf{B}}_0, \tilde{\phi}_1, \tilde{\phi}_0)|_{B_R}$, i.e. only on $(\mathbf{E}_0, \mathbf{B}_0,U_0,\phi_0)|_{B_R}$, by the finite speed of propagation of the system \eqref{Minkowski_Maxwell_equation_potential}--\eqref{Minkowski_scalar_equation_potential}.

\begin{lemma}
    \label{lem:local_solution_cylinder}
    Let $(\mathbf{E}_0, \mathbf{B}_0, U, \phi_0) \in L^2(\mathbb{S}^3)^4$ be a set of finite energy initial data for \eqref{general_MKG_system_with_projections} on $\mathbb{R} \times \mathbb{S}^3$ in the sense of \Cref{defn_finite_energy_data_on_cylinder}. Then in $\mathcal{R}_{R,\tau_*}$ there exists a local solution $(\phi, A_a, \mathbf{E}, \mathbf{B})$ to \eqref{general_MKG_system_with_projections} in Lorenz gauge on $\mathbb{R} \times \mathbb{S}^3$ such that for any $\delta > 0$ sufficiently small
    \begin{align*}
        &\mathbf{E},\, \mathbf{B} \in \mathcal{C}^0([0,\tau_*];L^2(\Sigma_\tau)), \\
        & \phi \in \mathcal{C}^0([0,\tau_*];H^{1-\delta}(\Sigma_\tau)) \cap \mathcal{C}^1([0,\tau_*];H^{-\delta}(\Sigma_\tau)), \\
        & A_a \in \mathcal{C}^0([0,\tau_*];H^{\frac{1}{2}-\delta}(\Sigma_\tau)) \cap \mathcal{C}^1([0,\tau_*];H^{-\frac{1}{2}-\delta}(\Sigma_\tau))
    \end{align*}
    and
    \[ \D_a \phi \in L^\infty([0,\tau_*];L^2(\Sigma_\tau)) \cap \mathcal{C}^0([0,\tau_*];H^{-\delta}(\Sigma_\tau)), \]
    and the solution verifies the initial condition
    \[ (\mathbf{E}, \mathbf{B}, \D_a \phi, \phi)|_{\tau=0} = (\mathbf{E}_0, \mathbf{B}_0, U, \phi_0). \]
    Moreover, we have the local spacetime regularity for $A_a$ and $\phi$ given by
    \[  A_a, \, \phi \in H^{1-\delta}(\mathcal{R}_{R,\tau_*}). \]
\end{lemma}

\begin{proof} Consider the solution \eqref{conformally_transported_local_solution_definition} obtained in the constructions of \Cref{sec:localization_to_Minkowski,sec:conformal_transport_of_solution,sec:change_of_foliation}. 

\subsubsection*{Regularity implied by \Cref{lem:regularity_with_respect_to_hyperboloidal_foliation}}

First, the local spacetime regularity of $A_a \in H^{1-\delta}(\mathcal{R}_{R,\tau_*})$ and $\phi \in H^{1-\delta}(\mathcal{R}_{R,\tau_*})$ is immediate from \Cref{lem:regularity_with_respect_to_hyperboloidal_foliation}. 

Next, on $\mathbb{M}$ the electromagnetic field $\tilde{F}_{ab}$ is given by
\[ \tilde{F}_{ab} = - 2 \tilde{\mathbf{E}}_{[a} \tilde{T}_{b]} + \frac{1}{3} \tilde{\epsilon}_{abc} \tilde{\mathbf{B}}^c, \]
so it is obvious from \Cref{lem:regularity_with_respect_to_hyperboloidal_foliation} that in $\tilde{\mathcal{R}}_{R,\tau_*}$ $\tilde{F}_{ab} \in \mathcal{C}^0([0,\tau_*];L^2(\tilde{\Sigma}_\tau))$. It is then clear that 
\[ F_{ab} \in \mathcal{C}^0([0,\tau_*];L^2(\Sigma_\tau))
\]
since continuity in $\tau$ is unchanged and the Riemannian metrics on $\tilde{\Sigma}_\tau$ and $\Sigma_\tau$ are related by a conformal factor which is uniformly equivalent to $1$ in $\mathcal{R}_{R,\tau_*}$. Similarly, it is clear from \Cref{lem:regularity_with_respect_to_hyperboloidal_foliation} that
\[ \phi \in \mathcal{C}^0([0,\tau_*];H^{1-\delta}(\Sigma_\tau)) \quad \text{and} \quad A_a \in \mathcal{C}^0([0,\tau_*];H^{\frac{1}{2}-\delta}(\Sigma_\tau)).  \]

\subsubsection*{Regularity of $\tau$-derivatives}

Now recall that the potential $\tilde{A}_a$ obtained in \Cref{sec:local_solution} satisfies the Minkowskian Lorenz gauge condition $\tilde{\nabla}_a \tilde{A}^a = 0$. In terms of the metric on $\mathbb{R} \times \mathbb{S}^3$ this is expressed as
\[ \tilde{\nabla}_a \tilde{A}^a = \Omega^2 ( \nabla_a A^a - 2 (\nabla^a \log \Omega )A_a ), \]
i.e. in $\mathcal{R}_{R,\tau_*}$
\[ \partial_\tau A_0 = \boldsymbol{\nabla} \cdot \mathbf{A} + 2 (\nabla^a \log \Omega ) A_a. \]
The first of the above terms is in $\mathcal{C}^0([0,\tau_*];H^{-\frac{1}{2}-\delta}(\Sigma_\tau))$ and the second, since $\Omega$ is smooth and uniformly equivalent to $1$ in $\mathcal{R}_{R,\tau_*}$, is in $\mathcal{C}^0([0,\tau_*];H^{\frac{1}{2}-\delta}(\Sigma_\tau))$. Hence $A_0 = (\partial_\tau)^a A_a \in \mathcal{C}^1([0,\tau_*];H^{-\frac{1}{2}-\delta}(\Sigma_\tau))$. For the $\tau$-derivatives of the $\Sigma_\tau$-spatial components, from the regularity of $F_{ab}$ above we obtain by contracting with $\partial_\tau$
\[ \partial_\tau \mathbf{A} = \mathbf{E} + \boldsymbol{\nabla}A_0 \in \mathcal{C}^0([0,\tau_*];L^2(\Sigma_\tau)) + \mathcal{C}^0([0,\tau_*];H^{-\frac{1}{2}-\delta}(\Sigma_\tau)), \]
i.e. $A_a \in \mathcal{C}^1([0,\tau_*];H^{-\frac{1}{2}-\delta}(\Sigma_\tau))$. Therefore
\[ A_a \in \mathcal{C}^0([0,\tau_*];H^{\frac{1}{2}-\delta}(\Sigma_\tau)) \cap \mathcal{C}^1([0,\tau_*];H^{-\frac{1}{2}-\delta}(\Sigma_\tau)). \]
To obtain the regularity for the $\tau$-derivative of $\phi$, recall that (from \eqref{general_scalar_equation_potential} and the gauge condition $\nabla_a A^a = 2(\nabla^a \log \Omega)A_a$), the scalar field $\phi$ satisfies the wave equation 
\begin{align*} (\partial_\tau^2 + 1 - \boldsymbol{\Delta}) \phi + 2i A_0 \partial_\tau \phi &= 2i \mathbf{A} \cdot \boldsymbol{\nabla} \phi + A_a A^a \phi - 2i (\nabla^a \log \Omega) A_a \phi 
\end{align*}
in $\mathcal{R}_{R,\tau_*}$. This may be rewritten as
\begin{align}
\label{ODE_for_phi_time_derivative}
\begin{split}
\partial_\tau^2 \phi + 2i \partial_\tau (A_0 \phi) & = 2i \mathbf{A} \cdot \boldsymbol{\nabla} \phi + A_a A^a \phi - 2i (\nabla^a \log \Omega) A_a \phi - (1-\boldsymbol{\Delta})\phi + 2i \phi \partial_\tau A_0 \\
& \eqdef \mathcalboondox{m}(A_a,\phi),
\end{split}
\end{align}
where the right-hand side $\mathcalboondox{m}$ has the regularity
\begin{align*} \mathcalboondox{m}(A_a,\phi) &\in L^\infty([0,\tau_*]; (H^{\frac{1}{2}-\delta} \cdot H^{-\delta})(\Sigma_\tau)) + L^\infty([0,\tau_*]; (H^{\frac{1}{2}-\delta} \cdot H^{\frac{1}{2}-\delta} \cdot H^{1-\delta})(\Sigma_\tau)) \\
& + L^\infty([0,\tau_*]; (H^{\frac{1}{2}-\delta} \cdot H^{1-\delta})(\Sigma_\tau)) + L^\infty([0,\tau_*];H^{-1-\delta}(\Sigma_\tau)) \\
& + L^\infty([0,\tau_*]; (H^{1-\delta} \cdot H^{-\frac{1}{2}-\delta})(\Sigma_\tau)) \\
& \subset L^\infty([0,\tau_*];H^{-1-3\delta}(\Sigma_\tau)),
\end{align*}
where we used the embeddings (cf. \Cref{lem:product_Sobolev_estimates}) $H^{\frac{1}{2}-\delta} \cdot H^{1-\delta} \hookrightarrow H^{- 2\delta}$, $H^{\frac{1}{2}-\delta_1}\cdot H^{-\delta_2} \hookrightarrow H^{-1-\delta_1 -\delta_2}$ and $H^{1-\delta} \cdot H^{-\frac{1}{2}-\delta} \hookrightarrow H^{-1-2\delta}$ in 3 dimensions. By directly integrating in $\tau$,
\[ \partial_\tau \phi  = - 2i A_0 \phi + \phi_1 + 2i a_0\phi_0 + \int_0^\tau \mathcalboondox{m}(A_a, \phi) \in \mathcal{C}^0([0,\tau_*];H^{-1 -3\delta}(\Sigma_\tau)) \]
as $A_0 \phi \in \mathcal{C}^0([0,\tau_*]; H^{-2\delta}(\Sigma_\tau))$. Hence $\phi \in \mathcal{C}^1([0,\tau_*];H^{-1 - 3\delta}(\Sigma_\tau))$.

\subsubsection*{Improved regularity of $\partial_\tau \phi$}

By \Cref{rmk:continuous_dependence_on_data}, the solution $(\tilde{\phi}, \tilde{\mathbf{E}}, \tilde{\mathbf{B}})$ constructed in \Cref{sec:localization_to_Minkowski,sec:conformal_transport_of_solution,sec:change_of_foliation} is a limit in $\mathcal{C}^0H^1 \times \mathcal{C}^0 L^2 \times \mathcal{C}^0 L^2$, as well as the energy space, with respect to the $t$-foliation, of smooth solutions with smooth data. By regularizing the initial data $(\mathbf{E}_0,\mathbf{B}_0,U,\phi_0)$ in $B_R$, we obtain in the same way a sequence of smooth solutions $(\phi^{(n)}, A^{(n)}_a, F^{(n)}_{ab})$ in $\mathcal{R}_{R,\tau_*}$ for which the stress-energy tensor \eqref{general_stress_tensor}, which we denote by $\mathbf{T}_{ab}^{(n)}$, is conserved. We integrate $\nabla^a(T^b \mathbf{T}_{ab}^{(n)}) = 0$ in $\mathcal{R}_{R,\tau}$ for any fixed $\tau < \tau_*$ and use the positivity of the integral on the null (incoming) part $\underline{C}_{R,\tau}$ of $\partial \mathcal{R}_{R,\tau}$. We have
\begin{equation}
\label{integration_by_parts}
    \int_{\Sigma_\tau} T^a T^b \mathbf{T}_{ab}^{(n)} \dvol_{\mathbb{S}^3} + \int_{\underline{C}_{R,\tau}} T^b \mathbf{T}_{b}^{(n)}{}^a \partial_a \intprod \dvol = \int_{\Sigma_0} T^a T^b \mathbf{T}_{ab}^{(n)} \dvol_{\mathbb{S}^3}.
\end{equation}
Choose a null Newman--Penrose tetrad\footnote{The metric then takes the form $g_{ab} = l_a n_b + n_a l_b - m_a \bar{m}_b - \bar{m}_a m_b$ and $T^a = n^a + \frac{1}{2}l^a$. Explicitly, $n^a = \frac{1}{2}(\partial_\tau - \partial_\zeta)$ and $l^a = \partial_\tau + \partial_\zeta$.} $(l^a, n^a, m^a, \bar{m}^a)$ in a neighbourhood of $\mathcal{R}_{R,\tau_*}$ such that $l^a$, $n^a$ are real, future-pointing, and tangent to radial null geodesics with $l^a n_a = 1$, $m^a$ is a complex null vector field tangent to spacelike $2$-spheres with $m^a \bar{m}_a = -1$, and all other contractions zero, and moreover such that $n^a$ is tangent to $\underline{C}_{R,\tau}.$ Then
\begin{align*} \int_{\underline{C}_{R,\tau}} T^b \mathbf{T}_b^{(n)}{}^a \partial_a \intprod \dvol &= \int_{\underline{C}_{R,\tau}} T^b n^a \mathbf{T}_{ab}^{(n)} l^\flat \wedge m^\flat \wedge \bar{m}^\flat \\
& =\int_{\underline{C}_{R,\tau}} \mathbf{T}^{(n)}_{nn} + \frac{1}{2} \mathbf{T}^{(n)}_{nl} \dvol_{\underline{C}},
\end{align*}
where $\mathbf{T}_{nn}^{(n)} = n^a n^b \mathbf{T}^{(n)}_{ab}$ and similarly for $\mathbf{T}^{(n)}_{nl}$, and $\dvol_{\underline{C}} = l^\flat \wedge m^\flat \wedge \bar{m}^\flat$. A computation then shows that
\begin{align*} \mathbf{T}^{(n)}_{nn} &= 2 |F_{nm}^{(n)} |^2 + |(\D_n \phi)^{(n)}|^2 \geq 0, \\
\mathbf{T}_{nl}^{(n)} & = \frac{1}{2} \left( |F_{nl}^{(n)}|^2 + |F_{m\bar{m}}^{(n)}|^2 + |(\D_m \phi)^{(n)}|^2 + |(\D_{\bar{m}}\phi)^{(n)}|^2 + |\phi^{(n)}|^2 \right) \geq 0.
\end{align*}
Therefore the second term in \eqref{integration_by_parts} is non-negative, and we deduce for $n$ sufficiently large that 
\begin{align*}
   \int_{\Sigma_\tau} |\mathbf{E}^{(n)}|^2 &+ |\mathbf{B}^{(n)}|^2 + |(\mathrm{D}_0 \phi)^{(n)}|^2 + |(\mathrm{\mathbf{D}}\phi)^{(n)}|^2 + |\phi^{(n)}|^2 \dvol_{\mathbb{S}^3} \\
   & \leq \int_{\Sigma_0} |\mathbf{E}^{(n)}|^2 + |\mathbf{B}^{(n)}|^2 + |(\mathrm{D}_0 \phi)^{(n)}|^2 + |(\mathrm{\mathbf{D}}\phi)^{(n)}|^2 + |\phi^{(n)}|^2 \dvol_{\mathbb{S}^3} \\
   & \leq \mathscr{E}(0) + 1.
\end{align*}
Hence in particular $(\mathrm{D}_a \phi)^{(n)}$ is uniformly bounded in $L^\infty([0,\tau_*];L^2(\Sigma_\tau))$ for any $\tau< \tau_*$, so contains a subsequence which converges weakly-$*$ in $L^\infty([0,\tau_*];L^2(\Sigma_\tau))$. But the regularized solution $(\mathrm{D}_a \phi)^{(n)}$ converges to $\mathrm{D}_a \phi$ in $\mathcal{C}^0 L^2$ with respect to the $t$-foliation (since $(\tilde{\mathrm{D}}_a \tilde{\phi})^{(n)}$ does), so in, say, $L^2(\mathcal{R}_{R,\tau_*})$. By the uniqueness of limits in $L^2(\mathcal{R}_{R,\tau_*})$, we therefore must have 
\[
\mathrm{D}_a \phi \in L^\infty([0,\tau_*];L^2(\Sigma_\tau)).
\]
But now 
\begin{align*} \partial_\tau \phi = D_0 \phi - i A_0 \phi &\in L^\infty([0,\tau_*];L^2(\Sigma_\tau)) + \mathcal{C}^0([0,\tau_*];H^{ -2\delta}(\Sigma_\tau)) \\
&\subset L^\infty([0,\tau_*];H^{-2\delta}(\Sigma_\tau)).
\end{align*}
Also, returning to \eqref{ODE_for_phi_time_derivative}, we have
\[ \partial_\tau^2 \phi + 2i A_0 \partial_\tau \phi \in L^\infty([0,\tau_*];H^{-1 -3\delta}(\Sigma_\tau)), \]
where $A_0 \partial_\tau \phi \in L^\infty([0,\tau_*];(H^{\frac{1}{2}-\delta}\cdot H^{-2\delta})(\Sigma_\tau)) \subset L^\infty([0,\tau_*];H^{-1 -3\delta}(\Sigma_\tau))$. Hence 
\[ 
\partial_\tau^2 \phi \in L^\infty([0,\tau_*];H^{-1 -3\delta}(\Sigma_\tau)).
\]
By the Aubin--Lions\footnote{Recall that the Aubin--Lions lemma \cite{Aubin1963} implies that if $X_0 \subset \subset X \hookrightarrow X_1$ are Banach spaces such that $X_0$ is compactly embedded in $X$ and $X$ is continuously embedded in $X_1$, then $L^\infty([0,\tau];X_0) \cap W^{1,\infty}([0,\tau];X_1)$ is compactly embedded in $\mathcal{C}^0([0,\tau];X)$. Our conclusion follows by choosing $X_1 = H^{-1-3\delta}(\Sigma_\tau)$, $X_0 = H^{ - 2\delta}(\Sigma_\tau)$, and $X=H^{-3\delta}(\Sigma_\tau)$, and noting that by the Kondrachov embedding theorem and Schauder's theorem (a bounded linear operator between Banach spaces is compact if and only if its adjoint is compact) $H^{-2\delta} \subset \subset H^{-3\delta} \subset \subset H^{-1-3\delta}$.} lemma, we then have
\[ \partial_\tau \phi \in \mathcal{C}^0([0,\tau_*];H^{-3\delta}(\Sigma_\tau)). \]

\subsubsection*{Regularity of $\mathrm{D}_a \phi$}

It then follows that
\[ \D_0 \phi = \partial_\tau \phi + iA_0 \phi \in \mathcal{C}^0([0,\tau_*];H^{-3\delta}(\Sigma_\tau)) \]
and moreover that 
\[ \boldsymbol{\mathrm{D}} \phi = \boldsymbol{\nabla} \phi  + i\mathbf{A} \phi \in \mathcal{C}^0([0,\tau_*];H^{-2\delta}(\Sigma_\tau)). \]
Hence $\mathrm{D}_a \phi \in L^\infty([0,\tau_*];L^2(\Sigma_\tau)) \cap \mathcal{C}^0([0,\tau_*];H^{-3\delta}(\Sigma_\tau))$.

\subsubsection*{Change of gauge}

It remains to deal with the gauge. By construction, the potential $A_a$ satisfies the Minkowskian Lorenz gauge $\tilde{\nabla}_a \tilde{A}^a = 0$, i.e., writing $\Upsilon_a = \nabla_a \log \Omega$,
\[ \nabla_a A^a = \Omega^{-2} \tilde{\nabla}_a \tilde{A}^a + 2 \Upsilon_a A^a = 2 (\Omega^{-1} \nabla^a \Omega) A_a. \]
We must therefore check that the gauge transformation to the Lorenz gauge on $\mathcal{R}_{R,\tau_*} \subset \mathbb{R} \times \mathbb{S}^3$ preserves regularity. For this we transform $A_a \rightsquigarrow A_a + \nabla_a \chi$, where $\chi$ solves 
\begin{align}
\label{gauge_change_wave_equation_from_Minkowski_to_cylinder_Lorenz}
\begin{split}
    & \Box \chi = - 2\Upsilon_a A^a \in H^{1-\delta}(\mathcal{R}_{R,\tau_*}), \\
    &(\chi, \partial_\tau \chi)|_{\tau =0} = (0,0).
\end{split}
\end{align}
By the energy inequality for linear waves \eqref{energy_inequality}, for any $\tau \in [0,\tau_*]$
\begin{align*} \| \partial_\tau \chi \|_{H^{1-\delta}(\Sigma_\tau)} + \| \chi \|_{H^{2-\delta}(\Sigma_\tau)} &\la \int_0^\tau \| A_a \|_{H^{1-\delta}(\Sigma_\sigma)} \, \d \sigma \\
& \la \| A_a \|_{L^2([0,\tau_*];H^{1-\delta}(\Sigma_\tau))} \\
& \la \| A_a \|_{H^{1-\delta}(\mathcal{R}_{R,\tau_*})},
\end{align*}
so that, via a regularization, $\chi \in \mathcal{C}^0([0,\tau_*];H^{2-\delta}(\Sigma_\tau)) \cap \mathcal{C}^1([0,\tau_*];H^{1-\delta}(\Sigma_\tau))$. Hence $\nabla_a \chi \in \mathcal{C}^0([0,\tau_*];H^{1-\delta}(\Sigma_\tau))$, and
\[ A_a \in \mathcal{C}^0([0,\tau_*];H^{\frac{1}{2}-\delta}(\Sigma_\tau)) \cap \mathcal{C}^1([0,\tau_*];H^{-\frac{1}{2}-\delta}(\Sigma_\tau)) \]
in the new gauge. Similarly, since $\mathbf{E}$ and $\mathbf{B}$ are gauge invariant,
\[\mathbf{E}, \, \mathbf{B} \in \mathcal{C}^0([0,\tau_*];L^2(\Sigma_\tau))\] 
in the new gauge. Additionally, by \Cref{lem:foliation_change_positive_s} we also have $\chi \in \mathcal{C}^0([0,t_*];H^{2-\delta}(\Sigma_t)) \cap \mathcal{C}^1([0,t_*];H^{1-\delta}(\Sigma_t))$ for some $t_* > 0$, where $\Sigma_t = \iota_\Omega (\tilde{\Sigma}_t)$ is the flat Minkowskian foliation of $\mathcal{R}_{R,\tau_*}$, with respect to which it is clear that $\chi \in H^{2-\delta}(\mathcal{R}_{R,\tau_*})$ by a Fourier argument as in \eqref{simple_Fourier_argument}. Hence also
\[ A_a \in H^{1-\delta}(\mathcal{R}_{R,\tau_*}) \]
in the new gauge.

For $\D_a \phi$, the gauge transformation sends $\D_a \phi \rightsquigarrow \e^{-i\chi} \D_a \phi$, so regularity in time is clear and, using product estimates in Sobolev spaces (cf. \Cref{lem:product_Sobolev_estimates}) and the regularity obtained above for $\chi$, we have
\[ \| \e^{-i\chi} \D_a \phi \|_{Y} \la \| \e^{-i \chi} \|_{H^{2-\delta}(\Sigma_\tau)} \| \D_a \phi \|_{Y}, \]
where $Y \in \{L^2(\Sigma_\tau),  H^{-3\delta}(\Sigma_\tau)\}$. Here we used \Cref{lem:composition_Sobolev_functions_1} and the fact that $1$, $x \mapsto \cos x - 1$ and $x \mapsto\sin x$ are smooth bounded functions on $\Sigma_\tau$, and $2 - \delta > \frac{3}{2}$, to deduce that $\e^{-i\chi}|_{\Sigma_\tau} \in H^{2-\delta}(\Sigma_\tau)$.
Hence 
\[ \D_a \phi \in L^\infty([0,\tau_*];L^2(\Sigma_\tau)) \cap \mathcal{C}^0([0,\tau_*];H^{-3\delta}(\Sigma_\tau)) \]
in the new gauge.

For $\phi$, the gauge transformation similarly sends $\phi \rightsquigarrow \e^{-i\chi} \phi$. Here we analogously have, using product Sobolev estimates, the regularity of $\e^{-i\chi}$, and the regularity of $\phi$ in the Minkowskian Lorenz gauge,
\[ \| \e^{-i \chi} \phi \|_{H^{1-\delta}(\Sigma_\tau)} \la \| \e^{-i\chi} \|_{H^{2-\delta}(\Sigma_\tau)} \| \phi \|_{H^{1-\delta}(\Sigma_\tau)}, \]
and for the time derivative $\partial_\tau ( \e^{-i\chi} \phi) = \e^{-i\chi}(\partial_\tau \phi - i \phi \partial_\tau \chi )$ 
\[ \| \e^{-i\chi} \partial_\tau \phi \|_{H^{-3\delta}(\Sigma_\tau)} \la \| \e^{-i\chi} \|_{H^{2-\delta}(\Sigma_\tau)} \| \partial_\tau \phi \|_{H^{-3\delta}(\Sigma_\tau)} \]
and
\begin{align*} \| \e^{-i \chi} \partial_\tau \chi \phi \|_{H^{-3\delta}(\Sigma_\tau)} &\la \| \e^{-i\chi} \partial_\tau \chi \|_{H^{1-\delta}(\Sigma_\tau)} \| \phi \|_{H^{1-\delta}(\Sigma_\tau)} \\
&\la \| \e^{-i\chi} \|_{H^{2-\delta}(\Sigma_\tau)} \| \partial_\tau \chi \|_{H^{1-\delta}(\Sigma_\tau)} \| \phi \|_{H^{1-\delta}(\Sigma_\tau)}. 
\end{align*}
Therefore the regularity
\[ \phi \in \mathcal{C}^0([0,\tau_*];H^{1-\delta}(\Sigma_\tau)) \cap \mathcal{C}^1([0,\tau_*];H^{-3\delta}(\Sigma_\tau)) \]
is preserved by the gauge transformation by $\chi$. Finally, for the spacetime regularity of $\phi$ in the new gauge, we apply a product Sobolev estimate in spacetime:
\[ \| \e^{-i\chi} \phi\|_{H^{1-3\delta}(\mathcal{R}_{R,\tau_*})} \la \| \e^{-i\chi} \|_{H^{2-\delta}(\mathcal{R}_{R,\tau_*})} \| \phi \|_{H^{1-\delta}(\mathcal{R}_{R,\tau_*})}. \]
Here, although $ 4/2 = 2 > 2-\delta$, \Cref{lem:composition_Sobolev_functions_1} nevertheless applies since 
\[ \chi \in L^\infty([0,\tau_*];H^{2-\delta}(\Sigma_\tau)) \hookrightarrow L^\infty(\mathcal{R}_{R,\tau_*}) \] 
to give $\e^{-i\chi} \in  H^{2-\delta}(\mathcal{R}_{R,\tau_*})$. Since $\delta > 0$ was arbitrarily small, the regularity claimed in the statement is then obtained by reducing $\delta$ (by a factor of $1/3$).
\end{proof}

\begin{remark}
\label{rmk:Lemma91_initial_gauge}
We remark that the Minkowskian solution in the proof of \Cref{lem:local_solution_cylinder}---before the gauge change to $\nabla_a A^a = 0$---satisfied the initial residual gauge conditions $(\tilde{a}_0, \dot{\tilde{a}}_0) = (0,0)$, i.e. $(a_0, \dot{a}_0) = (0, \Gamma a_Z)$ (cf. \Cref{lem:obtaining_Minkowski_residual_gauge_on_cylinder}). After the gauge change in the proof of \Cref{lem:local_solution_cylinder}, this condition is modified to 
\begin{equation}
    \label{Lemma91_initial_gauge}
    (a_0, \dot{a}_0) = (0, - \Gamma a_Z).
\end{equation}
Indeed, from \eqref{gauge_change_wave_equation_from_Minkowski_to_cylinder_Lorenz} one has
\[ a_0 \rightsquigarrow a_0 + \partial_\tau \chi|_{\tau = 0} = 0, \]
and
\begin{align*} \dot{a}_0 &\rightsquigarrow \dot{a}_0 + \partial_\tau^2 \chi|_{\tau=0} \\
    & = \dot{a}_0 + \boldsymbol{\Delta} \chi \vert_{\tau = 0} - 2 \Upsilon_a A^a|_{\tau = 0} \\
    & = \Gamma a_Z - 2(\partial_\tau \log \Omega)|_{\tau =0} a_0 + 2 (\partial_\zeta \log \Omega)|_{\tau=0} a_Z \\
    & = \Gamma a_Z - 2 \Gamma a_Z \\
    & = - \Gamma a_Z,
\end{align*}
where we used the expression \eqref{standard_conformal_factor_cylinder_coordinates} for $\Omega$, and the fact that $a_Z \rightsquigarrow a_Z + \partial_\zeta \chi|_{\tau=0} = a_Z$.
\end{remark}

\subsection{Covering a strip of $\mathbb{R} \times \mathbb{S}^3$ with two copies of Minkowski space}
\label{sec:two_copies_of_Minkowski}

We now embed two copies of Minkowski space into the Einstein cylinder $\mathbb{R} \times \mathbb{S}^3$ as follows. Let $\mathbb{M}$ be embedded in $\mathbb{R} \times \mathbb{S}^3$ as described in \Cref{sec:conformal_embedding}, with associated conformal factor $\Omega$ and the coordinates $(\tau, \zeta, \theta, \varphi)$. We then declare $\mathbb{M}'$ to be the image of $\mathbb{M}$, $\mathbb{M}' = J(\mathbb{M})$, under the map
\begin{equation} 
\label{reflection_mapping_definition}
J: (\tau, \zeta, \theta, \varphi) \longmapsto (\tau', \zeta', \theta', \varphi') \defeq (\tau, \pi - \zeta, \pi - \theta, \pi + \varphi).
\end{equation}
The $\mathbb{R} \times \mathbb{S}^3$ coordinates $(\tau, \zeta, \theta, \varphi)$, as well as their image under $J$, have ranges $\mathbb{R} \times [0, \pi] \times [0, \pi] \times \mathbb{S}^1$.
The map $J$ is the composition of a reflection in the $\zeta$ coordinate and the antipodal map on $\mathbb{S}^2_{\theta,\varphi}$, and so in particular is orientation-preserving on $\mathbb{S}^3$, and in fact an involution and a global isometry of $\mathbb{R} \times \mathbb{S}^3$. We observe that the intersection $\mathbb{M} \cap \mathbb{M}'$ is non-empty, i.e. $J$ induces a relation between the spherical coordinates $(t,r,\theta,\varphi)$ on $\mathbb{M}$ and $(t',r',\theta',\varphi')$ on $\mathbb{M}'$ in the region where $\mathbb{M}$ and $\mathbb{M}'$ overlap on the Einstein cylinder:
\[ (t',r',\theta' , \varphi' ) = \left( \frac{t}{r^2-t^2} , \frac{r}{r^2-t^2} , \pi - \theta , \pi + \varphi \right) \, .\]
This region is characterised by $r > \vert t \vert$ in $\mathbb{M}$ and $r' > \vert t' \vert$ in $\mathbb{M}'$. Moreover, this change of coordinates is conformal: 
\[ \eta'_{ab} = \Theta^2 \eta_{ab} , \]
where $\Theta = (r^2 -t^2)^{-1} = (\Theta')^{-1}$, $\eta'_{ab}$ and $\eta_{ab}$ are the metrics on $\mathbb{M}'$ and $\mathbb{M}$ respectively, and $\Theta'$ has the same expression as $\Theta$ but with $(t,r)$ replaced with $(t',r')$.

\begin{remark}
On the antipodal copy $\mathbb{M}'$, since $J$ maps $\zeta \mapsto \pi - \zeta$ (cf. \eqref{reflection_mapping_definition}), it is not difficult to see that the residual gauge fixing conditions on $\{ \tau = 0 \} \cap \mathbb{M}'$ which impose \eqref{residual_gauge_fixing} on $\mathbb{M}'$ are 
\[ a_0 = 0 \quad \text{and} \quad \dot{a}_0 = \Gamma' a_Z, \]
where $\Gamma' = \cot \frac{\zeta}{2}$.
\end{remark}

\subsection{Solution on a strip}
\label{sec:solution_on_strip}

\begin{theorem}
    \label{thm:local_strip_solution_cylinder}
    Let $(\mathbf{E}_0, \mathbf{B}_0, U, \phi_0) \in L^2(\mathbb{S}^3)^4$ be a set of finite energy initial data for \eqref{general_MKG_system_with_projections} on $\mathbb{R} 
    \times \mathbb{S}^3$ in the sense of \Cref{defn_finite_energy_data_on_cylinder} and let $\tau_*$ be sufficiently small. Then on $[0,\tau_*] \times \mathbb{S}^3$ there exists a unique local-in-time solution $(\phi, A_a, \mathbf{E}, \mathbf{B})$ to \eqref{general_MKG_system_with_projections} satisfying the Lorenz gauge $\nabla_a A^a = 0$ with the regularity, for any $\delta > 0$ sufficiently small,
    \begin{align*}
        & \mathbf{E}, \, \mathbf{B}, \, \D_a \phi \in \mathcal{C}^0([0,\tau_*];L^2(\mathbb{S}^3)), \\
        & \phi \in \mathcal{C}^0([0,\tau_*];H^{1-\delta}(\mathbb{S}^3)) \cap \mathcal{C}^1([0,\tau_*];H^{-\delta}(\mathbb{S}^3)), \\
        & A_a \in \mathcal{C}^0([0,\tau_*];H^{\frac{1}{2}-\delta}(\mathbb{S}^3)) \cap \mathcal{C}^1([0,\tau_*];H^{-\frac{1}{2}-\delta}(\mathbb{S}^3))
    \end{align*}
    and
    \[ A_a, \, \phi \in H^{1-\delta}([0,\tau_*]\times \mathbb{S}^3). \]
    Moreover, the solution verifies the residual initial gauge condition 
    \[ (a_0, \dot{a}_0) \defeq (A_0, \partial_\tau A_0) |_{\tau = 0} = (0,0), \]
    the initial condition
    \[ (\mathbf{E}, \mathbf{B}, \D_a \phi, \phi)|_{\tau = 0} = (\mathbf{E}_0, \mathbf{B}_0, U, \phi_0), \]
    and the conservation of energy $\mathscr{E}(\tau) = \mathscr{E}(0)$ for all $\tau \in [0,\tau_*]$.
\end{theorem}

\begin{proof}

\smallskip
\emph{Existence.}
By \Cref{lem:local_solution_cylinder}, a solution satisfying $\nabla_a A^a = 0$ with the stated regularity exists in the truncated future domain of dependence $\mathcal{R}_{R,\tau_*}$ of $B_R$ (see \Cref{fig:flat_Minkowski_intersections}). This solution satisfies $(a_0, \dot{a}_0) = (0, -\Gamma a_Z)$ (cf. \Cref{rmk:Lemma91_initial_gauge}), so in $\mathcal{R}_{R,\tau_*}$ it remains to check that the residual gauge transformation to set $\dot{a}_0 = 0$ preserves regularity. This follows from the same argument as in the proof of \Cref{lem:local_solution_cylinder}, provided we can show that the gauge transformation $\chi$ has at least the regularity $\chi \in \mathcal{C}^0([0,\tau_*];H^{2-\delta}(\Sigma_\tau)) \cap \mathcal{C}^1([0,\tau_*];H^{1-\delta}(\Sigma_\tau))$. Indeed, this is true: the required gauge transformation satisfies
\[ \Box \chi = 0, \qquad (\chi, \partial_\tau \chi)|_{\tau=0} = \left(\boldsymbol{\Delta}^{-1}\left( \Gamma a_Z \right), ~ 0\right), \]
so in $\mathcal{R}_{R,\tau_*}$, by the energy inequality \eqref{energy_inequality},
\begin{align*}
    \| \chi (\tau, \cdot) \|_{H^\sigma(\Sigma_\tau)} + \| \partial_\tau \chi \|_{H^{\sigma-1}(\Sigma_\tau)} & \la \| \boldsymbol{\Delta}^{-1}(\Gamma a_Z)\|_{H^{\sigma}(\Sigma_\tau)} \\
    & \la \|\Gamma a_Z \|_{H^{\sigma-2}(\Sigma_\tau)} \\
    & \la \| \mathbf{A} \|_{L^\infty([0,\tau_*];H^{\sigma-2}(\Sigma_\tau))},
\end{align*}
where we used the fact that $\Gamma = \tan(\zeta/2)$ remains bounded on $B_R$. Thus, choosing, say, $\sigma = 2$ gives the required regularity for $\chi$. Note that this fact is a consequence of the form of the residual gauge condition \eqref{residual_gauge_fixing_cylinder_one_copy} despite the fact that the potential $A_a$ does not have $\mathcal{C}^0 H^1$ regularity, since the first order derivative $\dot{a}_0$ is expressed in terms of an undifferentiated component of $\mathbf{A}$.

\begin{figure}[h]
    \centering
    \begin{tikzpicture}
    \node[inner sep=0pt] (intersections) at (0,-0.6)
        {\includegraphics[scale=0.8]{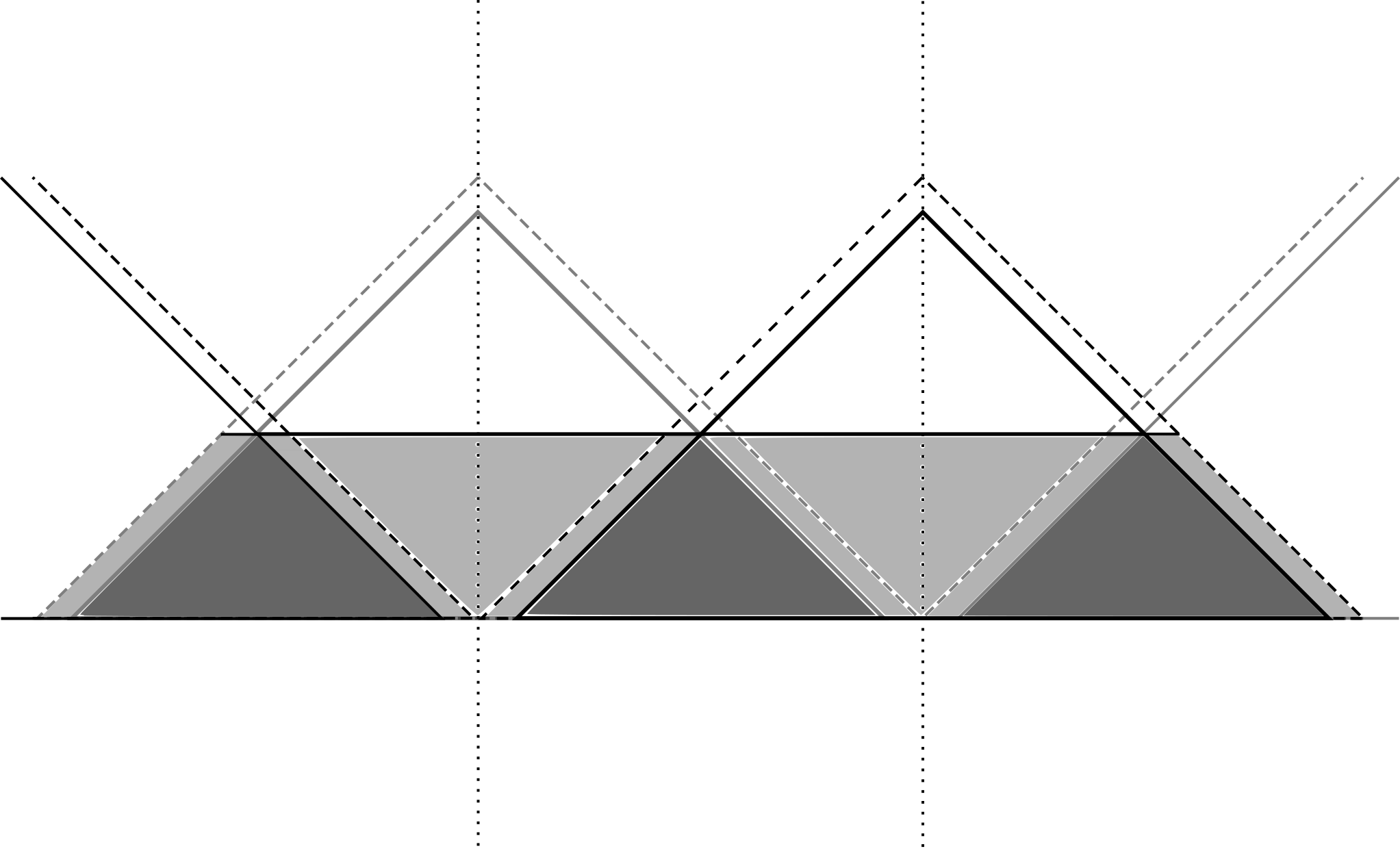}};
    \node[label={[shift={(0,-5)}]$\mathcalboondox{O} \defeq  \mathcal{R}_{R,\tau_*} \cap \mathcal{R}_{R,\tau_*}'$}] {};
    \draw[-latex] (1.6,-4.4) .. controls (3, -4) and (4,-3.5) .. (4.5,-2);
    \draw[-latex] (-1.8,-4.4) .. controls (-3, -4) and (-4,-3.5) .. (-4.5,-2);
    \draw[-latex] (0,-4) .. controls (0, -4) and (0,-3) .. (0,-2);
    \node[label={[shift={(1.8,-0.8)}]$\Sigma_{\tau_*}$}] {};
    \node[label={[shift={(-2.6,-0.8)}]$\Sigma_{\tau_*}'$}] {};
    \node[label={[shift={(2.2,3.5)}]$N$}] {};
    \node[label={[shift={(-2.2,3.5)}]$S$}] {};
    \node[label={[shift={(5.4,-4.5)}]$(i^0)' = O$}] {};
    \draw[-latex] (4.6,-4) .. controls (0, -3) and (5,-3.5) .. (2.25,-2.57);
    \node[label={[shift={(-5.3,-4.4)}]$O' = i^0$}] {};
    \draw[-latex] (-4.6,-4) .. controls (0, -3) and (-5,-3.5) .. (-2.25,-2.57);
    \node[label={[shift={(-4,2)}]$(\scri^+)'$}] {};
    \draw[-latex] (-4.6,2.45) .. controls (-5.5, 1.5) and (-4.5,1) .. (-3.6,0.6);
    \draw[-latex] (-3.4,2.5) .. controls (-0.2, 2) and (-0.1,1) .. (-0.8,0.6);
    \node[label={[shift={(4,2.1)}]$\scri^+$}] {};
    \draw[-latex] (4.4,2.45) .. controls (5.5, 1.5) and (4.5,1) .. (3.6,0.6);
    \draw[-latex] (3.6,2.5) .. controls (0.2, 2) and (0.1,1) .. (0.8,0.6);
    \end{tikzpicture}
    \caption{The overlap of the two (future halves of) antipodal copies of Minkowski space $\mathbb{M}_+$ and $\mathbb{M}'_+$ conformally embedded in $\mathbb{R} \times \mathbb{S}^3$, unwrapped. The picture should be read as being horizontally periodic. The North and South poles of the $3$-sphere are denoted by $N$ and $S$ respectively.}
    \label{fig:flat_Minkowski_intersections}
\end{figure}
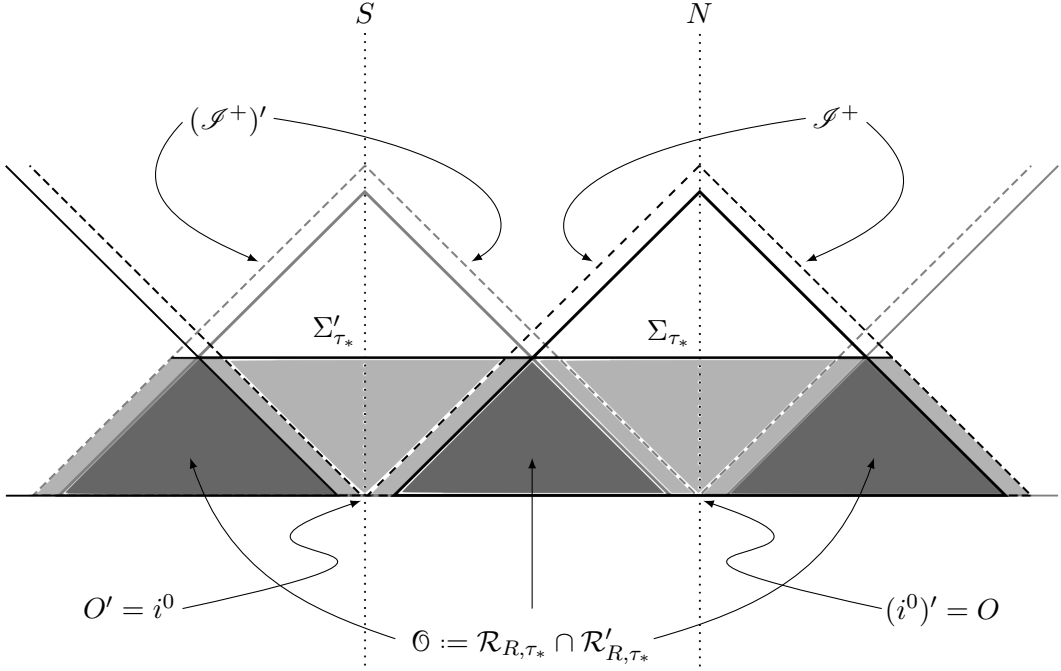

Next, we perform the same construction in $(D^+(B_R))' \defeq J(D^+(B_R)) \subset \mathbb{M}'$, where $J$ is the map \eqref{reflection_mapping_definition}, to obtain a solution $(\phi', A_a', \mathbf{E}', \mathbf{B}')$ with the same regularity in $\mathcal{R}_{R,\tau_*}' \defeq J(\mathcal{R}_{R,\tau_*})$. Since $\mathcal{R}_{R,\tau_*} \cup \mathcal{R}_{R,\tau_*}'$ covers $[0,\tau_*] \times \mathbb{S}^3$, to complete the existence part it remains only to show that the primed and unprimed solutions agree on the overlap region $\mathcal{R}_{R,\tau_*} \cap \mathcal{R}_{R,\tau_*}'$. Observe that the overlap region is the domain of dependence of $B_R \cap J(B_R) = B_R \cap B_R'$, the latter intersection being characterized by 
\[ \pi - 2 \arctan R < \zeta < 2 \arctan R \]
and non-empty if and only if $R>1$ (and approaching the full range $(0,\pi)$ of $\zeta$ as $R \to \infty$). Denote by $\Sigma_\tau'$ the image of $\Sigma_\tau$ under $J$ and consider the two solutions on the overlap $\mathcal{R}_{R,\tau_*} \cap \mathcal{R}_{R,\tau_*}'$. As in the proof of \Cref{lem:local_solution_cylinder}, changing the gauges of both solutions to
\begin{align}
\label{gauge_change_on_overlap}
\begin{split}
    & \nabla^a A_a = 2 \Upsilon^a A_a, \\
    & \nabla^a A_a' = 2 \Upsilon^a A_a'
\end{split}
\end{align}
preserves regularity, i.e. we have
\begin{align*} 
    & F_{ab}, \, F_{ab}' \in \mathcal{C}^0([0,\tau_*];L^2(\Sigma_\tau \cap \Sigma_\tau')), \\
    & \phi, \, \phi' \in \mathcal{C}^0([0,\tau_*];H^{1-\delta}(\Sigma_\tau \cap \Sigma_\tau')) \cap \mathcal{C}^1([0,\tau_*];H^{-\delta}(\Sigma_\tau \cap \Sigma_\tau')), \\
    & A_a, \, A_a'  \in \mathcal{C}^0([0,\tau_*];H^{\frac{1}{2}-\delta}(\Sigma_\tau \cap \Sigma_\tau')) \cap \mathcal{C}^1([0,\tau_*];H^{-\frac{1}{2}-\delta}(\Sigma_\tau \cap \Sigma_\tau')),
    \end{align*}
and under the pullback along $\iota_\Omega$, both $A_a$ and $A_a'$ are in the Minkowskian Lorenz gauge on $\mathcal{R}_{R,\tau_*} \cap \mathcal{R}_{R,\tau_*}' \subset \mathbb{M}$. The gauge change \eqref{gauge_change_on_overlap} also modifies the initial gauge by the inverse of the change in \Cref{rmk:Lemma91_initial_gauge}, but, by the argument in the previous paragraph, we can fix $(\tilde{a}_0, \dot{\tilde{a}}_0) = (0,0) = (\tilde{a}_0', \dot{\tilde{a}}_0')$ while preserving the above regularity (note that $B_R \cap B_R'$ is bounded away from both spatial infinities $i^0$ and $(i^0)'$). By construction, these solutions have the data on $B_R \cap B_R'$ given by $(\mathbf{E}_0, \mathbf{B}_0, U, \phi_0) \in (L^2)^4$ (the data being the same for both the primed and unprimed solutions), and hence the conformally related Minkowskian solutions
\[ \tilde{F}_{ab} = F_{ab}, \quad \tilde{F}_{ab}' = F_{ab}', \quad \tilde{\phi} = \Omega \phi, \quad \tilde{\phi}' = \Omega \phi' \]
have data on $B_R \cap B_R'$ given by $(\tilde{\mathbf{E}}_0, \tilde{\mathbf{B}}_0, \tilde{\phi}_1, \tilde{\phi}_0) \in (L^2)^3 \times H^1$, by \Cref{lem:equivalence_of_Minkowskian_initial_data_Lorenz_gauge_and_temporal_initially} and our choice of initial residual gauge fixing. Write $\mathcalboondox{O} = \mathcal{R}_{R,\tau_*} \cap \mathcal{R}_{R,\tau_*}'$ and
\[ \mathcalboondox{E}_k^s(\mathcalboondox{O}) = \bigcap_{l=0}^k \mathcal{C}^l([0,\tau_*];H^{s-l}(\Sigma_\tau \cap \mathcalboondox{O})) \]
and 
\[ \tilde{\mathcalboondox{E}}^s_k(\mathcalboondox{O}) = \bigcap_{l=0}^k \mathcal{C}^l([0,t_*];H^{s-l}(\Sigma_t \cap \mathcalboondox{O})), \]
where $t_* = \sup_{[0,\tau_*]} t$, and $\{ \Sigma_t \}_t$ is the foliation of $\mathcalboondox{O}$ by the leaves of constant $t$. Then by \Cref{Selberg_Tesfahun_theorem},
\[ F_{ab}, \, F_{ab}' \in \mathcalboondox{E}^1_0(\mathcalboondox{O}) \cap \tilde{\mathcalboondox{E}}^1_0(\mathcalboondox{O}), \]
\[ \phi, \, \phi' \in \mathcalboondox{E}^{1-\delta}_1(\mathcalboondox{O}) 
\cap \tilde{\mathcalboondox{E}}^1_1(\mathcalboondox{O}), \]
and
\[ A_a, \, A_a' \in \mathcalboondox{E}^{\frac{1}{2}-\delta}_1(\mathcalboondox{O}) \cap \tilde{\mathcalboondox{E}}^{1-\delta}_1(\mathcalboondox{O}). \]
By the uniqueness clause in \Cref{Selberg_Tesfahun_theorem} (in the spaces $\tilde{\mathcalboondox{E}}^1_0$, $\tilde{\mathcalboondox{E}}^1_1$, and $\tilde{\mathcalboondox{E}}^{1-\delta}_1$ for the Maxwell field, scalar field, and potential, respectively), we have $F_{ab} = F_{ab}'$, $\phi = \phi'$ and $A_a = A_a'$.

\emph{Uniqueness.} Uniqueness follows by domain of dependence arguments and precisely the same argument as for agreement on the overlap region $\mathcalboondox{O}$ in the above paragraph.

\emph{Conservation of energy.} Regularize the initial data to obtain a set of smooth data 
\[
(\mathbf{E}_0^{(n)}, \mathbf{B}^{(n)}_0, U^{(n)}, \phi_0^{(n)})
\]
which converges to $(\mathbf{E}_0, \mathbf{B}_0, U, \phi_0)$ in $L^2(\mathbb{S}^3)^4$. Integrating 
\[ \nabla^a( T^b \mathbf{T}_{ab}^{(n)} ) = 0 
\]
on the strip $[0,\tau] \times \mathbb{S}^3$, $0 \leq \tau \leq \tau_*$, where $\mathbf{T}_{ab}^{(n)}$ is the stress energy tensor \eqref{general_stress_tensor} on $\mathbb{R} \times \mathbb{S}^3$ corresponding to the (smooth) solution with index $n$, gives
\[ \mathscr{E}^{(n)}(\tau) = \mathscr{E}^{(n)}(0) \]
for the energy \eqref{cylinder_energy} corresponding to the solution with index $n$. By construction, we immediately have $\mathscr{E}^{(n)}(0) \to \mathscr{E}(0)$, so it remains to show that $\mathscr{E}^{(n)}(\tau)$ approaches $\mathscr{E}(\tau)$; for this, it is enough to show that this is true locally, i.e. in each Minkowskian region. For each $\tau$, write $\{ \tau \} \times \mathbb{S}^3$ as a union of disjoint (incomplete) hypersurfaces $\overline{\Sigma_\tau^1} \cup \Sigma_\tau^2$ such that $\Sigma_\tau^1 \in \mathcal{R}_{R,\tau_*}$ and $\Sigma_\tau^2 \in \mathcal{R}_{R,\tau_*}'$. We show that on each piece $\Sigma_{\tau}^k$
\begin{align*}
    & \| F_{ab}^{(n)} \|_{L^2(\Sigma_{\tau}^k)}  \to \| F_{ab} \|_{L^2(\Sigma_{\tau}^k)}, \\
    & \| (\D_a \phi)^{(n)} \|_{L^2(\Sigma_{\tau}^k)} \to \| \D_a \phi \|_{L^2(\Sigma_{\tau}^k)}, \\
    & \| \phi^{(n)} \|_{L^2(\Sigma_{\tau}^k)} \to \| \phi \|_{L^2(\Sigma_{\tau}^k)}.
\end{align*}
Indeed, note that in $\mathcal{R}_{R,\tau_*}$, the solution $(\phi^{(n)}, A_a^{(n)}, F_{ab}^{(n)})$ has an associated smooth Minkowskian solution $(\tilde{\phi}^{(n)}, \tilde{A}_a^{(n)}, \tilde{F}_{ab}^{(n)})$ which converges in the energy norm with respect to the Minkowskian foliation on $\mathcal{R}_{R,\tau_*}$ (cf. \Cref{rmk:continuous_dependence_on_data}). Observe that for any function $f \in L^\infty_t L^2_x$ on $\mathbb{R}^{1+3}$,
\begin{align*}
    \| f \|^2_{L^2(\Sigma^k_\tau)} & = \int_{\Sigma^k_\tau} |f(t(\tau,\zeta),r(\tau,\zeta),\omega)|^2\, \d \zeta \, \d^2 \omega \\
    & \leq \int_{\Sigma^k_\tau} \sup_t |f(t,r(\tau,\zeta), \omega)|^2\, \d \zeta \, \d^2 \omega \\
    & \la \sup_t \int |f(t,r(\tau,\zeta),\omega)|^2 \, \d r|_\tau \, \d^2 \omega \\
    & \la \| f \|^2_{L^\infty_t L^2_x},
\end{align*}
where the penultimate inequality follows from the fact that locally $C^{-1} \leq \frac{\d r}{\d \zeta}|_{\tau} \leq C$ for some $C>0$. We therefore have, using the uniform boundedness of all metric components and the conformal factor $\Omega$ on and in a neighbourhood of $\Sigma_{\tau}^1$,
\begin{equation} 
    \label{limit_in_L2_Maxwell_field}
\| F_{ab}^{(n)} - F_{ab} \|_{L^2(\Sigma_{\tau}^1)} \la \| \tilde{F}_{ab}^{(n)} - \tilde{F}_{ab} \|_{L^\infty_t L^2_x} \to 0,
\end{equation}
\begin{equation}
    \label{limit_in_L2_scalar_field}
\| \phi^{(n)} - \phi \|_{L^2(\Sigma_{\tau}^1)} \la \| \tilde{\phi}^{(n)} - \tilde{\phi} \|_{L^\infty_t L^2_x} \to 0,
\end{equation}
and
\begin{align}
\label{limit_in_L2_gauge_covariant_derivative}
\begin{split}
    \| (\D_a \phi)^{(n)} -\D_a \phi\|_{L^2(\Sigma_{\tau}^1)} & \leq \| \Omega^{-1} ( (\tilde{\D}_a \tilde{\phi})^{(n)} - \tilde{\D}_a \tilde{\phi}) \|_{L^2(\Sigma_{\tau}^1)} + \| \nabla_a(\Omega^{-1}) (\tilde{\phi}^{(n)} - \tilde{\phi})\|_{L^2(\Sigma_{\tau}^1)} \\
    & \la \| (\tilde{\D}_a \tilde{\phi})^{(n)} - \tilde{\D}_a \tilde{\phi} \|_{L^\infty_t L^2_x} + \| \tilde{\phi}^{(n)} - \tilde{\phi} \|_{L^\infty_t L^2_x} \\
    &\to 0.
\end{split}
\end{align}
The same argument (but instead pulling the solution back to $\mathbb{M}'$) applies on $\Sigma_\tau^2$, and hence we have in particular $\mathscr{E}^{(n)}(\tau) \to \mathscr{E}(\tau)$. Moreover, the final estimate also shows that $\D_a \phi \in L^\infty([0,\tau_*];L^2(\mathbb{S}^3))$ is a limit in $L^\infty([0,\tau_*];L^2(\mathbb{S}^3))$ of smooth functions $(\D_a \phi)^{(n)}$, so in particular is in $\mathcal{C}^0([0,\tau_*];L^2(\mathbb{S}^3))$.
\end{proof}

\begin{remark}
    \label{rmk:limit_of_smooth_solutions_cylinder}
    The estimates \eqref{limit_in_L2_Maxwell_field}, \eqref{limit_in_L2_scalar_field} and \eqref{limit_in_L2_gauge_covariant_derivative} show that the solution obtained in \Cref{thm:local_strip_solution_cylinder} is a limit of smooth solutions in the energy norm given by \eqref{cylinder_energy}.
\end{remark}

\subsection{Global solution}
\label{sec:global_solution}

To extend the solution globally, we iterate the solution obtained in \Cref{thm:local_strip_solution_cylinder} with the same time of existence $\tau_*$; this is justified by the conservation of energy. Applying \Cref{thm:local_strip_solution_cylinder} from $\tau = \tau_*$, the solution on the next strip will satisfy
\[ (A_0, \partial_\tau A_0)|_{\tau = \tau_*} = (0,0), \]
so it only remains to check that the gauge transformation which sends $(A_0, \partial_\tau A_0)|_{\tau = \tau_*} \rightsquigarrow (0,0)$ for the solution in \Cref{thm:local_strip_solution_cylinder} preserves regularity. Once this is known, we then construct the solution on the time interval $[0, \tau_*]$ and change residual gauge to $(A_0, \partial_\tau A_0)|_{\tau=\tau_*} = (0,0)$, and construct the solution on $[\tau_*, 2\tau_*]$. We then reverse this gauge transformation to return to the initial gauge condition $(a_0, \dot{a}_0) = (0,0)$, preserving the continuity in time of the solution at $\tau=\tau_*$. The required gauge transformation is $A_a \rightsquigarrow A_a + \nabla_a \chi$, where $\chi$ satisfies
\begin{align*}
    &\Box \chi = 0, \\
    &(\chi, \partial_\tau \chi)|_{\tau=\tau_*} = (-\boldsymbol{\Delta}^{-1}(\partial_\tau A_0|_{\tau_*}), - A_0|_{\tau_*}) \in H^{\frac{3}{2}-\delta}(\mathbb{S}^3) \times H^{\frac{1}{2}-\delta}(\mathbb{S}^3).
\end{align*}
By the energy inequality \eqref{energy_inequality},
\begin{equation} \label{gauge_transformation_regularity_extension_to_global_solution} \chi \in \mathcal{C}^0([0,\tau_*];H^{\frac{3}{2}-\delta}(\mathbb{S}^3)) \cap \mathcal{C}^1([0,\tau_*];H^{\frac{1}{2}-\delta}(\mathbb{S}^3)).
\end{equation}
This clearly preserves the regularity of the potential, i.e.
\[ A_a \in \mathcal{C}^0([0,\tau_*];H^{\frac{1}{2}-\delta}(\mathbb{S}^3)) \cap \mathcal{C}^1([0,\tau_*];H^{-\frac{1}{2}-\delta}(\mathbb{S}^3)) \]
in the new gauge. Since $F_{ab}$ is gauge-invariant, it immediately inherits the regularity $F_{ab} \in \mathcal{C}^0([0,\tau_*];L^2(\mathbb{S}^3))$. Notably, however, here $\chi(\tau,\cdot)$ is not regular enough to be bounded on $\mathbb{S}^3$. By \Cref{lem:composition_Sobolev_functions_2} (using a partition of unity on $\mathbb{S}^3$ and applying the lemma to $x \mapsto \cos x - 1$ and $x \mapsto \sin x$), we have
\[ \e^{-i\chi} \in \mathcal{C}^0([0,\tau_*]; H^{\varrho_1}(\mathbb{S}^3)), \]
where 
\[ \varrho_1 = \frac{3}{2(1+\delta)}. \]
Using \Cref{lem:product_Sobolev_estimates},
\[ \| \e^{-i\chi} \phi \|_{H^{1-4\delta}(\mathbb{S}^3)} \la \| \e^{-i\chi} \|_{H^{\varrho_1}(\mathbb{S}^3)} \| \phi \|_{H^{1-\delta}(\mathbb{S}^3)}, \]
and similarly
\begin{align*} \| \partial_\tau ( \e^{-i\chi} \phi ) \|_{H^{-4\delta}(\mathbb{S}^3)} &\la \| \e^{-i\chi} \partial_\tau \phi \|_{H^{-3\delta}(\mathbb{S}^3)} + \| \e^{-i\chi} \phi \partial_\tau \chi \|_{H^{-4\delta}(\mathbb{S}^3)} \\ 
& \la \| \e^{-i\chi} \|_{H^{\varrho_1}(\mathbb{S}^3)} \| \partial_\tau \phi \|_{H^{-\delta}(\mathbb{S}^3)} + \| \phi \|_{H^{1-\delta}(\mathbb{S}^3)} \|\e^{-i\chi} \partial_\tau \chi \|_{H^{\frac{1}{2}-3\delta}(\mathbb{S}^3)} \\
& \la \| \e^{-i\chi} \|_{H^{\varrho_1}(\mathbb{S}^3)} \| \partial_\tau \phi \|_{H^{-\delta}(\mathbb{S}^3)} +\| \phi \|_{H^{1-\delta}(\mathbb{S}^3)} \| \e^{-i\chi} \|_{H^{\varrho_1}(\mathbb{S}^3)} \|\partial_\tau \chi \|_{H^{\frac{1}{2}-\delta}(\mathbb{S}^3)}.
\end{align*}
Hence in the new gauge $ \phi \in \mathcal{C}^0([0,\tau_*];H^{1-4\delta}(\mathbb{S}^3)) \cap \mathcal{C}^1([0,\tau_*];H^{-4\delta}(\mathbb{S}^3))$. It remains to treat the gauge covariant derivative $\D_a \phi$. This transforms as $\D_a \phi \rightsquigarrow \e^{-i\chi} \D_a \phi$, so, given the regularity \eqref{gauge_transformation_regularity_extension_to_global_solution}, it may appear like some regularity in $\D_a \phi$ may be lost. This is not the case, however, since the $L^2$ norm of $\D_a \phi$ is gauge-invariant:
\[ \| \e^{-i\chi} \D_a \phi \|_{L^2(\mathbb{S}^3)} = \| \D_a \phi \|_{L^2(\mathbb{S}^3)}, \]
so $\D_a \phi \in \mathcal{C}^0([0,\tau_*];L^2(\mathbb{S}^3))$ in the new gauge.

By applying the above gauge transformation to the solution on $[0,\tau_*]$, then solving from $\tau = \tau_*$ using \Cref{thm:local_strip_solution_cylinder}, and subsequently reversing the above gauge transformation on $[0,2\tau_*]$, we obtain a solution on $[0,2\tau_*]$ which satisfies $(a_0, \dot{a}_0) = (0,0)$,
\[ A_a \in \mathcal{C}^0([0,2\tau_*];H^{\frac{1}{2}-\delta}(\mathbb{S}^3)) \cap \mathcal{C}^1([0,2\tau_*];H^{-\frac{1}{2}-\delta}(\mathbb{S}^3)), \]
\[ F_{ab}, ~ \D_a \phi \in \mathcal{C}^0([0,2\tau_*];L^2(\mathbb{S}^3)), \]
and
\[ \phi \in \mathcal{C}^0([0,2\tau_*];H^{1-16\delta}(\mathbb{S}^3)) \cap \mathcal{C}^1([0,2\tau_*];H^{-16\delta}(\mathbb{S}^3)).\]
Iterating this procedure, the scalar field on $[0,n\tau_*]$ has spatial regularity $H^{1-4^n \delta}(\mathbb{S}^3)$. Pick a (sufficiently small) $\delta > 0$. Then, since the $\delta > 0$ in the above procedure is arbitrarily small (and, in particular, does not depend on $\tau_*$), we can construct a sequence of solutions indexed by $n$, with time of existence $[0,n\tau_*]$, such that for each $n \in \mathbb{N}$ the scalar field has spatial regularity $H^{1-2^{-n}\delta}(\mathbb{S}^3)$ (by choosing $\delta_n = 2^{-3n} \delta$). Since $\bigcup_{n=1}^{\infty} H^{1-2^{-n}\delta}(\mathbb{S}^3) \subset H^{1-\delta}(\mathbb{S}^3) $, by the uniqueness part of \Cref{thm:local_strip_solution_cylinder}, these are all the same solution in the regularity class given by the chosen $\delta > 0$, with time of existence $[0,n\tau_*]$. Induction on $n$ therefore gives a global solution in the regularity class given by $\delta > 0$.

This completes the proof of \Cref{thm:main_theorem}. \hspace{\fill} \qed

\begin{remark}
    We remark that the gauge transformation \eqref{gauge_transformation_regularity_extension_to_global_solution} does not preserve the local spacetime regularity $H^{1-\delta}_{\text{loc}}(\mathbb{R} \times \mathbb{S}^3)$ of the potential, only the weaker regularity $A_a \in H^{1/2-\delta}_{\text{loc}}(\mathbb{R} \times \mathbb{S}^3)$. For $\phi$ we can obtain $\phi \in H^{1-\delta}_{\text{loc}}(\mathbb{R} \times \mathbb{S}^3)$ by a localization, partition of unity, and Fourier argument as we have done at several points throughout the paper.
\end{remark}

\subsection*{Acknowledgements}

We thank the ICMS Edinburgh and the James Clerk Maxwell Foundation for their hospitality in the Summer of 2023, where a part of this paper was written, and the support of the Institut Henri Poincaré (UAR 839 CNRS-Sorbonne Université), and LabEx CARMIN (ANR-10-LABX-59-01). We also thank Lionel Mason, Claude Warnick, Sigmund Selberg and Kenji Nakanishi for useful discussions.

\section{Appendix}
\label{sec:appendix}

The following proofs follow a strategy similar to that used in \cite{WarnickReall2022}.

\subsection{Proof of \Cref{lem:foliation_change_positive_s}}
    The statement follows directly from the energy inequality for the linear wave equation \eqref{energy_inequality} with $\sigma = s$ and the fact that for $s \geq 0$
\[ H^s_{\text{loc}}(\mathcalboondox{M}) \hookrightarrow L^1_{\text{loc}}(\mathbb{R}_\tau; H^s(S_\tau)). \]

\subsection{Proof of \Cref{lem:foliation_change_negative_s}}
\label{sec:proof_of_foliation_change_negative_s}

We prove this using a duality argument. For $s \in [-1, 0)$, let 
\[ \eta \in L^1_{\text{loc}}(\mathbb{R}_\tau;H^{-s-1}(S_\tau)) \]
and consider the dual Cauchy problem
\[
\begin{cases}
    \Box_g \chi = \eta, \\
    (\chi, \partial_\tau \chi)|_{\tau = T} =(0,0)
\end{cases}
\]
for any given $T>0$. By the energy inequality, $\chi \in \mathcal{C}^0(\mathbb{R}_\tau; H^{-s}(S_\tau)) \cap \mathcal{C}^1(\mathbb{R}_\tau; H^{-s-1}(S_\tau))$. This in turn implies $\chi \in H^{-s}_{\text{loc}}(\mathcalboondox{M})$. Indeed, this property can be shown on $\mathbb{R}^{1+3}$ in Fourier space by observing that
\[ (1+ \tau^2 + |\xi|^2)^{|s|} \leq ( \langle \tau \rangle^2 + \langle \xi \rangle^2 )^{|s|} \leq \langle \xi \rangle^{2|s|} \left( 1+ \frac{\langle \tau \rangle^2}{\langle \xi \rangle^2} \right) = \langle \xi \rangle^{2|s|} + \langle \tau \rangle^2 \langle \xi \rangle^{2|s|-2}, \]
which proves the embedding
\[ L^2(\mathbb{R}; H^{|s|}(\mathbb{R}^3)) \cap H^1(\mathbb{R};H^{|s|-1}(\mathbb{R}^3)) \hookrightarrow H^{|s|}(\mathbb{R}^4). \]
Using cut-offs in time, we infer
\[ \mathcal{C}^0(\mathbb{R};H^{|s|}(\mathbb{R}^3)) \cap \mathcal{C}^1(\mathbb{R};H^{|s|-1}(\mathbb{R}^3)) \hookrightarrow H^{|s|}_{\text{loc}}(\mathbb{R}^4). \]
The embedding
\begin{equation} \label{HsEmbedd}
\mathcal{C}^0 (\R_\tau ; H^{|s|} (S_\tau)) \cap \mathcal{C}^1 (\R_\tau ; H^{|s|-1} (S_\tau)) \hookrightarrow H^{|s|}_\mathrm{loc} (\mathcalboondox{M} )
\end{equation}
can be obtained using local charts on $S_\tau \simeq S$ and a partition of unity. In what follows let us write $\langle \cdot, \cdot \rangle$ to denote the appropriate duality pairing; by $\langle u, \eta \rangle$, for example, we mean the duality pairing between $L^\infty([0,T]_\tau;H^{s+1}(S_\tau))$ and $L^1([0,T]_\tau;H^{-s-1}(S_\tau))$. We then have
\begin{align*}
    \langle u, \eta \rangle & = \langle u, \Box \chi \rangle\\
    & = \langle f , \chi \rangle + \langle u, \partial_\tau \chi \rangle|_{S_0} - \langle \partial_\tau u , \chi\rangle|_{S_0},
\end{align*}
so 
\begin{align*}
    | \langle u, \eta \rangle | & \la \| f \|_{H^{s}([0,T]_\tau \times S)} \| \chi \|_{H^{-s}([0,T]_\tau \times S)}  + \| u_0 \|_{H^{s+1}(S_0)} \| \partial_\tau \chi \|_{H^{-s-1}(S_0)}  + \| u_1 \|_{H^{s}(S_0)} \| \chi \|_{H^{-s}(S_0)}.
\end{align*}
Since $\chi \in \mathcal{C}^0(\mathbb{R}_\tau;H^{-s}(S_\tau)) \cap \mathcal{C}^1(\mathbb{R}_\tau;H^{-s-1}(S_\tau))$, we have, by \eqref{HsEmbedd} and the energy inequality \eqref{energy_inequality}, that
\begin{align*} \| \chi \|_{H^{-s} ([0,T]_\tau \times S)} &\la \| \chi \|_{L^\infty([0,T]_\tau;H^{-s}(S_\tau))} + \| \partial_\tau \chi \|_{L^\infty([0,T]_\tau;H^{-s-1} (S_\tau))} \\
& \la \| \eta \|_{L^1([0,T]_\tau;H^{-s-1}(S_\tau))}, 
\end{align*}
and so, again using the energy inequality for $\chi$,
\[ | \langle u , \eta \rangle | \la \left( \| u_0 \|_{H^{s+1}(S_0)} + \| u_1 \|_{H^{s}(S_0)} + \| f \|_{H^{s}([0,T]_\tau \times S)} \right) \| \eta \|_{L^1([0,T]_\tau;H^{-s-1}(S_\tau))}. \]
This implies that
\[ \phi \in L^\infty([0,T]_\tau ; H^{s+1}(S_\tau)) \]
and
\begin{equation}
\label{foliation_change_negative_s_estimate}
     \| u \|_{L^\infty([0,T]_\tau ; H^{s+1}(S_\tau))} \la \| u_0 \|_{H^{s+1}(S_0)} + \| u_1 \|_{H^{s}(S_0)} + \| f \|_{H^{s}([0,T]_\tau \times S)} .
\end{equation}
Continuity in $\tau$ then follows easily by approximating $u$ by smooth functions and noting the above bound for its $L^\infty([0,T]_\tau;H^{s+1}(S_\tau))$ norm. 
\hspace{\fill}$\qed$

\printbibliography

\end{document}